\documentclass[11pt, letterpaper]{amsart}
\usepackage[utf8]{inputenc}
\usepackage{amsmath}
\usepackage{amsfonts}
\usepackage{amssymb}
\usepackage{graphicx}
\usepackage{slashed}
\usepackage[english]{babel}
\usepackage{amsthm}
\usepackage{mathtools}
\usepackage[margin=0.6in]{geometry}
\usepackage{hyperref}
\usepackage{bbm}
\usepackage{accents}
\usepackage{units}

\makeatletter
\newcounter{savesection}
\newcounter{apdxsection}
\renewcommand\appendix{\par
  \setcounter{savesection}{\value{section}}%
  \setcounter{section}{\value{apdxsection}}%
  \setcounter{subsection}{0}%
  \gdef\thesection{\@Alph\c@section}}
\newcommand\unappendix{\par
  \setcounter{apdxsection}{\value{section}}%
  \setcounter{section}{\value{savesection}}%
  \setcounter{subsection}{0}%
  \gdef\thesection{\@arabic\c@section}}
\makeatother

\usepackage{color}
\usepackage{amsmath,amsthm,amssymb}
\usepackage{graphicx}
\usepackage{color}

\title{Asymptotics and Scattering for wave Klein-Gordon system{s}}
\author{Xuantao Chen}
\address[X.C.]{Johns Hopkins University, Department of Mathematics, 3400 N.\@ Charles St., Baltimore, MD 21218, USA}
\email{xchen165@jhu.edu}

\author{Hans Lindblad}
\address[H.L.]{Johns Hopkins University, Department of Mathematics, 3400 N.\@ Charles St., Baltimore, MD 21218, USA}
\email{lindblad@math.jhu.edu}

\numberwithin{equation}{section}

\newcommand{\beq}{\begin{equation}}\newcommand{\eq}{\end{equation}}
\newcommand{\beqs}{\begin{equation*}}\newcommand{\eqs}{\end{equation*}}
\def\pa{\partial}

\theoremstyle{plain}
\newtheorem{prop}{Proposition}[section]
\newtheorem{lemma}[prop]{Lemma}
\newtheorem{remark}{Remark}[section]
\newtheorem{cor}[prop]{Corollary}

\newtheorem{theorem}{Theorem}[section]

\allowdisplaybreaks[1]

\def\beaa{\begin{eqnarray*}}
\def\eeaa{\end{eqnarray*}}

\def\ab{{\alpha\beta}}
\def\rr{{\rho\rho}}

\begin{document}


\maketitle

\begin{abstract}
We study the coupled wave-Klein-Gordon systems, introduced in \cite{LMmodel} and \cite{IPmodel}, to model the nonlinear effects from the Einstein-Klein-Gordon equation in harmonic coordinates. We first go over a slightly simplified version of global existence based on \cite{LMmodel}, and then derive the asymptotic behavior of the system. The asymptotics of the Klein-Gordon field consist of a modified phase
times a homogeneous function, and the asymptotics of the wave equation consist of a radiation field in the wave zone and an interior homogeneous solution coupled to the Klein-Gordon asymptotics. We then consider the inverse problem, the scattering from infinity. We show that given the type of asymptotic behavior at infinity, there exist solutions of the system that present the exact same behavior.
\end{abstract}

\section{Introduction}
\subsection{The system{s}}We study the coupled nonlinear wave-Klein-Gordon system
\begin{equation}\label{semilinearsystem}
\begin{split}
        &-\Box u=(\pa_t \phi)^2+\phi^2,\\
        &-\Box \phi+\phi=u\phi,
\end{split}
\end{equation}
and
\begin{equation}\label{quasilinearsystem}
\begin{split}
    &-\Box u=(\pa_t\phi)^2+\phi^2,\\
    &-\Box \phi+\phi=H^{\alpha \beta} u\, \pa_\alpha\pa_\beta \phi
\end{split}
\end{equation}
in $\mathbb{R}_+\times \mathbb{R}^3=\{(t,x)\colon t>0,x\in\mathbb{R}^3\}$, where {$H^{\alpha\beta}$ is a symmetric tensor defined so that $H^{ab}$ are constants for $a,b\in\{0,1,2,3\}$ (in our coordinate system)}. The wave operator $\Box=-\pa_t^2+\sum_{i=1}^3 \pa_i^2$ is the Laplace-Beltrami operator on the Minkowski spacetime. 

The latter system was studied as a model system for the Einstein-Klein-Gordon (self-gravitating massive field) system in wave coordinates:
\begin{equation}
    \begin{split}
        g^{\alpha\beta}\pa_\alpha\pa_\beta g_{\mu\nu}=F_{\mu\nu}(g)(\pa g,\pa& g)+\pa_\mu \phi\, \pa_\nu \phi+\frac 12 g_{\mu\nu} \phi^2,\\
        g^{\alpha\beta}\pa_\alpha \pa_\beta \phi-\phi&=0.
    \end{split}
\end{equation}
The nonlinear term $F_{\mu\nu}(g)(\pa g,\pa g)$ is semilinear and {possesses the weak null structure}, a concept introduced in \cite{LRweaknull} to study the Einstein vacuum equation in wave coordinates. The model focuses on the remaining nonlinearities, especially the interaction between the metric and the massive field. It turns out that this model indeed describes the main behavior of the interaction, and the proof of small data global existence of this model leads to the proof of the stability of the Minkowski solution of the Einstein-Klein-Gordon system by LeFloch-Ma \cite{LMEKG} and Ionescu-Pausader \cite{IPEKG} using different methods.

In this paper, we study the asymptotic behavior of the systems. After deriving the asymptotic behavior, we then study the backward problem (scattering from infinity), i.e., given such an asymptotic expansion at infinity, we consider if there exists a solution presenting the exact behavior at the infinity. For completeness, in the asymptotics part, we first give a slightly simplified version of small data global existence for \eqref{semilinearsystem} based on the proof in \cite{LMmodel}, and then show the following main theorem. 

We state the theorem for compactly supported initial data, but we shall remark later that a similar result holds for the non-compact case.

\begin{theorem}[Asymptotic behavior]\label{thmasymptotics}
Consider the system \eqref{semilinearsystem} or \eqref{quasilinearsystem} with compactly supported initial data. Without loss of generality, we impose the initial data at $\, t=2$, and assume the initial data is supported in $\{|x|\leq 1\}$.
Then, for system \eqref{semilinearsystem}, with some smallness assumption (say, of size $\varepsilon$) on the initial data, we have
\begin{equation}\label{eq:homogeneous}
    u=\frac{U(y)}\rho+O(\rho^{-2+\delta} (1-|y|^2)^{\nicefrac\delta 2})=\frac{U(y)}\rho+O(\varepsilon t^{-1}(t-r)^{-1+\delta}), \quad t-r>4
\end{equation}
for some homogeneous function $U(y)$ and small $\delta>0$, where $r=|x|$, $y=x/t$, $\rho=\sqrt{t^2-|x|^2}$, and
\begin{equation}
    \phi=\rho^{-\frac 32}(e^{i\rho-\frac i2 U(y)\ln \rho} a_+(y)+e^{-i\rho+\frac i2 U(y)\ln \rho} a_-(y))+O(\varepsilon t^{-\nicefrac 52+\delta}),\quad t-r>4
\end{equation}
where the functions $a_\pm (y)$ decay fast when $|y|\rightarrow 1$. In the region $\{t-r\leq 4\}$, we have $|\phi|\leq C\varepsilon t^{-\nicefrac 52+\delta}$. 
The radiation field of $u$, defined as $F(q,\omega):=\lim_{r\rightarrow\infty} r u(r-q,r\omega)$ where $\omega$ is the angular variable, exists and satisfies
\begin{equation}
    \left|u-\frac{F(r-t,\omega)}r\right|\leq C\varepsilon (1+t+r)^{-2}(t-r).
\end{equation}
We also have the expansion
\begin{equation}\label{eq:radiation}
    F(r-t,\omega)=A(\omega)+O((t-r)_+^{-1+\delta}).
\end{equation}
for some function $A(\omega)$. Similar results hold for system \eqref{quasilinearsystem} with the modified phase function $\rho-\frac 12 U(y)\ln\rho$ for the Klein-Gordon field replaced by $\rho^*$, a function that only depends on $U(y)$ and the constants $H^{ab}$. Along each hyperboloidal ray (i.e. with $y$ fixed), $\rho^*$ differs from $\rho$ by $O(\varepsilon\ln \rho)$.
\end{theorem}

The homogeneous function $U(y)$ is related to $a_\pm (y)$ through the equation 
\begin{equation}
    -\Box(\frac{U(y)}\rho)=2\rho^{-3} (1+(1-|y|^2)^{-1}) a_+(y) a_-(y).
\end{equation}
We have the following properties of $U(y)$.
\begin{lemma}
The function $U(y)$ is uniquely determined by $a_\pm (y)$, and the limit $\lim_{|y|\rightarrow 1}U(|y|\omega)(1-|y|^2)^{-1/2}$ exists, which is exactly the function $A(\omega)$ above. Moreover, this holds with vector fields, i.e. $\lim_{|y|\rightarrow 1} \Omega^I U(y)(1-|y|^2)^{-1/2}$ exists, where $\Omega$ is boost or rotation vector fields, as long as $|I|$ does not exceed the regularity from the initial data.
\end{lemma}

We then study the scattering from infinity problem. We want to find solutions with the same type of prescribed asymptotic behavior. By a set of scattering data, we mean $(a_\pm (y), F(r-t,\omega))$, where $a_\pm (y)$ correspond to the homogeneous functions in the asymptotic expansion of the Klein-Gordon field, and $F(r-t,\omega)$ is the radiation field of the wave field. 

In view of the asymptotics result, we need to require that the radiation field satisfies the following expansion:
\begin{equation}
    F(q,\omega)=A(\omega)+O((t-r)^{-1+\alpha})\quad \text{as }q\rightarrow -\infty,\quad F(q,\omega)=O((r-t)^{-1+\alpha}) \quad\text{as }q\rightarrow +\infty
\end{equation}
where $A(\omega)$ is determined by $a_\pm (y)$ as in the remark above. Here we allow more generality by replacing the $\delta$ in the forward problem with some $\alpha<1/6$.

We have the following theorem:
\begin{theorem}[Scattering from infinity]\label{thmscattering}
Consider a set of scattering data $(a_\pm (y), F({r-t},\omega))$ where $a_\pm (y)$ decay well as $|y|\rightarrow 1$, say $|\nabla^k a_\pm(y)|\leq C\varepsilon (1-|y|^2)^l$ for some $l\geq N_1$ and $k\leq N_1$, and $F(q,\omega)$ satisfies
\begin{equation}\label{decayF}
    |(q\pa_q)^k \pa_\omega^\beta (F(q,\omega)-A(\omega))|\leq C\varepsilon{\langle q\rangle^{-1+\alpha}},\quad q\leq 0,\quad|(q\pa_q)^k \pa_\omega^\beta F(q,\omega)|\leq C\varepsilon{\langle q\rangle^{-1+\alpha}},\quad q>0,\quad k+|\beta|\leq N_1
\end{equation}
for some positive integer $N_1\geq 8$ and $\alpha\in (0,\frac 16)$. Note that $a_\pm (y)$ determine $U(y)$ and $A(\omega)$ from above. Then there exists a solution $(u,\phi)$ of \eqref{semilinearsystem} with the property that
\begin{equation}
    u=\frac{U(y)}\rho+O(\varepsilon\rho^{-2+\alpha}(1-|y|^2)^{\nicefrac\alpha 2})=\frac{U(y)}\rho+O(\varepsilon t^{-1} (t-r)^{-1+\alpha}),\quad t-r>4,
\end{equation}
\begin{equation}
    \phi=\rho^{-\frac 32}(e^{i\rho-\frac i2 U(y)\ln \rho} a_+(y)+e^{-i\rho+\frac i2 U(y)\ln \rho} a_-(y))+O(\varepsilon t^{-\nicefrac 52}),\quad t>r,
\end{equation}
and $F(q,\omega)$ is the radiation field of $u$.
\end{theorem}
{We remark here that a similar result can be obtained for the system \eqref{quasilinearsystem} by the same method.}
\vspace{1ex}

The radiation field $F$ for a linear wave equation goes back to Friedlander \cite{F62};
see also \cite{H97}. Similar asymptotics for nonlinear wave equations with modified behavior at infinity, in particular for Einstein's equations {was done} in Lindblad \cite{L17}. {Also there the asymptotics for the wave equation consists of a radiation field and an interior homogeneous function, and moreover the asymptotic behavior is also logarithmically modified.} Asymptotics for linear Klein-Gordon can be found in \cite{H97} and for nonlinear Klein-Gordon in one space dimension
in Delort \cite{D1,D2} see also \cite{L-S2}. Asymptotics for Klein-Gordon with variable coefficient nonlinearities was first studied in Lindblad-Soffer \cite{LS15} and
Sterbenz \cite{sterbenz2016dispersive}, with further results in \cite{LLS20,LLS21,LLSS22,LuhrmannSchlag21}. {Scattering from infinity for Klein-Gordon in one space dimension was first done in Lindblad-Soffer \cite{LS05a}.} Scattering from infinity for semilinear wave equations with modified behavior at infinity was first done in Lindblad-Schlue \cite{lindblad2017scattering}, and for quasilinear wave equations by \cite{Y21a,Y21b}. Peeling estimates for Maxwell-Klein-Gordon {with vanishing mass} were done in \cite{LS06b} using a fractional Morawetz, and this was used to derive asymptotics for the Maxwell-Klein-Gordon system
in \cite{CKL17} and scattering in \cite{H21}. Asymptotics for nonlinear wave equations was also studied in \cite{BB15,BVW15,W14}. {The scattering from infinity problem on the Fourier side for the wave-Klein-Gordon system was studied by Ouyang \cite{Ouyang}, which follows the space-time resonance method used in \cite{IPmodel,IPEKG}.
To our knowledge, our paper is the first result on either asymptotics or scattering on the physical side for a combined nonlinear wave-Klein-Gordon system with the asymptotics for wave and Klein-Gordon fields affecting each other, which provides an explicit construction of approximate solutions. }

\subsection{Heuristics for Asymptotics}
For completeness, we provide a short proof of global existence in the semilinear case based on the proof in \cite{LMmodel}. With the presence of the massive field, one cannot use the scaling vector field for commutation. It turns out to be natural to consider the hyperboloidal foliation, first used in \cite{K85}. The conformal energy estimate derived in \cite{ma2017conformal} is used to deal with $L^2$ estimate of the wave component. The use of conformal estimates improves the half-order growth of the higher-order energy of the Klein-Gordon field in the original work \cite{LMmodel}, as is also recently pointed out in \cite{LMnew}. The proof of the asymptotics requires the decay estimates obtained from the global existence result.

We now discuss the asymptotics. Recall that by \cite{H97}, the solution of the linear Klein-Gordon equation
\begin{equation}
(-\Box+1)\phi=0
\end{equation}
in $3$ space dimensions has an asymptotic expansion of the form
\begin{equation}\label{linearKGexpansion}
\phi\sim \rho^{-\frac 32}(e^{i\rho}a_+(y)+e^{-i\rho}a_-(y)),\quad y=x/t,\quad t>|x|,
\end{equation}
with $a_\pm (y)$ decay fast when $|y|\rightarrow 1$, and the solution decays sufficiently fast when in the exterior $t>|x|$.


\subsubsection{{The semilinear model} }
Before studying \eqref{semilinearsystem}, we first consider a simple coupled system:
\begin{equation}
(-\Box+1)\phi=0,\qquad -\Box u=(\pa_t \phi)^2 +\phi^2.
\end{equation}
Note that the derivative $\pa_t \phi$ presents the same type of asymptotics as $\phi$, so for simplicity we only consider the source term $\phi^2$ here. In view of the asymptotics for the Klein-Gordon equation above, the wave equation can be modeled by
\begin{equation}
-\Box u= 2\rho^{-3}a_+(y) a_-(y)+\rho^{-3}(e^{2i\rho}a_+^2(y)+e^{-2i\rho} a_-^2(y)).
\end{equation}
We expect terms with oscillating factors $e^{\pm 2i\rho}$ provide extra cancellation, so the leading behavior is given by
\begin{equation}
    -\Box u_1=2\rho^{-3} a_+(y) a_-(y).
\end{equation}

As in \cite{K85}, in polar hyperbolic coordinates $\rho$ and $y=x/t$ the wave operator can be written as
\begin{equation}
-\Box= \pa_\rho^2 +3\rho^{-1}\pa_\rho -\rho^{-2}\triangle_y,
\end{equation}
where $\triangle_y$ is the Laplace-Beltrami operator on the unit hyperboloid $t^2\!-|x|^2=1$ with respect to the natural metric $|dx|^2\!-dt^2$. Making the ansatz
\begin{equation}
u_1=U(y)/\rho,
\end{equation}
we get
\begin{equation}
-\Box u_1= \rho^{-3}(\triangle_y+1)U(y)=2\rho^{-3} a_+(y) a_-(y).
\end{equation}
This suggests that $u$ indeed behaves like $U(y)/\rho$ if we can solve
\begin{equation}\label{ellipticequationfromwave}
(\triangle_y+1)U(y)=2a_+(y) a_-(y).
\end{equation}
However, this is not an ideal problem to solve directly, both because of the equation itself and that the function $U(y)/\rho$ is not good in terms of the differentiability near the light cone $\{t=r\}$, which comes from the singularity of the source $2\rho^{-3} a_+(y) a_-(y)$ at the origin. As such, we instead solve the wave equation starting at $\{t=2\}$ with vanishing initial data. Let $P(y)=2t^3 \rho^{-3} a_+(y) a_-(y)=2(1-|y|^2)^{-3/2} a_+(y) a_-(y)$ so that $-\Box u_1=t^{-3} P(y)$. Using the representation formula and a change of variable used in \cite{LMmodel}, we have
\begin{equation}
    u_1(t,x)
    =\frac 1{4\pi t}\int_{{2/t}}^1 \int_{\mathbb{S}^2} (1-\lambda) \lambda^{-3} P\left (\frac{{x/t}-(1-\lambda)\eta}{\lambda}\right )\, d\sigma(\eta) d\lambda.
\end{equation}

Recall that $P(y)$ is zero when $|y|>1$; this implies that the integrand is zero when $\lambda<\frac 12(1-{r/t})$. Therefore, when ${2/t}<\frac 12(1-{r/t})$, i.e., $t-r>4$, we have the integrand being zero for $\lambda<2/t$, so
\begin{equation}
    u_1(t,x)
    =\frac 1{4\pi t}\int_0^1 \int_{\mathbb{S}^2} (1-\lambda) \lambda^{-3} P\left (\frac{{x/t}-(1-\lambda)\eta}{\lambda}\right )\, d\sigma(\eta) d\lambda
\end{equation}
which is now of the form $U(x/t)/\rho$, where
\begin{equation}\label{expressionforU}
    U(y)=\frac {(1-|y|^2)^{1/2}}{4\pi }\int_0^1 \int_{\mathbb{S}^2} (1-\lambda) \lambda^{-3} P\left (\frac{y-(1-\lambda)\eta}{\lambda}\right )\, d\sigma(\eta) d\lambda
\end{equation}
Hence we have $u_1=U(y)/\rho$ in the region $\{t-r>4\}$. Compared with $U(y)/\rho$, $u_1$ is much smoother near the light cone. It is not hard to show the existence of the radiation field, which is the main part of $u_1$ when $t-r\leq 4$.


We still need to deal with the oscillating part. We need to estimate $u_2^{\pm}$ where \begin{equation}
    -\Box u_2^{\pm}=\rho^{-3} e^{\pm 2i\rho} a_\pm^2(y)
\end{equation} with vanishing initial data. We will use an integration by part argument to show that we can get one more power of decay in $t$, at the expense of an unfavorable factor near the light cone $(1-{r/t})^{-1}$. As a result, we can show the estimate
\begin{equation}
    |u_2^\pm|\lesssim t^{-1} (1+(t-r)_+)^{-1}.
\end{equation}
This corresponds to $\rho^{-2}$ decay along hyperboloidal rays, which is one order better than $u_1$.

Now we turn to our model \eqref{semilinearsystem}. The Klein-Gordon field does not satisfy the expansion \eqref{linearKGexpansion} anymore. To get modified behavior, we first write the equation in hyperboloidal coordinates $(\rho,y)$. One has
\begin{equation}
    \pa_\rho^2 \Phi+\Phi=u\Phi+h,\quad \Phi=\rho^\frac 32 \phi
\end{equation}
where $h$ consists of terms that {decay faster}, which can be seen from the decay estimates obtained in the global existence proof. Solving this equation gives a phase correction compared with the linear solution:
\begin{equation}
    \phi\sim \rho^{-\frac 32}(e^{i\rho-\frac i2\int u d\rho} a_+(y)+e^{-i\rho+\frac i2\int ud\rho}a_-(y)).
\end{equation}
We note that $a_\pm (y)$ may be different from above. We need to determine, in this case, the behavior of the wave component $u$. Notably, the leading (i.e. non-oscillating) contribution from the expansion above is again $2\rho^{-3} a_+(y) a_-(y)$ where the phase correction from $u$ is not present. Therefore, for this part we can use the same estimate. The oscillating parts are now with the correction factor, but the integration by part argument still {applies}.

This means that we can decompose $u$ as a leading part in the interior, $U(y)/\rho$, and a remainder which is bounded by $\rho^{-2+\delta}$ along hyperboloidal rays. (Note that, however, the remainder is not ignorable towards null infinity.) Therefore, we can in fact prove the following asymptotics of the Klein-Gordon field:
\begin{equation}
    \phi\sim \rho^{-\frac 32}(e^{i\rho-\frac i2 U(y)\ln \rho} a_+(y)+e^{-i\rho+\frac i2 U(y)\ln \rho} a_-(y)).
\end{equation}

In the case of compactly supported initial data, we can let the solution be supported in the region $\{t-r\geq 1\}$. We note that in this case, when integrating along the hyperboloidal ray, we have the integration start at the boundary $t-r=1$, i.e. $\rho=(1-|y|^2)^{-\frac 12}$. This provides improved control of the Klein-Gordon field as $|y|\rightarrow 1$, i.e. higher power of $(1-|y|^2)$. In fact, one can apply this integration iteratively to get better and better decay in $(1-|y|^2)$, at the expense of losing two order derivatives (recall we need the control of the hyperboloidal Laplacian).

\subsubsection{{The quasilinear model} }
We now point out its relation with the quasilinear model, {\eqref{quasilinearsystem}}.
For the quasilinear model, we can decompose the quasilinear part in $(\rho,y)$ coordinates
\begin{equation}
    H^{\ab} u \, \pa_\alpha \pa_\beta \phi=H^\rr u \, \pa_\rho^2\phi+2H^{\rho y_i}u \, \pa_\rho \pa_{y_i}\phi+H^{y_i y_j} u \, \pa_{y_i} \pa_{y_j} \phi+R_1
\end{equation}
where $R_1$ is given by transition matrices and can be shown to behave well. Combining this with our decomposition of linear Klein-Gordon equation, we get
\begin{equation}
    (1-H^\rr u)\pa_\rho^2 \Phi+\Phi=h,\quad \Phi=\rho^\frac 32 \phi,
\end{equation}
where $h$ again denotes a collection of terms that behaves well. This is similar, as there will also be a phase correction (one needs to be more careful here, however, as $H^\rr u$ is merely bounded by $\varepsilon$ in the region $\{t-r\geq 1\}$ because of the singular behavior of $H^\rr$. This issue was first resolved in \cite{LMmodel} in another version). One can also improve the decay as $|y|\rightarrow 1$.

\subsubsection{Non-compact data}In this part, we discuss a little bit about the results on the exterior problem. Apart from the work \cite{IPmodel} which does not require compactness of the data from the beginning, we also note the existence proof by LeFloch-Ma \cite{LMextmodel}. In this case, the solution is nontrivial outside the cone $\{t-r=1\}$. Along this light cone, they utilized the structure of the Klein-Gordon equation to get the following improved decay estimates:
\begin{equation}
    |\pa^I\phi(t,x)|\leq C\varepsilon t^{-2+\delta},\quad t-|x|=1.
\end{equation}
This implies $|\Phi_\pm|\leq C\varepsilon (1-|y|^2)^{5/4-\delta}$ when $\rho\sim (1-|y|^2)^{-\frac 12}$. We see that in this case, we cannot obtain arbitrarily good decay of $a_\pm (y)$ as we want, but we can still get enough decay of them to apply the method used in this work.

{Another observation on the exterior problem is from the work on a related massive Maxwell-Klein-Gordon model \cite{18MKG}.} In the exterior of a light cone they showed the same decay rate for the following quantities of the Klein-Gordon field:
\def\lb{\underline{L}}
\begin{equation*}
    (1+r)|\slashed{\pa}\phi|,\quad (1+r)|L\phi|,\quad (1+|r-t|)|\lb\phi|
\end{equation*}
where $\slashed\pa$ is the angular derivative, $L=\pa_t+\pa_r$, and $\lb=\pa_t-\pa_r$.
This type of decay estimates is similar to the structure of solutions to a wave equation. For wave equations this can be easily obtained from the vector-field method with the use of the scaling vector field (which is not included in the symmetry of the Klein-Gordon equation). While it is not clear yet if the proof in \cite{18MKG} applies to our model, the different behavior of the Klein-Gordon fields in the interior and exterior might lead to further simplifications of the exterior proof.

\subsection{Heuristics for Scattering from infinity}\label{heuristicscattering}
We now assume the same type of asymptotic behavior as we get in the forward problem. Again, for simplicity, we only consider the source term $\phi^2$ as the treatment of the term $(\pa_t\phi)^2$ is similar. We look for solutions where the Klein-Gordon component behaves like
\begin{equation}
    \phi\sim \rho^{-\frac 32}(e^{i\rho-\frac i2 U(y)\ln\rho}a_+(y)+e^{-i\rho+\frac i2 U(y)\ln\rho}a_-(y)),
\end{equation}
and the wave component $u\sim U(y)/\rho$ towards timelike infinity. We see from above that $U(y)$ is in fact determined by $a_\pm (y)$, so we should only give $a_\pm(y)$ as the scattering data. In the forward problem we have good decay of $a_\pm(y)$, e.g. $|a_\pm (y)|\lesssim (1-|y|^2)^{5/4-\delta}$ for non-compact data and any rate for compact data. We can reduce the number of partial derivatives to make the decay rate optimal, but here it would be enough to illustrate the idea by assuming sufficiently fast decay of $a_\pm (y)$. {This assumption also gets rid of the worse behavior as $|y|\rightarrow 1$ in the quasilinear case, and as a result, the proof for the system \eqref{quasilinearsystem} will be similar and below we focus on \eqref{semilinearsystem}.}

As in the forward problem, the function $U(y)/\rho$ is not well-behaved near the light cone, and it is not a good approximation in this zone either. Therefore, we consider an approximate solution given by the equation
\begin{equation}
    -\Box u_1=2\rho^{-3}a_+(y) a_-(y)
\end{equation}
with vanishing initial data at $\{t=2\}$. We have already studied this equation, and we know that $u_1=U(y)/\rho$ when $t-r>4$. For technical reasons we use $u_1$ instead of $U(y)/\rho$ to define the phase, i.e. define $\phi_0$ as
\begin{equation}
    \phi_0:=\rho^{-\frac 32}(e^{i\rho-\frac i2 \int u_1 d\rho}a_+(y)+e^{-i\rho+\frac i2 \int u_1 d\rho}a_-(y)).
\end{equation}

Apart from $u_1$, we also have the part of the wave component governed by the oscillating source. We denote the part by $u_2$, so that
\begin{equation}
    -\Box u_2=\rho^{-3}(e^{2i\rho-i\int u_1 d\rho}a_+(y)^2+e^{-2i\rho+i\int u_1 d\rho}a_-(y)^2)
\end{equation}
with vanishing data at $\{t=2\}$.
Since $a_\pm (y)$ and $u_1$ are now given, we also have established estimates of $u_2$. We further introduce $u_3$ so that\footnote{In fact, $u_3$ would be zero if we only consider the term $\phi_0^2$. The presence of $u_3$ comes from the lower order terms in $(\pa_t\phi_0)^2$.}
\begin{equation}
    -\Box (u_1+u_2+u_3)=(\pa_t\phi_0)^2+\phi_0^2
\end{equation}
again with vanishing data. Note that the source term of $u_3$ decays faster as it is the remainder after one subtracts the leading parts. We denote
\begin{equation}\label{u_0}
    u_0:=u_1+u_2+u_3+\psi_{01},
\end{equation}
where
\begin{equation}
    \psi_{01}:=\chi\left(\frac{\langle r-t\rangle}r\right)\left(\frac{F_0(r-t,\omega)}r+\frac{F_1(r-t,\omega)}{r^2}\right),
\end{equation}
with $\chi(s)$ being a cutoff function supported in $\{s\leq 1/4\}$. We now explain the field $F_0$ and $F_1$. We need these because apart from the contribution of $u_1$, $u_2$ and $u_3$ to the radiation field, there should also be other parts (e.g. the initial data) of the wave component that affect the radiation field. Therefore, we include a piece of information on the radiation field in our scattering data, called the \textit{free radiation field} in our context, given by a function $F_0(q,\omega)$ of $q=r-t$ and $\omega$. Therefore, $F_0$ is the original prescribed radiation field subtracted by the radiation field of $u_1$, $u_2$ and $u_3$. We also consider a second-order approximation by introducing a function $F_1(r-t,\omega)$ satisfying the differential equation $\pa_q F_1(q,\omega)=\triangle_\omega F_0(q,\omega)$ and $F_1(0,\omega)=0$. This is needed for us to close the argument.


Therefore, when considering the scattering from infinity problem, we have two pieces of independent scattering data $(a_\pm (y),F_0(q,\omega))$. In view of the decay in the forward problem, we make the assumption that $F_0(q,\omega)$ (and its higher order versions) decay in $q$ at the rate $\langle q\rangle^{-1+\alpha}$ for some $\alpha>0$.

Following the idea in \cite{lindblad2017scattering} and \cite{H21} for massless problems, we seek to establish the estimate of the perturbation part. We consider a large time $T$, and solve the system backward. We let the perturbation part vanish for the ``initial" data at time $T$ so that the solution behaves like the desired asymptotics at time $T$. Denote the perturbation part by $v_T$ and $w_T$, so that $u=u_0+v_T$ and $\phi=\phi_0+w_T$ solve the system \eqref{semilinearsystem}. Then we derive
\begin{equation}\label{systemwTvT}
    \begin{split}
        -\Box v_T&=2(\pa_t \phi_0)\pa_t w_T+(\pa_t w_T)^2+2\phi_0 w_T+w_T^2+\Box \psi_{01},\\
        -\Box w_T+w_T&=(u_2+u_3+\psi_0)\phi_0+u_0 w_T+\phi_0 v_T+v_T w_T+R_0
    \end{split}
\end{equation}
where $R_0=O(t^{-\frac 72})$ decays well. For the term $(u_2+u_3+\psi_{01})\phi_0$, we note that $u_2$, $u_3$ and $\psi_{01}$ all have decay in the interior better than $\rho^{-1}$. For the wave equation we have 
\begin{equation}
     |\Box\psi_{01}|\lesssim \frac{\chi({\frac{\langle t-r\rangle}{2r}})}{\langle t+r\rangle^4}\langle q\rangle^\alpha,
\end{equation}
which is good. We will see that these terms without $w_T$ and $v_T$ determine the behavior of the system \eqref{systemwTvT}. We seek to work in some energy space, and argue that $v_T$ and $w_T$ converge as $T\rightarrow \infty$. Because the right-hand side of the system \eqref{systemwTvT} contains nonlinearity of the unknown functions, we need to estimate them with energy estimates.


\subsection{Energy estimates}
We discuss the energy estimates in this part. For the system \eqref{systemwTvT}, one may use the standard energy estimate (but the backward version, also using that $v_T$ and $w_T$ vanish at time $T$)
\begin{equation}
    (\int_{t=t_1} |\pa f|^2+m^2|f|^2 dx)^\frac 12 \leq \int_{t_1}^T ||(-\Box+m^2)f||_{L^2(\{t=s\})} ds,\quad m=0,1
\end{equation}
and the yields a term e.g. for the second equation
\begin{equation}
    \int_{t_1}^T ||(u_2+u_3+\psi_{01})\phi_0+u_0 w_T+\phi_0 v_T+v_T w_T+R_0||_{L^2(\{t=s\})} ds.
\end{equation}
Note again that the term $(u_2+u_3+\psi_{01})\phi_0$ and $R_0$ are known and one can show that the $L^2$ norm of them decays in time, say at the rate $t^{-1-\lambda}$. We also pretend at this stage that the term quadratic in perturbation, i.e. $v_T w_T$, is ignorable. Then in view of $|u_0|\lesssim \varepsilon t^{-1}$ and $|\phi_0|\lesssim \varepsilon t^{-\frac 32}$, we get the estimate
\begin{equation}
    (\int_{t=t_1} |\pa w_T|^2+|w_T|^2 dx)^\frac 12\lesssim \varepsilon (t_1)^{-\lambda}+\int_{t_1}^T \varepsilon t^{-1}||w_T||_{L^2(\{t=s\})}+\varepsilon t^{-\frac 32} ||v_T||_{L^2(\{t=s\})} ds+\text{Lower order}
\end{equation}
One can expect the decay of $||v_T||_{L^2}$ to be $1/2$-power better than $||w_T||_{L^2}$ (one get a term like $t^{-\frac 32} ||\pa w_T||_{L^2}$ in the estimate of $v_T$ instead of $t^{-1}||\pa w_T||_{L^2}$), so heuristically we get the estimate
\begin{equation}
    ||\pa w_T||_{L^2(\{t=t_1\})}+||w_T||_{L^2(\{t=t_1\})}\lesssim \varepsilon(t_1)^{-\lambda}+\int_{t_1}^T \varepsilon t^{-1} ||w_T||_{L^2(\{t=s\})}ds+\text{Lower order}.
\end{equation}
We note that in forward problems, similar integrals appear frequently, and one gets logarithmic growth of the energy as a result. However, there is no such issue for backward problems: using Gr\"{o}nwall argument and the energy vanishes at time $T$, one gets
\begin{equation}
    ||\pa w_T||_{L^2(\{t=t_1\})}+||w_T||_{L^2(\{t=t_1\})}\lesssim \varepsilon (t_1)^{-\lambda},
\end{equation}
so this is not causing problems. One also sees from this that the decay of the norm of known terms (e.g. $R_0$) determines the decay of the energy we can get. In particular, if one does not introduce the second-order approximation $\psi_{01}$ but only the $F_0$ part, the decay would not be enough for us to close the argument in our work. One can then also use this argument between time $T_1$ and $T_2$ for $\hat v:=v_2-v_1$ and $\hat w:=w_2-w_1$ to show the convergence.

However, in order to argue that nonlinear terms in \eqref{systemwTvT} are ignorable, we also need to derive $L^\infty$ estimates. Recall that for wave-Klein-Gordon coupled systems, one cannot commute the system with the scaling vector field, and the constant time slices do not work well for massive systems. Instead, it is more natural to foliate the interior by hyperboloids, on which we have a version of Klainerman-Sobolev inequality to help us derive the $L^\infty$ estimates from $L^2$ bounds. The energy hierarchy then presents a similar behavior.

Nevertheless, since we are now solving the backward problem, the solution will not be supported in the interior of any forward light cone, which makes it necessary to have something more than the standard hyperboloid foliations. 
We consider the following foliation: truncate the hyperboloid at $\{t-r=r^\sigma\}$, $0\leq \sigma<1$, and then extend each slice to the exterior using the constant time slice. In this way, the parameter $\rho$ extended to the exterior $\{t-r<r^\sigma\}$. Note that now in the exterior, the ordinary time $t$ and the parameter $\rho$ is related by $t=t(\rho)\sim \rho^{2/(1+\sigma)}$ so $\rho\sim t^\frac{1+\sigma}2$, and $dt\sim \rho^{\frac 2{1+\sigma}-1} d\rho$. We denote the whole slice extended from the part of $H_\rho=\{t^2-|x|^2=\rho^2\}$ by $\Sigma_\rho$, and $\widetilde H_\rho=\Sigma_\rho\cap H_\rho$, $\Sigma_\rho^e=\Sigma_\rho\backslash\widetilde H_\rho$. Then one has its corresponding version of energy estimates:
\begin{equation}E_w(\rho_1,v_T)^\frac 12 \lesssim E_w(\rho_2,v_T)^\frac 12
    +\int_{\rho_1}^{\rho_2} ||\Box v_T||_{L^2(\widetilde H_\rho)} d\rho+\int_{\rho_1}^{\rho_2} \rho^{\frac{2}{1+\sigma}-1} ||\Box v_T||_{L^2(\Sigma_\rho^e)} d\rho,
\end{equation}
\begin{equation}E_{KG}(\rho_1,w_T)^\frac 12 \lesssim E_w(\rho_2,w_T)^\frac 12
    +\int_{\rho_1}^{\rho_2} ||(-\Box+1) w_T||_{L^2(\widetilde H_\rho)} d\rho+\int_{\rho_1}^{\rho_2} \rho^{\frac{2}{1+\sigma}-1} ||(-\Box+1) w_T||_{L^2(\Sigma_\rho^e)} d\rho,
\end{equation}

Because of the observations similar to what we did above for constant time foliations, one can at most expect the energy decay
\begin{equation}
    E_{w}(\rho,v_T)^\frac 12\leq C\varepsilon \rho^{-\frac 32+\alpha}, \quad E_{KG}(\rho,w_T)^\frac 12\leq C\varepsilon \rho^{-1+\alpha}.
\end{equation}
The Klainerman-Sobolev inequality on hyperboloids takes care of the $L^\infty$ estimates in the interior and the nonlinear terms are indeed ignorable there. In spite of this, one does not have such an estimate in the exterior. Instead, one can derive the $L^\infty$ decay using Sobolev estimates on spheres to get $|w_T|\lesssim \varepsilon r^{-1}\rho^{-1+\alpha}$ when $t-r\geq r^\sigma$. In view of the energy estimate above and the system, in the exterior we need to control a term involving the $L^2$ norm of $v_T w_T$. This requires a control of, e.g., the integral
\begin{equation}
    \int_{\rho}^{T} \rho^{\frac 2{1+\sigma}-1} ||r^{-1} v_T||_{L^2(\Sigma_\rho^e)} ||r w_T||_{L^\infty(\Sigma_\rho^e)} d\rho
\end{equation}
where we have to control the $L^2$ norm of $v_T$ itself using Hardy's inequality. This appears in the energy estimate of $w_T$, so we need to bound this by the rate $\rho^{-1+\alpha}$. We have
\begin{equation}
    \rho^{\frac 2{1+\sigma}-1} ||r^{-1} v_T||_{L^2(\Sigma_\rho^e)} ||r w_T||_{L^\infty(\Sigma_\rho^e)}\lesssim \varepsilon^2 \rho^{\frac 2{1+\sigma}-1} \rho^{-\frac 32+\alpha} \rho^{-1+\alpha}\lesssim\varepsilon^2 \rho^{-\frac 72+\frac{2}{1+\sigma}+2\alpha},
\end{equation}
so the integral decays at the rate $\rho^{-\frac 52+\frac{2}{1+\sigma}+2\alpha}$. Let $-\frac 52+\frac{2}{1+\sigma}+2\alpha\leq -1+\alpha$, and we get $\sigma\geq (\alpha+1/2)/(3/2-\alpha)$. To make this hold, we can e.g. let $0<\alpha<\frac 16$ and $\sigma=\frac 12$, which is what we use in the proof. This deals with the control of nonlinear terms in the exterior, and now we can close the argument.

\section{Notations}

\subsection{Hyperboloidal coordinates}
We will frequently use the hyperboloidal coordinates $\rho=\sqrt{t^2-r^2}$ where $r=|x|$, and $y=x/t$, so $|y|=|x|/t$. We also define the angular variables $\omega_i=x_i/|x|$.

The $\rho$-hyperboloid is defined by $H_\rho=\{(t,x)\colon t^2-|x|^2=\rho^2\}$. For simplicity, we often write $f(x,t)$ by $f(\rho,y)=f(t(\rho,y),x(\rho,y))$ as we often work in the latter coordinates.

In $(\rho,y_i)$ coordinates, we have\footnote{We use the Einstein summation convention. Also, when the repeated index is spatial, we define the expression to be the sum regardless of whether it is upper or lower, as the spatial part of the Minkowski metric is Euclidean.} \[\pa_\rho=\frac{t}{\rho}\pa_t+\frac{x^i}\rho \pa_i,\quad \pa_{y_i}=\frac{\rho}{(1-|y|^2)^{\frac 32}}\left((y_i y_j+\delta_{ij})\pa_j+y_i \pa_t\right).\]
We also have the dual frame \[d\rho=\frac t\rho dt-\frac{x^i}\rho dx^i,\quad dy^i=-\frac{x_i}{t^2} dt+\frac 1t dx^i.\] In this way, for a symmetric $(2,0)$-tensor $T$, we can define the quantities $T^\rr$, $T^{\rho y_i}$ and $T^{y_i y_j}$.

We will define the truncated hyperboloids $\widetilde{H}_\rho$ in Section \ref{scatteringsection}.

\subsection{Vector fields}The Minkowski commuting vector fields \[\mathcal{Z}=\{\pa_\alpha,\Omega_{0i}=t\pa_i+x_i\pa_t,\Omega_{ij}=x_i\pa_j-x_j\pa_i,S=t\pa_t+x_i\pa_i\}.\]
We also define its subsets
\begin{equation}
    \mathcal{\pa}=\{\pa_\alpha,\alpha=0,1,2,3\},\quad \mathcal{L}=\{\Omega_{0i},i=1,2,3\}, \quad \mathbf{\Omega}=\{\Omega_{\alpha\beta},\alpha=0,1,2,3,\beta=1,2,3\}.
\end{equation}

We use the multi-index notation: for $I=(\alpha_1,\alpha_2,\cdots,\alpha_n)$, $J=(\beta_1,\beta_2,\cdots,\beta_m)$, we define $$Z^I=Z_1^{\alpha_1} Z_2^{\alpha_2} \cdots Z_n^{\alpha_n}$$ where $Z_1,\cdots Z_n \in \mathcal{Z}$, and similarly 
$${\pa^I L^J=\pa_0^{\alpha_0}\cdots\pa_3^{\alpha_3} L_1^{\beta_1}\cdots L_m^{\beta_m},\quad \pa^I \Omega^J=\pa_0^{\alpha_0}\cdots\pa_3^{\alpha_3} \Omega_1^{\beta_1}\cdots \Omega_m^{\beta_m},}$$ {where $I=(\alpha_0,\alpha_1,\alpha_2,\alpha_3)$, $J=(J_1,\cdots,J_m)$, $\pa_0,\cdots,\pa_3\in \pa$}, $L_1,\cdots,L_m\in\mathcal{L}$, $\Omega_1,\cdots,\Omega_m \in \mathbf{\Omega}$. The size of a multi-index $I=(\beta_1,\beta_2,\cdots,\beta_m)$ is defined as $|I|:=\sum_{i=1}^m \beta_i$. We also define $|(I_1,J_1,I_2,J_2)|:=|I_1|+|J_1|+|I_2|+|J_2|$.

Note that the rotation and boost vector fields are tangent to the hyperboloids $H_\rho$, and we have the following relation between $\pa_{y_i}$ and vector fields:
\begin{equation}\label{dy}
    \pa_{y_i}=\frac 1{1-|y|^2} \left(\Omega_{0i}+\frac{x_j}{t} \Omega_{ij}\right).
\end{equation}
Moreover, the Laplacian-Beltrami operator on the unit hyperboloid $\triangle_y$ can be expressed as
\begin{equation}\label{laplaciany}
    \triangle_y=\sum_{i=1}^2 \Omega_{0i}^2+\sum_{j=1}^3 \sum_{i=1}^j \Omega_{ij}^2.
\end{equation}

The following commutation relations are useful:
$[\pa_i,\Omega_{0j}]=\delta_{ij}\pa_t$, $[\pa_i,\Omega_{jk}]=\delta_{ij}\pa_k-\delta_{ik}\pa_j$.


\subsection{Integration}
It is convenient to consider the integral of a function $u(t,x)$ on hyperboloids with respect to the measure $dx$:
\[\int_{H_\rho} u \, dx:=\int_{\mathbb{R}^3}u(\sqrt{\rho^2+|x|^2},x)\, dx,\]
and then we can define the $L^2$ norm by
\[||u||_{L^2(H_\rho)}=(\int_{H_\rho} |u|^2\, dx)^\frac 12\]

\begin{remark}
In \cite{LMmodel} and many its subsequent works, the coordinate $(\rho,x)$ is used. The corresponding coordinate basis is $\{\overline\pa_i,\overline\pa_\rho\}$, where $\overline{\pa}_i:=t^{-1}\Omega_{0i}$, and $\overline \pa_\rho:=\frac \rho t \pa_t$. The derivative $\overline\pa_\rho$ is related to $\pa_\rho$ with the difference tangent to the hyperboloid:
\begin{equation}
    \overline\pa_\rho=\pa_\rho-\rho^{-1}\textstyle\frac rt\, \omega^i\, \Omega_{0i}.
\end{equation}
\end{remark}

\section{Existence proof for semilinear model}
We start with the proof of the global existence of the following system:
\begin{equation}\label{uphisystem}
    -\Box u=(\pa_t\phi)^2+\phi^2,\ \ \ -\Box \phi+\phi=u\phi.
\end{equation}
We impose the initial data on $\{t=2\}$, and suppose that the data is supported in $|x|\leq 1$. By local theory, we can alternatively impose the initial value on $\{\rho=2\}$ (see e.g. \cite{LMbook}).
\subsection{Bootstrap assumption}
We have the standard energy estimate on hyperboloid: define the energy
\begin{equation}
    \begin{split}
        E_{KG}(\rho,\phi):=&\int_{H_\rho} \left(\frac\rho t\right)^2|\pa_t\phi|^2+\sum_i |\overline{\pa}_i \phi|^2+|\phi|^2\, dx\\
        =&\int_{H_\rho} \left(\frac \rho t\right)^2(|\pa_\rho \phi|^2+\sum_i |\pa_i \phi|^2)+t^{-2}\sum_{i<j}|\Omega_{ij}\phi|^2+|\phi|^2\, dx.
    \end{split}
\end{equation}
This is energy flux induced on the hyperboloids with the multiplier $\pa_t$.
We have (see proof in e.g. \cite{LMmodel})
\[E_{KG}(\rho,\phi)^\frac 12\leq E_{KG}(\rho_0,\phi)^\frac 12+\int_{\rho_0}^\rho ||(-\Box+1) \phi||_{L^2(H_s)} ds.\]
For the wave component, we use the conformal energy introduced in \cite{ma2017conformal}. Let $K=\frac{t^2+r^2}t \pa_t+2r\pa_r$. We define
\[E_{\mathrm{con}}(\rho,u):=\int_{H_{\rho}} |Ku|^2+\sum_i |\rho\, \overline{\pa}_i u|^2+4uKu+4|u|^2 dx=\int_{H_\rho}|Ku+2u|^2+\sum_i |\rho\, \overline{\pa}_i u|^2 dx.\]
Using $\rho K$ as the multiplier, one has the following estimate:
\[E_{\mathrm{con}}(\rho,u)^\frac 12\leq E_{\mathrm{con}}(\rho_0,u)^\frac 12+2\int_{\rho_0}^\rho s||\Box u||_{L^2(H_s)} ds.\]
We also define the higher order energy:
\[{E_{KG;k}(\rho,\phi):=\sum_{|I|+|J|\leq k}E_{KG;k}(\rho,\pa^I L^J \phi),}\quad E_{\mathrm{con},k}(\rho,u):=\sum_{|I|+|J|\leq k}E_{\mathrm{con}}(\rho,\pa^I L^J u).\]

Fix an integer $N \geq 9$. We consider initial data satisfying the smallness condition
\begin{equation}
    E_{KG;N}(2,u)^\frac 12+E_{\mathrm{con};N}(2,u)^\frac 12\leq \varepsilon.
\end{equation}
Let $\rho^*$ be the maximal hyperboloidal time such that
\begin{equation}
    E_{KG;N}({\rho,\phi})^\frac 12\leq C_b\, \varepsilon\rho^\delta,\quad E_{\mathrm{con};N}({\rho,u})^\frac 12\leq C_b\, \varepsilon \rho^{\frac 12+\delta}
\end{equation}
holds for all $\rho\in [2,\rho^*]$, where $C_b$ is a constant that we will determine. In the remaining part of this section, the symbol ``$\lesssim$'' represents an implicit constant independent of $C_b$. We shall improve these bounds and hence show that $\rho^*=\infty$.

The bootstrap assumptions give weak decay estimates in view of the following Klainerman-Sobolev inequality (see \cite{H97} or \cite{LMmodel}):
\begin{prop}[Klainerman-Sobolev inequality on hyperboloids]
Let $f$ be a function supported in $\{|x|\leq t-1\}$. Then
\[\sup_{H_\rho} t^\frac 32 |f|\leq C\sum_{|J|\leq 2}||L^J f||_{L^2(H_\rho)}\]
for some constant $C$.
\end{prop}
Therefore, we get the following weak decay estimates:
\begin{equation}
    |\pa^I L^{J}\phi|\lesssim C_b\varepsilon t^{-\frac 32}\rho^\delta=C_b \varepsilon \rho^{-\frac 32+\delta}(\rho/t)^\frac 32,\quad |I|+|J|\leq N-2.
\end{equation}

For the wave component, we first present terms that are bounded by the conformal energy.
\begin{lemma}
We have
$$||\rho\, t^{-1} u||_{L^2(H_\rho)}\leq 2E_{\mathrm{con}}(\rho,u)^\frac 12,\quad ||\rho^2 t^{-1} \pa_\rho u||_{L^2(H_\rho)}\leq 6E_{\mathrm{con}}(\rho,u)^\frac 12.$$
\end{lemma}

\begin{proof}
The first estimate follows from applying the standard Hardy inequality to the function $u(\sqrt{\rho^2+|x|^2},x)$, and replacing the factor $|x|^{-1}$ by $t^{-1}$ as $|x|<t$ on $H_\rho$. For the second estimate, we notice that
$Ku+2u=\rho\, \pa_\rho u+\sum_i x^i\overline\pa_i u+2u$, and hence we have
\[||\rho^2 t^{-1}\pa_\rho u||_{L^2(H_\rho)}\leq ||\frac \rho t(Ku+2u)||_{L^2(H_\rho)}+||\rho\, \overline{\pa}_i u||_{L^2(H_\rho)}+2||\frac \rho t u||_{L^2(H_\rho)}\leq 6E_{\mathrm{con}}(\rho,u)^\frac 12,\]
where we use the first estimate in the last inequality.
\end{proof}

Therefore, we have the bound
\[||\rho\, t^{-1} \pa^I L^J u||_{L^2(H_\rho)}\leq 2C_b\,  \varepsilon\,  \rho^{\frac 12+\delta},\quad ||\rho^2 t^{-1} \pa_\rho u||_{L^2(H_\rho)}\leq 6\, C_b\, \varepsilon \, \rho^{\frac 12+\delta}, \quad |I|+|J|\leq N.\]

To get the decay for derivatives in $\pa_\rho$ direction, we need estimates of the commutator.

\begin{lemma}
We have $[\Omega_{0i},S]=0$ and $[\Omega_{ij},S]=0$. Therefore, we have $[\Omega_{\alpha\beta},\pa_\rho]=0$.
\end{lemma}
\begin{proof}
We have
\begin{equation*}
    \begin{split}
        [\Omega_{0i}&,S]=(t\pa_i+x_i \pa_t)(t\pa_t+\sum_j x_j \pa_j)-(t\pa_t+\sum_j x_j \pa_j)(t\pa_i+x_i \pa_t)\\
        &=(t\pa_i+x_i \pa_t)(t)\pa_t+\sum_j (t\pa_i+x_i \pa_t)(x_j) \pa_j)-(t\pa_t+\sum_j x_j \pa_j)(t)\pa_i-(t\pa_t+\sum_j x_j \pa_j)(x_i) \pa_t\\
        &=x_i\pa_t+t\pa_i-t\pa_i-x_i\pa_t=0.
    \end{split}
\end{equation*}
The verification for the second identity is similar and we omit it. The last result follows from $\pa_\rho=\rho^{-1}S$ and $\Omega_{\ab}\rho=0$.
\end{proof}

We are now ready to estimate the decay for $\pa_\rho$ derivatives. Using the Klainerman-Sobolev inequality and $|L^J(t^{-1})|\lesssim t^{-1}$, we have \begin{multline}
    \sup_{H_\rho}t^\frac 32|\frac{\rho^2} t\pa_\rho \pa^I L^J u|\lesssim \sum_{|J'|\leq 2}||L^{J'}(\frac{\rho^2} t \pa_\rho \pa^I L^J u)||_{{L^2(H_\rho)}}\lesssim \rho^2\sum_{|J_1'|+|J_2'|\leq 2}||L^{J_1'} (\frac 1t) L^{J_2'} \pa_\rho \pa^I L^J u||_{L^2(H_\rho)}\\
    \lesssim \rho^2\sum_{|J_2'|\leq 2}||\frac 1t \pa_\rho L^{J_2'} \pa^I L^J u||_{L^2(H_\rho)}\lesssim E_{\mathrm{con},N}(\rho,u)^\frac 12\lesssim C_b \varepsilon \rho^{\frac 12+\delta}, \quad |I|+|J|\leq N-2.
\end{multline}
Therefore we have $|\pa_\rho \pa^I L^J u|\lesssim C_b \varepsilon\rho^{-\frac 32+\delta} t^{-\frac 12}$.
Similarly we can get $|\pa_\rho \pa^I L^J \phi|\lesssim C_b \varepsilon t^{-\frac 32}\rho^\delta$ for $|I|+|J|\leq N-2$.

\subsection{Decay estimates}
We now perform the $L^\infty$-$L^\infty$ estimates to improve the decay of both components. We need the following lemma, which follows from Lemma \ref{lemmaintegral} and Remark \ref{remarklambdalemma}. This is in fact proven in {\cite[Prop 3.1]{LMmodel}}. 
\begin{lemma}\label{lemmawaveexistencedecay}
If $u$ is the solution to $-\Box u=F$ with vanishing initial data, with $F$ {supported in $\{t-r\geq 1\}$} and $|F|\leq C \varepsilon t^{-3+a}$ for some $0\leq a<\frac 12$, then $|u|\leq C\varepsilon t^{-1+a}(1-\frac rt)^{a}=\varepsilon t^{-1}(t-r)^{a}$.
\end{lemma}

Commuting the equation with vector fields, we have
\[-\Box \pa^I L^J u=\pa \pa^{I_1} L^{J_1} \phi\cdot \pa \pa^{I_2} L^{J_2}\phi{+\pa^{I_1} L^{J_1} \phi\cdot\pa^{I_2} L^{J_2}\phi}\]
where the right hand side means a sum of terms of this form, where $|I_1|+|I_2|\leq |I|$ and $|J_1|+|J_2|\leq |J|$.

Then using the weak decay for the right hand side, we get $|\Box \pa^I L^J u|\lesssim C_b^2 \varepsilon^2 t^{-3}\rho^{2\delta}\lesssim C_b^2 \varepsilon^2 t^{-3+2\delta}$, so by considering the part governed by this source term (with zero initial data) using Lemma \ref{lemmawaveexistencedecay}, and the free part from initial data (which behaves better), we get $|\pa^I L^J u|\lesssim C_b\varepsilon t^{-1} (t-r)^{2\delta}=C_b\varepsilon t^{-1}\rho^{2\delta}(\rho/t)^{2\delta}$.

Now we try to improve the decay of the Klein-Gordon field.  We have the decomposition of the wave operator:
\begin{equation}
-\Box= \pa_\rho^2 +3\rho^{-1}\pa_\rho -\rho^{-2}\triangle_y,
\end{equation}
so one can show for $\Phi=\rho^\frac 32\phi$ that
$$\pa_\rho^2 \Phi+\Phi=u\Phi+\rho^{-1/2}(\triangle_y \phi+\textstyle\frac 34\phi).$$
Let $\Phi_\pm=e^{\mp i\rho} (\pa_\rho \Phi \pm i\Phi)$. Then we have the equation for $\Phi_{\pm}$:
\begin{equation}
    \pa_\rho \Phi_\pm=e^{\mp i\rho}(u\Phi+\rho^{-1/2}(\triangle_y \phi+\textstyle\frac 34\phi)).
\end{equation}
Since $\Phi=-\frac i2 (e^{i\rho} \Phi_+ -e^{-i\rho}\Phi_-)$, we get
\begin{equation}
    \pa_\rho \Phi_\pm=\mp\frac i2 u\Phi_\pm \pm \frac i2 e^{\mp 2i\rho} u\Phi_\mp +e^{\mp i\rho}\rho^{-1/2}(\triangle_y \phi+\textstyle\frac 34\phi).
\end{equation}

We consider $\Phi_+$ for example. We have
\begin{equation}\label{drhoPhip}
    \pa_\rho \Phi_+=-\frac i2 u\Phi_+ + \frac i2 e^{-2i\rho} u\Phi_- +e^{- i\rho}\rho^{-1/2}(\triangle_y \phi+\textstyle\frac 34\phi).
\end{equation}
This implies that
\begin{equation}\label{firstorderODEKG}
    \pa_\rho\left(e^{\frac i2\int u(\rho,y) d\rho} \left(\Phi_+ +\textstyle\frac 14 u e^{-2i\rho}\Phi_-\right)\right)=h(\rho,y),
\end{equation}
where
\begin{equation}
    h(\rho,y)=\frac 14 e^{-2i\rho} \pa_\rho (e^{\frac i2\int u(\rho,y)d\rho} u\Phi_-)+e^{-i\rho+\frac i2\int u(\rho,y)d\rho} \rho^{-1/2} (\triangle_y \phi+\textstyle\frac 34\phi).
\end{equation}
In view of the relation $\triangle_y=\sum_{\alpha<\beta}\Omega_{\alpha\beta}^2$ and the weak decay estimates above, we have \[|h(\rho,y)|\lesssim C_b^2\varepsilon^2\rho^{-2+2\delta}(\rho/t)^2+C_b^2\varepsilon^2\rho^{-2+5\delta}(\rho/t)^{\frac 72+4\delta}+C_b\varepsilon\rho^{-2+\delta}(\rho/t)^\frac 32.\]

Since everything is supported in $\{t-r\geq 1\}$, we have $\rho>t/\rho$. Integrating along the hyperboloidal ray, we get
\begin{multline*}
    |\Phi_+|\lesssim (|\Phi_+|+|u||\Phi_-|)|_{\rho=2}+|u||\Phi_-|+C_b^2\varepsilon^2\rho^{-1+5\delta}(\rho/t)^{\frac 72+\delta}+C_b\varepsilon\rho^{-1+\delta}(\rho/t)^\frac 32\\
    \lesssim c_+(x/t)+C_b^2\varepsilon^2\rho^{-1+3\delta} (\rho/t)^{\frac 52+2\delta}+C_b^2\varepsilon^2\rho^{-1+5\delta}(\rho/t)^{\frac 72+\delta}+C_b\varepsilon\rho^{-1+\delta}(\rho/t)^\frac 32\lesssim C_b\varepsilon (\rho/t)^{\frac 52-\delta},
\end{multline*}
where $c_+(y)$ is a function compactly supported in $\{|y|<1\}$.
Clearly we can get the same estimate for $\Phi_-$. Therefore we obtain $|\phi|\lesssim C_b\varepsilon\rho^{-\frac 32}(\rho/t)^{\frac 52-\delta}$. {Note that if we commute the equation with $\pa$ once, we get $-\Box(\pa\phi)+\pa\phi=u(\pa\phi)+(\pa u)\phi$. The last term behaves well in view of the weak decay estimates of derivatives of $u$, hence ignorable, so we can do the same thing for $\pa\phi$ to get $|\pa\phi|\lesssim C_b \varepsilon \rho^{-\frac 32} (\rho/t)^{\frac 52-\delta}$.}

To get the estimate with vector fields, we do the induction. Suppose we have the estimate $|\pa^I L^J \phi|{+|\pa \pa^I L^J\phi|}\lesssim C_b\varepsilon \rho^{-\frac 32+c_{k-1}\varepsilon}(\rho/t)^{\frac 52-\delta}$ for $|I|+|J|\leq k-1$. We improve the wave component first. We have $|\Box \pa^I L^J u|\lesssim C_b^2\varepsilon^2\rho^{-3+2c_{k-1}\varepsilon}(\rho/t)^{5-2\delta}$. This gives the improved estimate by using Lemma \ref{lemmawaveexistencedecay} with $\delta$ replaced by $c_{k-1}\varepsilon$, i.e.,
\begin{equation}
    |\pa^I L^J u|\lesssim C_b\varepsilon t^{-1} (t-r)^{2c_{k-1}\varepsilon}=C_b\varepsilon t^{-1}\rho^{2c_{k-1}\varepsilon}(\rho/t)^{2c_{k-1}\varepsilon}
\end{equation}
{for $|I|+|J|\leq k-1$}.
Note that for $k=0$ which we have already dealt with, we get $|u|\lesssim C_b\varepsilon t^{-1}$.

Now for the Klein-Gordon equation, we have
\[-\Box \pa^I L^J \phi+\pa^I L^J \phi=u \pa^I L^J \phi+\sum_{\substack{|I_1|+|I_2|\leq |I|,|J_1|+|J_2|\leq |J|\\|I_1|+|J_1|<|I|+|J|}}\pa^{I_1} L^{J_1} u\, \pa^{I_2} L^{J_2}\phi\]
Then similarly we have for $\Phi^{I,J}:=\rho^\frac 32 \pa^I L^J \phi$ that
$$\pa_\rho\left(e^{\frac i2\int u(\rho,y) d\rho} \left(\Phi^{I,J}_+ +\textstyle\frac 14 u e^{-2i\rho}\Phi^{I,J}_-\right)\right)=h^{I,J}(\rho,y),$$
where
\begin{multline}
    h^{I,J}(\rho,y)=\frac 14 e^{-2i\rho} \pa_\rho (e^{\frac i2\int u(\rho,y)d\rho} u\Phi^{I,J}_-)+e^{-i\rho+\frac i2\int u(\rho,y)d\rho} \rho^{-1/2} (\triangle_y \pa^I L^J \phi+\textstyle\frac 34 \pa^I L^J \phi)\\
    +\rho^\frac 32 \sum_{\substack{|I_1|+|I_2|\leq |I|,|J_1|+|J_2|\leq |J|\\|I_1|+|J_1|<|I|+|J|}}\pa^{I_1} L^{J_2} u\, \pa^{I_2} L^{J_2}\phi.
\end{multline}
Then
$$|h^{I,J}(\rho,y)|\lesssim C_b^2\varepsilon^2\rho^{-2+2\delta}(\rho/t)^2+C_b^2\varepsilon^2\rho^{-2+5\delta}(\rho/t)^{\frac 72+4\delta}+C_b\varepsilon\rho^{-2+\delta}(\rho/t)^\frac 32+C_b^2\varepsilon^2\rho^{-1+3c_{k-1}\varepsilon}(\rho/t)^{\frac 72-\delta+2c_{k-1}\varepsilon}$$
for $|I|+|J|\leq N-4$.
Therefore, similar to the zero-order case, we get $|\pa^I L^J \phi|\lesssim C_b \varepsilon \rho^{-\frac 32+3c_{k-1}\varepsilon}(\rho/t)^{\frac 52-\delta}$. {As above, we can commute the equation with $\pa$ once to get the same decay estimate for $|\pa\pa^I L^J\phi|$.} This closes the induction once we make $c_k>3c_{k-1}$ (we shall let $c_k>4c_{k-1}$ for later use).

\subsection{Improved energy estimate}
We now improve the bootstrap bounds. For zeroth order energy we have
\[E_{KG;0}(\rho,\phi)^\frac 12\lesssim E_{KG;0}(2,\phi)^\frac 12+\int_2^\rho ||u||_{L^\infty(H_s)}||\phi||_{L^2(H_s)} ds\lesssim \varepsilon+C_b \varepsilon \int_{2}^\rho s^{-1} E_{KG;0}(s,\phi)^\frac 12 ds,\]
so by Gr\"{o}nwall's inequality we get $E(\phi,\rho)^\frac 12\leq C\varepsilon \rho^{CC_b\varepsilon}$.
For the wave component we have
\begin{multline}
    E_{con;0}(\rho,u)^\frac 12\lesssim E_{con;0}(2,u)^\frac 12+\int_2^\rho s(||\phi||_{L^\infty(H_s)}||\phi||_{L^2(H_s)}+||\frac ts \pa \phi||_{L^\infty(H_s)}||\frac st \pa\phi||_{L^2(H_s)}) ds\\
    \lesssim \varepsilon+C_b \varepsilon^2 \int_2^\rho s^{-\frac 12} s^{CC_b\varepsilon} ds,
\end{multline}
so $E_{\mathrm{con}}(u,\rho)^\frac 12\lesssim \varepsilon+C\varepsilon^2 \rho^{\frac 12+CC_b \varepsilon}$. We want to prove
\[E_{KG;k}(\rho,\phi)^\frac 12+\rho^{-\frac 12} E_{{con},k}(\rho,u)^\frac 12\lesssim (\varepsilon+C_b\varepsilon^2)\rho^{(c_k +CC_b)\varepsilon},\]
where $C$ represents some constant, and may change for different $k$, but is independent of $C_b$. From above we know that this holds when $k=0$. Now suppose it holds for $k-1$. Then we have
\begin{multline}
    E_{KG;k}(\phi)^\frac 12\lesssim E_{KG;k}(2,\phi)^\frac 12\\
    +\sum_{\substack{|(I_1,I_2,J_1,J_2)|\leq k\\ |I_1|+|J_1|\leq k/2}}\int_{2}^\rho ||\pa^{I_1}L^{J_1} u||_{L^\infty(H_s)}||\pa^{I_2}L^{J_2} \phi||_{L^2(H_s)}+||\pa^{I_1} L^{J_1} u||_{L^2(H_s)}||\pa^{I_2} L^{J_2} \phi||_{L^\infty(H_s)}ds\\
    \lesssim \varepsilon+C_b\varepsilon\int_2^\rho s^{-1}(E_{KG;k}(s,\phi)^\frac 12+(\varepsilon+C_b\varepsilon^2)(s^{2c_{k-1}\varepsilon+c_{k-1}\varepsilon+CC_b\varepsilon}))\\
    +s^{-\frac 32} (E_{{\mathrm{con}};k}(\phi,s)^\frac 12+(\varepsilon+C_b\varepsilon^2)(s^{\frac 12+c_{k-1}\varepsilon+3c_{k-1}\varepsilon+CC_b\varepsilon}))ds\\
    \lesssim \varepsilon+C_b\varepsilon \int_2^\rho s^{-1} E_{KG;k}(s,\phi)^\frac 12+s^{-\frac 32} E_{{\mathrm{con}};k}(s,u)^\frac 12 ds
    +C_b\varepsilon\, (\varepsilon+C_b\varepsilon^2)\frac{1}{(4c_{k-1}+CC_b)\varepsilon}\rho^{4c_{k-1}\varepsilon+CC_b\varepsilon},
\end{multline}
and
\begin{multline}
    E_{\mathrm{con};k}(\rho,u)^\frac 12\lesssim E_{{\mathrm{con}};k}(2,u)^\frac 12+\sum_{\substack{|(I_1,I_2,J_1,J_2)|\leq k\\ |I_1|+|J_1|\leq k/2}}\int_2^\rho s||\frac ts\pa \pa^{I_1} L^{J_1} \phi||_{L^\infty(H_s)}||\frac st\pa \pa^{I_2} L^{J_2} \phi||_{L^2(H_s)}\\
    +s||\pa^{I_1} L^{J_1} \phi||_{L^\infty(H_s)}||\pa^{I_2} L^{J_2} \phi||_{L^2(H_s)}ds \\
    \lesssim \varepsilon+C_b\varepsilon\int_2^\rho s(s^{-\frac 32}E_k(\phi,s)^\frac 12+(\varepsilon+C_b\varepsilon^2)s^{-\frac 32+3c_{k-1}\varepsilon}s^{c_{k-1}\varepsilon+CC_b\varepsilon})ds\\
    \lesssim \varepsilon+C_b\varepsilon \int_2^\rho s^{-\frac 12} E_k(\phi,s)^\frac 12 ds+C_b\varepsilon (\varepsilon+C_b\varepsilon^2)\rho^{\frac 12+4c_{k-1}\varepsilon+CC_b\varepsilon}.
\end{multline}
Therefore, adding two estimates together, we have for $E_k(\rho)^\frac 12:=E_{KG;k}(\rho,\phi)^\frac 12+\rho^{-\frac 12}E_{{con},k}(\rho,u)^\frac 12$
\[E_k(\rho)^\frac 12\lesssim \varepsilon+C_b\varepsilon \int_2^\rho s^{-1} E_k(s)^\frac 12 ds+(\varepsilon+C_b\varepsilon^2)\rho^{4c_{k-1}\varepsilon+CC_b\varepsilon}.\]
Then using Gr\"{o}nwall's inequality, we get
\[E_k(\rho)^\frac 12\leq C(\varepsilon+(\varepsilon+C_b\varepsilon^2)\rho^{4c_{k-1}\varepsilon+CC_b\varepsilon})\rho^{CC_b\varepsilon}\leq C(\varepsilon+C_b\varepsilon^2)\rho^{4c_{k-1}\varepsilon+CC_b\varepsilon},\]
so as long as $4c_{k-1}<c_k$, we close the induction argument.

Then if we set $C_b$ big and let $\varepsilon<\frac{C_b-2C}{2CC_b}$, we can improve the bootstrap assumptions by replacing the coefficient $C_b \varepsilon$ by $\frac 12 C_b \varepsilon$. Therefore, we have proved the global existence.


\section{Semilinear model}
In this section, we prove Theorem \ref{thmasymptotics}.
From above we know that the global solution exists, and we have the following decay estimates:
\begin{equation}\label{phi}
    |\phi|+|\pa_\rho \phi|\lesssim \varepsilon \rho^{-\frac 32}(\rho/t)^{\frac 52-\delta},\quad |u|\lesssim \varepsilon t^{-1}
\end{equation}
\begin{equation}\label{Lphi}
    |\pa^I L^J \phi|+|\pa_\rho \pa^I L^J \phi|\lesssim \varepsilon \rho^{-\frac 32+\delta}(\rho/t)^{\frac 52-\delta},\quad |L^J u|\lesssim \varepsilon t^{-1}\rho^{\delta},\quad |I|+|J|\leq N-4
\end{equation}
\begin{equation}\label{duconformaldecay}
    |\pa_\rho \pa^I L^J u|\lesssim \varepsilon \rho^{-2+\delta}(\rho/t)^\frac 12, \quad |I|+|J|\leq N-2
\end{equation}
for some small $\delta>0$.

From this we also get the estimates of the derivatives in $y$ coordinates. Using \eqref{dy} and \eqref{laplaciany}, we have
\begin{equation}\label{dyonuandphi}
    |\pa_{y_i} u|\lesssim \varepsilon (1-|y|^2)^{-\frac 12}\rho^{-1+\delta} \,\quad |\pa_{y_i}\phi|\lesssim \varepsilon (1-|y|^2)^{\frac 14-\frac\delta 2}\rho^{-\frac 32+\delta},\quad |\triangle_y \phi|\lesssim \varepsilon \rho^{-\frac 32+\delta}(\rho/t)^{\frac 52-\delta}.
\end{equation}


\begin{prop}
We have \begin{equation}\label{duimproved2}
    |\pa_t \pa^I L^J u|\lesssim \varepsilon t^{-1} (1+|t-r|)^{-1+2\delta},\quad |I|+|J|\leq N-4.
\end{equation}
\end{prop}
\begin{proof}
This holds immediately from \eqref{duconformaldecay} when $r\leq\frac t2$. When $\frac t2<r<t$ we use the decomposition $-\Box u=\frac 1r (\pa_t+\pa_r)(\pa_t-\pa_r)(ru){+\frac 1{r^2} \triangle_\omega u}$, and integrate along $t-r=\mathrm{const}$ backward to hit the boundary $\{r=\frac t2\}\cup\{\rho=2\}$. We then using $|\Box u|\lesssim \varepsilon^2 t^{-3}$ to get
\[|w(t-r)(\pa_t-\pa_r)(ru)|\lesssim |w(t-r)(\pa_t-\pa_t)(ru)|\big|_{(2(t-r),t-r)}+\int_{3(t-r)}^{t+r} \varepsilon^2 rt^{-3}w(t-r)\, ds\]
where we can take the weight function $w(t-r)=(1+|t-r|)^{1-\delta}$. Then using the decay we have, we get
\[|(\pa_t-\pa_r) u|\lesssim t^{-1} (1+|t-r|)^{-1+\delta} (\varepsilon+\varepsilon^2 \delta^{-1})\lesssim \varepsilon t^{-1}(1+|t-r|)^{-1+\delta}.\]
Since $(\pa_t+\pa_r)u=\frac{\rho}{t+r}\pa_\rho u+\frac{t}{t+r}\omega^i \overline\pa_i u$ behaves better, we get the estimate for $|I|=|J|=0$. The case with vector fields holds similarly.
\end{proof}

\subsection{The Klein-Gordon field}
We now give an asymptotic description of the Klein-Gordon field.
\begin{prop}Suppose $(u,\phi)$ solves the system, and the decay estimates above hold. Then
$$\phi=\rho^{-\frac 32} \left (e^{i\rho-\frac i2\int u(\rho,y)d\rho}a_+(y)+h_+(\rho,y)\right)+\rho^{-\frac 32} \left (e^{-i\rho+\frac i2\int u(\rho,y)d\rho}a_-(y)+h_-(\rho,y)\right),$$
in the region $\{t-r\geq 1\}$, with the estimate
$$|a_\pm(y)|\lesssim \varepsilon (1-|y|^2)^{\frac 74-\delta},\ \ |h_\pm(\rho,y)|\lesssim \varepsilon \rho^{-1+\delta} (1-|y|^2)^{\frac 54-\frac \delta 2}.$$
for $|y|<1$. A similar expansion holds for $\pa_\rho \phi$.
\end{prop}

\begin{proof}
Recall in \eqref{firstorderODEKG}, we have
$$\pa_\rho\left(e^{\frac i2\int u(\rho,y) d\rho} (\Phi_+ +\textstyle\frac 14 u e^{-2i\rho}\Phi_-)\right)=h(\rho,y)$$
for $\Phi_\pm$ defined there, where
\begin{equation}\label{h}
    h(\rho,y)=\textstyle\frac 14 e^{-2i\rho} \pa_\rho (e^{\frac i2\int u(\rho,y)d\rho} u\Phi_-)+e^{-i\rho+\frac i2\int u(\rho,y)d\rho} \rho^{-1/2} (\triangle_y \phi+\frac 34\phi).
\end{equation} 
Using the decay estimates above, we have the bound $|h(\rho,y)|\lesssim \varepsilon \rho^{-2+\delta}(\rho/t)^{\frac 52-\delta}$. 
Then we integrate along the direction of $\pa_\rho$. Again we note that everything is zero when $t-r<1$, i.e. $\rho<\left (\frac{1+|y|}{1-|y|}\right)^\frac 12$, so the integration starts at $\rho=\rho(y)=\max\{2,\left (\frac{1+|y|}{1-|y|}\right)^\frac 12\}$. We have
\begin{multline}
    \Phi_+= -\left(\Phi_+(2,y)+\frac 14 e^{-4i} u(2,y)\Phi_-(2,y)\right)e^{-\frac i2\int u(\rho,y)d\rho}+\left(\int_{\rho(y)}^\infty h(\rho,y)d\rho\right)e^{-\frac i2\int u(\rho,y) d\rho}\\
    -\left(\int_\rho^\infty h(\tau,y) d\tau\right)e^{-\frac i2 \int u(\rho,y)d\rho}-\frac 14 e^{-4i\rho} u\Phi_-
\end{multline}
In this way, we have already written $\Phi_+$ as
\[\Phi_+=b_+(y)e^{-\frac i2\int u(\rho,y)d\rho}+h_+(\rho,y),\]
with required decay.

Repeating the same steps for $\Phi_-$, and using $\Phi=\frac{1}{2i}(e^{i\rho}\Phi_+-e^{-i\rho}\Phi_-)$, we obtain the expression of $\phi$.
\end{proof}

\begin{lemma}\label{dalemma}
$|\Omega^2_{\alpha\beta} a_\pm (y)|\lesssim \varepsilon (1-|y|^2)^{\frac 74-\delta}$. As a result, $|\nabla a_\pm(y)|\lesssim \varepsilon (1-|y|^2)^{\frac 34-\delta}$ for $|y|<1$.
\end{lemma}


\begin{proof}
We have e.g.
\begin{equation}
    a_+(y)=(2i)^{-1}b_+(y)=-(2i)^{-1}\left(\Phi_+(2,y)+\frac 14 e^{-4i} u(2,y)\Phi_-(2,y)+\int_{\rho(y)}^\infty h(\rho,y)d\rho\right),\quad \rho(y)\sim (1-|y|^2)^{-\frac 12}.
\end{equation}
Recall the vector fields and $\nabla_y$ are related by \eqref{dy}, and the rotation can be expressed by boost vector fields: $\Omega_{ij}=\frac {x_i}t{\Omega_{0j}}-\frac{x_j}t\Omega_{0i}$. Then the estimates follow directly by applying the $y$ derivatives and using the decay estimates with vector fields.
\end{proof}




\subsection{Wave equation with model source}
We now have \[\phi^2=2\rho^{-3} a_+(y) a_-(y)+\rho^{-3} \left(e^{2i\rho-i\int u d\rho} a_+(y)^2+e^{-2i+i\int ud\rho}a_-(y)^2\right)+R\]
when $t-r>1$, and we know that $\phi$ vanishes outside this region. This appears as the source term of the wave equation. The remainder $R$ satisfies $|R|\lesssim \rho^{-4+\delta} (1-|y|^2)^{3-\frac 32\delta}$. This is better than $t^{-4+\delta}$ when $t-r>1$.

For technical reasons we still use this expression when $t-r\leq 1$, i.e. let \[R=-2\rho^{-3} a_+(y) a_-(y)-\rho^{-3} \left(e^{2i\rho-i\int u d\rho} a_+(y)^2+e^{-2i+i\int ud\rho}a_-(y)^2\right)\] in this region, where $|R|\lesssim \varepsilon^2 \rho^{-3}(1-|y|^2)^{\frac 72-2\delta}\lesssim \varepsilon^2 t^{-5+2\delta}$ since $t-r<1$. Therefore we have $|R|\lesssim \varepsilon^2 t^{-4+\delta}$ everywhere in $\{t>r\}$.

\begin{remark}
For the other source term of the wave equation, $(\pa_t\phi)^2$, we can commute the Klein-Gordon equation with $\pa_t$ to get $-\Box (\pa_t\phi)+\pa_t\phi=u(\pa_t\phi)+(\pa_t u)\phi$. Using \eqref{duimproved2}, the last term here satisfies the decay $|(\pa_t u)\phi|\lesssim \varepsilon^2 \rho^{-2+\delta} \rho^{-\frac 32} (\frac\rho t)^{\frac 52-\delta}\lesssim \varepsilon^2\rho^{-\frac 72+\delta} (\frac \rho t)^{\frac 52-\delta}$. This is the same decay rate as the term with $\triangle_y \phi$ in \eqref{h}, so we can get a similar asymptotic expansion. Therefore, we focus on the part $\phi^2$.
\end{remark}

We decompose the source into several parts.

\subsubsection{The non-oscillating source}
We consider $-\Box u_1=t^{-3}P(x/t)$ the equation with vanishing initial data on $\{t=2\}$, where $$P(y)=\frac{t^3}{\rho^3} a_+(y) a_-(y)=\frac 1{(1-|y|^2)^\frac 32} a_+(y) a_-(y).$$ Using the estimates above we have $|P(y)|\lesssim \varepsilon^2 (1-|y|^2)^{2-2\delta}$, and $|\nabla P(y)|\lesssim \varepsilon^2 (1-|y|^2)^{1-2\delta}$.

We first study the asymptotic behavior of $u_1$ towards null infinity. We have $|P(x/t)|\lesssim \varepsilon^2 (\rho/t)^{4-4\delta}$. Then in view of $P(x/t)$ is nonzero only when $t>r$, we have
$$\langle t\rangle^{-3}|P(x/t)|\lesssim \varepsilon^2 (1+t+r)^{-5+2\delta} (t-r)_+^{2-2\delta},$$
where $x_+:=\max\{x,0\}$.
Decomposing the wave operator in $\pa_t+\pa_r$ and $\pa_t-\pa_r$ direction, we get
$$|(\pa_t-\pa_r)(\pa_t+\pa_r)(ru_1)|\lesssim \varepsilon^2 (1+t+r)^{-4+2\delta} (t-r)_+^{2-2\delta}+r^{-1}|\triangle_\omega u_1|.$$
To establish the bound for $\triangle_\omega u_1$, we recall that $\triangle_\omega=\sum_{i<j} \Omega_{ij}^2$. Therefore, we have\[-\Box (\triangle_\omega u_1)=\rho^{-3} \sum_{i<j}\Omega_{ij}^2(a_+(y) a_-(y)).\] Then, using Lemma \ref{dalemma} and Lemma \ref{lemmawaveexistencedecay}, we get $|\triangle_\omega u_1|\lesssim \varepsilon^2 (1+t+r)^{-1}$.

For any point in the region $\{t>r/2\}$, we integrate along $\pa_t-\pa_r$ direction to $t-r=0$. We have
$$|(\pa_t+\pa_r)(ru_1)|\lesssim \varepsilon (1+t+r)^{-4+2\delta}(t-r)_+^{3-2\delta}+\varepsilon (1+t+r)^{-2}(t-r)_+.$$
Now integrating along $\pa_t+\pa_r$ direction over $r\in [r_1,r_2]$, we have
$$|r_1 u_1(r_1,q,\omega) - r_2 u_1(r_2,q,\omega)|\lesssim \varepsilon (1+2r_1+q_-)^{-3+2\delta} q_-^{3-2\delta}+\varepsilon (1+2r_1+q_-)^{-1} q_-,$$
where $q=r-t$, and $q_-=\max\{-q,0\}$. This means that $\lim\limits_{r\rightarrow \infty} ru_1(r,q,\omega)$ exists. Denote the limit by $F_1(q,\omega)$, and we have
$$|ru_1(r,q,\omega)-F_1(q,\omega)|\lesssim \varepsilon (1+t+r)^{-1}q_-.$$


\vspace{2ex}

We now turn to the behavior of $u_1$ in the interior. Using the representation formula, and the change of variables $\lambda=s/t$, $\eta=(1-\lambda)^{-1}t^{-1}\bar z$, we have
\begin{equation}\label{lambdaformula}
\begin{split}
    u_1(t,x)&=\frac 1{4\pi}\int_2^t \frac{1}{t-s} \int_{|\bar{z}|=t-s} s^{-3} P\left (\frac{x-\bar z}{s}\right ) \, d\sigma(\bar z) ds\\
    &=\frac 1{4\pi} \int_{\frac 2t}^1 \frac{1}{1-\lambda} \int_{|\bar z|=(1-\lambda)t} (\lambda t)^{-3} P\left (\frac{x-\bar z}{\lambda t}\right )\, d\sigma(\bar z) d\lambda\\
    &=\frac 1{4\pi}\int_{\frac 2t}^1 \int_{\mathbb{S}^2} (1-\lambda)t^{-1}\lambda^{-3} P\left (\frac{x/t-(1-\lambda)\eta}{\lambda}\right )\, d\sigma(\eta) d\lambda
    \end{split}
\end{equation}
Notice that if $\lambda$ near $0$ satisfies $1-\lambda-r/t>\lambda$, i.e., $\lambda<\frac 12(1-\frac rt)$, then the integrand is zero in view of the support of $P$. Therefore, if $2/t<\frac 12(1-r/t)$, i.e., $t-r>4$, we can replace the lower bound $2/t$ by $\frac 12(1-r/t)$. As a result, in the region $\{t-r>4\}$ we have $u_1=t^{-1} \widetilde U(x/t)$, where $$\widetilde U(y):=\frac 1{4\pi}\int_{\frac 12(1-|y|)}^1 \int_{\eta\in \mathbb{S}^2} (1-\lambda) \lambda^{-3} P\left (\frac{y-(1-\lambda)\eta}{\lambda}\right )\, d\sigma(\eta) d\lambda.$$

We will need to deal with this type of integral several times in this work. We first analyze the support. Since $P(y)$ is nonzero only when $|y|<1$, we need to have $|\frac xt-(1-\lambda)\eta|^2<\lambda^2$ if the integrand is nonzero. Expanding the square, we get $2(1-\frac xt\cdot\eta)\lambda-|\frac xt-\eta|^2>0$, i.e. $\lambda>\frac{|\frac xt-\eta|^2}{2(1-\frac xt\cdot \eta)}=\frac{|y-\eta|^2}{2(1-y\cdot\eta)}$ {using the shorthand notation $y=x/t$}. We will need the following lemma.

\begin{lemma}\label{boundforetaquantity}
For $|y|<1$ and $\eta\in\mathbb{S}^2$, we have \[\frac 12(1-|y|)\leq \frac{|y-\eta|^2}{2(1-y\cdot \eta)}<1.\]
\end{lemma}

\begin{proof}
Notice that $2(1-y\cdot \eta)=(1-|y|^2)+|y-\eta|^2$. Then we have
\[\frac{|y-\eta|^2}{2(1-y\cdot \eta)}=\frac{|y-\eta|^2}{(1-|y|^2)+|y-\eta|^2}=\frac{1}{\frac{1-|y|^2}{|y-\eta|^2}+1},\]
and the upper bound follows directly. The lower bound holds once we apply the inequality $|y-\eta|\geq 1-|y|$ to the last expression.
\end{proof}

\begin{lemma}\label{lemmaintegral}
Suppose $Q$ is a function satisfying $|Q(z)|\lesssim \varepsilon^2 (1-|z|^2)^\alpha$ for $|z|<1$, where $-\frac 34\leq\alpha\leq 1$, and $\gamma,\beta,\mu,\nu\geq 0$. Then
\begin{multline}\int_{\mathbb{S}^2}\int_{\frac {|y-\eta|^2}{2(1-y\cdot\eta)}}^{1}  \lambda^{-3-\gamma} Q\left(\frac{y-(1-\lambda)\eta}{\lambda}\right) \left(\lambda^2-|y-(1-\lambda)\eta|^2\right)_+^\beta \left|y-\eta\right|^\mu \left(1-y\cdot \eta\right)^{-\nu}  d\lambda d\sigma(\eta)\\
\lesssim \varepsilon^2\left(1-|y|\right)^{-\gamma-\nu+2\beta+\mu}.
\end{multline}
\end{lemma}

\begin{proof}
Without loss of generality we can assume $x/|x|=(1,0,0)$. For $\frac rt>\frac 89$, we split the integral of $\lambda$ using $1=\mathbbm{1}_{\{\lambda\geq\frac 18\}}+\mathbbm{1}_{\{\lambda<\frac 18\}}$. For the first part, using Lemma \ref{boundforetaquantity}, we have
\begin{multline*}
    \int_{\mathbb{S}^2}\int_{\frac {|y-\eta|^2}{2(1-y\cdot\eta)}}^{1} \lambda^{-3-\gamma} Q\left(\frac{y-(1-\lambda)\eta}{\lambda}\right) \left(\lambda^2-|y-(1-\lambda)\eta|^2\right)_+^\beta \left|y-\eta\right|^\mu \left(1-y\cdot \eta\right)^{-\nu}\mathbbm{1}_{\{\lambda\geq\frac 18\}} d\lambda d\sigma(\eta)\\
    \lesssim \varepsilon^2\int_{\mathbb{S}^2}\int_{\frac {|y-\eta|^2}{2(1-y\cdot\eta)}}^{1} \lambda^{-3-\gamma} \left(1-\frac{|y-(1-\lambda)\eta|^2}{\lambda^2}\right)^\alpha \left(\lambda^2-|y-(1-\lambda)\eta|^2\right)_+^\beta \left|y-\eta\right|^\mu \left(1-y\cdot \eta\right)^{-\nu} \mathbbm{1}_{\{\lambda\geq\frac 18\}} d\lambda d\sigma(\eta)\\
    \lesssim \varepsilon^2\int_{\mathbb{S}^2}\int_{\frac {|y-\eta|^2}{2(1-y\cdot\eta)}}^{1} \lambda^{-3-\gamma-2\alpha} \left(\lambda^2-|y-(1-\lambda)\eta|^2\right)_+^{\alpha+\beta} \left(1-y\cdot \eta\right)^{-\nu+\frac \mu 2}  \mathbbm{1}_{\{\lambda\geq\frac 18\}} d\lambda d\sigma(\eta)\\
    \lesssim \varepsilon^2\int_{\mathbb{S}^2}\int_{\frac {|y-\eta|^2}{2(1-y\cdot\eta)}}^{1} \left(2\left(1-y\cdot \eta\right)\lambda-\left|y-\eta\right|^2\right)_+^{\alpha+\beta} \left(1-y\cdot \eta\right)^{-\nu+\frac \mu 2} d\lambda d\sigma(\eta)\\
    \lesssim \varepsilon^2\int_{\mathbb{S}^2} \left(1-y\cdot \eta\right)^{\min\{\alpha+\beta,0\}} \left(1-y\cdot \eta\right)^{-\nu+\frac \mu 2} d\sigma(\eta) \lesssim \varepsilon^2\left(1-|y|\right)^{\min\{\alpha+\beta,0\}-\nu+\frac \mu 2}.
\end{multline*}
We now estimate the other part, which requires extra work. For $|y|=\frac rt \leq \frac 89$, a similar estimate as above gives a uniform bound (instead of $\frac 18$, we have $\frac{|y-\eta|^2}{1-y\cdot \eta}$ as the lower bound for the integral over $\lambda$ inside), since $|y|$ is bounded away from $1$. Therefore, we focus on the case when $|y|=\frac rt>\frac 89$. We want to estimate
\begin{multline*}
    \int_{\mathbb{S}^2}\int_{\frac {|y-\eta|^2}{2(1-y\cdot\eta)}}^{1} (1-\lambda)\lambda^{-3-\gamma} Q\left(\frac{y-(1-\lambda)\eta}{\lambda}\right) (\lambda^2-|y-(1-\lambda)\eta|^2)^\beta |y-\eta|^\mu (1-y\cdot \eta)^{-\nu} \mathbbm{1}_{\{\lambda<\frac 18\}}d\lambda d\sigma(\eta) \\
    \lesssim \varepsilon^2\int_{\mathbb{S}^2}\int_{\frac {|y-\eta|^2}{2(1-y\cdot\eta)}}^{1} \lambda^{-3-\gamma} \left(1-\frac{|y-(1-\lambda)\eta|^2}{\lambda^2}\right)^\alpha \left(\lambda^2-|y-(1-\lambda)\eta|^2\right)^\beta \left|y-\eta\right|^\mu \left(1-y\cdot \eta\right)^{-\nu} \mathbbm{1}_{\{\lambda<\frac 18\}} d\lambda d\sigma(\eta)
\end{multline*}
The integrand is now nonnegative, so we change the order of integration here. In view of the lemma above, the lower limit of $\lambda$ should be $\frac 12(1-|y|)$, and $\eta$ need to satisfy $|y-(1-\lambda)\eta|^2<\lambda^2$. Under the assumption that $x/|x|=(1,0,0)$, this is equivalent to $\eta_1>\frac{|y|^2+1-2\lambda}{2|y|(1-\lambda)}\geq \frac{27}{28}$, with $\eta=(\eta_1,\eta_2,\eta_3)$. Also in this case, we have $d\sigma(\eta)= 2\pi \eta_1 |d\eta_1|$ since $S=2\pi (1-\eta_1)$ where $S$ is the area of the part of sphere with $x$ component greater than $\eta_1$. In addition, we have $|y-\eta|=\sqrt{|y|^2-2|y|\eta_1+1}=(2|y|)^\frac 12\sqrt{\frac{1+|y|^2}{2|y|}-\eta_1}$. Note that in this case we also have $1-\lambda\geq \frac 98$. Therefore
\begin{equation*}
\begin{split}
    &\int_{\mathbb{S}^2}\int_{\frac {|y-\eta|^2}{2(1-y\cdot\eta)}}^{1} \lambda^{-3-\gamma} \left(1-\frac{|y-(1-\lambda)\eta|^2}{\lambda^2}\right)^\alpha \left(\lambda^2-|y-(1-\lambda)\eta|^2\right)^\beta \left|y-\eta\right|^\mu \left(1-y\cdot \eta\right)^{-\nu} \mathbbm{1}_{\{\lambda<\frac 18\}} d\lambda d\sigma(\eta)\\
    &\lesssim \left(1-|y|\right)^{-\nu}\int_{\frac 12(1-|y|)}^{\frac 18} (1-\lambda)\lambda^{-3-\gamma} \int_{\frac{|y|^2+1-2\lambda}{2|y|(1-\lambda)}}^1 \left(1-\frac{|y|^2-2|y|(1-\lambda)\eta_1+(1-\lambda)^2}{\lambda^2}\right)^{\alpha}\\
    &\quad \quad\quad\cdot \left(2|y|(1-\lambda)\eta_1-(1+|y|^2-2\lambda)\right)^\beta \left(2|y|\right)^\frac \mu 2 \left(\frac{1+|y|^2}{2|y|}-\eta_1\right)^\frac \mu 2 d\eta_1 d\lambda\\
    &\lesssim \left(1-|y|\right)^{-\nu}\int_{\frac 12(1-|y|)}^{\frac 18} (1-\lambda)\lambda^{-3-\gamma-2\alpha} \int_{\frac{|y|^2+1-2\lambda}{2|y|(1-\lambda)}}^1 \left(2|y|(1-\lambda)\eta_1-|y|^2-1+2\lambda\right)^{\alpha+\beta}\left(\frac{1+|y|^2}{2|y|}-\eta_1\right)^\frac \mu 2 d\eta_1 d\lambda\\
    &\lesssim \left(1-|y|\right)^{-\nu}\int_{\frac 12(1-|y|)}^{\frac 18} (1-\lambda)\lambda^{-3-\gamma} \lambda^{-2\alpha}(2|y|(1-\lambda))^{\alpha+\beta}
    \displaystyle\int_{\frac{|y|^2+1-2\lambda}{2|y|(1-\lambda)}}^1
    \left(\eta_1-\frac{|y|^2+1-2\lambda}{2|y|(1-\lambda)}\right)^{\alpha+\beta}\\
    &\quad \quad\quad\cdot \left(\frac{1+|y|^2}{2|y|}-\frac{|y|^2+1-2\lambda}{2|y|(1-\lambda)}\right)^\mu d\eta_1  d\lambda\\
    &\lesssim \left(1-|y|\right)^{-\nu}\int_{\frac 12(1-|y|)}^{\frac 18} (1-\lambda)\lambda^{-3-\gamma} \lambda^{-2\alpha}(2|y|(1-\lambda))^{-1} \left(2|y|(1-\lambda)-|y|^2-1+2\lambda\right)^{\alpha+\beta+1} \left(\frac{(1-|y|^2)\lambda}{2|y|(1-\lambda)}\right)^{\frac \mu 2} d\lambda\\
    &\lesssim \left(1-|y|\right)^{-\nu}\int_{\frac 12(1-|y|)}^{\frac 18} \lambda^{-3-\gamma-2\alpha} \left(2(1-|y|)\lambda-(1-|y|)^2\right)^{\alpha+\beta+1} (1-|y|)^{\frac \mu 2}\lambda^{\frac \mu 2} d\lambda\\
    &\lesssim \left(1-|y|\right)^{-\nu+\alpha+\beta+\frac\mu 2+1} \int_{\frac 12(1-|y|)}^{\frac 18} \lambda^{-3-\gamma-2\alpha+\frac \mu 2} \left(\lambda-\frac 12(1-|y|)\right)^{\alpha+\beta+1}d\lambda\\
    &\lesssim \left(1-|y|\right)^{-\nu+\alpha+\beta+\frac\mu 2+1} \int_{\frac 12\left(1-|y|\right)}^{\frac 18} \lambda^{-2-\gamma-\alpha+\beta+\frac \mu 2} d\lambda \lesssim \left(1-|y|\right)^{-\gamma-\nu+2\beta+\mu},
    \end{split}
\end{equation*}
where we used $|y|=\frac rt>\frac 89$.
\end{proof}
\begin{remark}\label{remarklambdalemma}
In view of the estimates, the bound actually holds once $\beta+\frac\mu 2<\alpha+\gamma+1$ and $\alpha+\beta+1>0$. {The above lemma provides a way to estimate the decay of solutions to inhomogeneous wave equations, in view of the representation formula appeared in \eqref{lambdaformula}.}
\end{remark}

\begin{cor}
The function $\widetilde U(y)$ is uniformly bounded for $|y|<1$.
\end{cor}
We also define $U(y)=(1-|y|^2)^\frac 12 \widetilde U(y)$ so that $\frac {U(y)}\rho=\frac {\widetilde U(y)}t$.

Now we know that when $t-r>4$, we have $u_1=\frac{\widetilde U(y)}\rho$. This determines the radiation field $F_1(q,\omega)$ for $q<-4$:
\[F_1(q,\omega)=\lim_{q=\text{const},\ t\rightarrow \infty} (t+q)\frac{\widetilde U(\frac{t+q}{t}\omega)}{t}=\lim_{|y|\rightarrow 1}\widetilde U(|y|\omega).\]
This also implies the existence of the limit on the right hand side, as we have already shown the existence of the radiation field in the beginning of this section. Therefore, we have $F_1(q,\omega)=A(\omega)$ for some function $A$ when $q<-4$.




\subsubsection{Source with oscillation}\label{subsectionoscillation}
We also need to deal with terms with nontrivial phases. We study the equation
\[-\Box u_2^+=t^{-3} P(x/t) e^{2i(\rho+\varphi(\rho,y))}=t^{-3} P(x/t) e^{2i(\sqrt{t^2-|x|^2}+\varphi(\sqrt{t^2-|x|^2},\frac xt))}\] with vanishing initial data on $\{t=2\}$. Using the representation formula again, we have
\begin{multline}
    u_2^+(t,x)=\frac{1}{4\pi} \int_2^t \frac{1}{t-s} \int_{|\bar{z}|=t-s} s^{-3} e^{2i\sqrt{s^2-|x-\bar z|^2}+2i\varphi (\sqrt{s^2-|x-\bar z|^2},\frac{x-\bar z}{s})} P\left (\frac{x-\bar z}{s}\right ) \, d\sigma(\bar z) ds\\
    =\frac 1{4\pi}\int_{\frac 2t}^1 \int_{\mathbb{S}^2} (1-\lambda)t^{-1}\lambda^{-3} e^{2it\sqrt{\lambda^2-|\frac xt-(1-\lambda)\eta|^2}+2i\varphi(t\sqrt{\lambda^2-|\frac xt-(1-\lambda)\eta|^2},\frac{\frac xt-(1-\lambda)\eta}{\lambda})} P\left (\frac{\frac xt-(1-\lambda)\eta}{\lambda}\right )\, d\sigma(\eta) d\lambda.
\end{multline}
We have, in fact, shown the absolute convergence of this integral for $t-r>4$ above. Therefore, we now change the order of integration. By the same analysis of the support, we have (recall $y=x/t$)
\begin{multline}
    u_2^+(t,x)=\frac 1{4\pi t} \int_{\mathbb{S}^2}\int_{\frac{|y-\eta|^2}{2(1-y\cdot\eta)}}^1 (1-\lambda)\lambda^{-3} \frac{1}{2it}\frac{\sqrt{\lambda^2-|y-(1-\lambda)\eta|^2}}{1-y\cdot \eta}
    \pa_\lambda(e^{2it\sqrt{\lambda^2-|y-(1-\lambda)\eta|^2}})\\
    \cdot e^{2i\varphi(t\sqrt{\lambda^2-|y-(1-\lambda)\eta|^2},\frac{y-(1-\lambda)\eta}{\lambda})}
    P\left (\frac{y-(1-\lambda)\eta}{\lambda}\right )\,  d\lambda d\sigma(\eta)
\end{multline}
where we used $\pa_\lambda (\lambda^2-|y-(1-\lambda)\eta|^2)=2-2y\cdot \eta$. Now integrating by parts, we have
\begin{multline}
    u_2^+(t,x)=-\frac{1}{4\pi} \frac{1}{2it^2}\int_{\mathbb{S}^2}\int_{\frac{|y-\eta|^2}{2(1-y\cdot\eta)}}^1 e^{2it\sqrt{\lambda^2-|y-(1-\lambda)\eta|^2}}\\
    \cdot \pa_\lambda \left((1-\lambda)\lambda^{-3} \frac{\sqrt{\lambda^2-|y-(1-\lambda)\eta|^2}}{1-y\cdot \eta} e^{2i\varphi(t\sqrt{\lambda^2-|y-(1-\lambda)\eta|^2},\frac{y-(1-\lambda)\eta}{\lambda})} P\left (\frac{y-(1-\lambda)\eta}{\lambda}\right )\right)\, d\sigma(\eta) d\lambda\\
    +\frac{1}{4\pi}\frac{1}{2it}\int_{\mathbb{S}^2} \lambda^{-3} (1-\lambda)e^{2it\sqrt{\lambda^2-|y-(1-\lambda)\eta|^2}+2i\varphi(t\sqrt{\lambda^2-|y-(1-\lambda)\eta|^2},\frac{y-(1-\lambda)\eta}{\lambda})}\\
    \cdot\frac{\sqrt{\lambda^2-|y-(1-\lambda)\eta|^2}}{1-y\cdot \eta} P\left (\frac{y-(1-\lambda)\eta}{\lambda}\right ) \Bigg |_{\frac{|y-\eta|^2}{2(1-y\cdot\eta)}}^1 d\sigma(\eta).
\end{multline}
We see that the second term is zero. We now focus on the first term, and note that when $\lambda>\frac{|y-\eta|^2}{2(1-y\cdot\eta)}$, everything is differentiable. We have
\begin{equation}\label{ABCD}
\begin{split}
    &\pa_\lambda \left((1-\lambda)\lambda^{-3} \frac{\sqrt{\lambda^2-|y-(1-\lambda)\eta|^2}}{1-y\cdot \eta} e^{2i\varphi(t\sqrt{\lambda^2-|y-(1-\lambda)\eta|^2},\frac{y-(1-\lambda)\eta}{\lambda})} P\left (\frac{y-(1-\lambda)\eta}{\lambda}\right )\right)\\
    =&(-3\lambda^{-4}(1-\lambda)-\lambda^{-3})\frac{\sqrt{\lambda^2-|y-(1-\lambda)\eta|^2}}{1-y\cdot \eta} e^{2i\varphi(t\sqrt{\lambda^2-|y-(1-\lambda)\eta|^2},\frac{y-(1-\lambda)\eta}{\lambda})}P\left (\frac{y-(1-\lambda)\eta}{\lambda}\right )\\
    &+(1-\lambda)\lambda^{-3} \frac{1}{\sqrt{\lambda^2-|y-(1-\lambda)\eta|^2}} e^{2i\varphi(t\sqrt{\lambda^2-|y-(1-\lambda)\eta|^2},\frac{y-(1-\lambda)\eta}{\lambda})}P\left (\frac{y-(1-\lambda)\eta}{\lambda}\right )\\
    &+(1-\lambda)\lambda^{-3}\frac{\sqrt{\lambda^2-|y-(1-\lambda)\eta|^2}}{1-y\cdot \eta}e^{2i\varphi(t\sqrt{\lambda^2-|y-(1-\lambda)\eta|^2},\frac{y-(1-\lambda)\eta}{\lambda})} P\left (\frac{y-(1-\lambda)\eta}{\lambda}\right )\\
    &\quad\quad\quad\Big(\frac{t-x\cdot \eta}{\sqrt{\lambda^2-|y-(1-\lambda)\eta|^2}}\pa_\rho\varphi(\textstyle t\sqrt{\lambda^2-|y-(1-\lambda)\eta|^2},\frac{y-(1-\lambda)\eta}{\lambda})\\
    &\quad\quad\quad\quad\quad\quad\quad+\nabla_y \varphi(\textstyle t\sqrt{\lambda^2-|y-(1-\lambda)\eta|^2},\frac{y-(1-\lambda)\eta}{\lambda})\cdot \frac{-y+\eta}{\lambda^2}\Big) \\
    &+(1-\lambda)\lambda^{-3} \frac{\sqrt{\lambda^2-|y-(1-\lambda)\eta|^2}}{1-y\cdot \eta} \nabla P\left (\frac{y-(1-\lambda)\eta}{\lambda}\right )\cdot(-\frac{y -\eta}{\lambda^2})\\
    =&:A+B+C_1+C_2+D,
\end{split}
\end{equation}
where $\nabla_y\varphi$ means the gradient of $\varphi(\rho,y)$ with respect to the second component. (In particular it is not differentiating $y$ in the equation above)

In our case, we have $\varphi(\rho,y)=-\frac 12\int u(\rho,y) d\rho$, we have $\nabla_y \varphi=\int \nabla_y u(\rho,y) d\rho$, so $|\pa_\rho\varphi|=|u|\lesssim \varepsilon t^{-1}$, and $|\nabla_y \varphi|\lesssim \varepsilon\rho^\delta (1-|y|^2)^{-\frac 12}$ using \eqref{phi} and \eqref{dyonuandphi}. We also recall that $|P(y)|\lesssim \varepsilon^2 (1-|y|^2)^{2-2\delta}$, and $|\nabla P(y)|\lesssim \varepsilon^2 (1-|y|^2)^{1-2\delta}$.
We have
\begin{multline}
    |A|\lesssim \varepsilon^2 \left(3\lambda^{-4}(1-\lambda)+\lambda^{-3}\right)\frac{\sqrt{\lambda^2-|y-(1-\lambda)\eta|^2}}{1-y\cdot \eta}\lambda^{-1+8\delta}\left(\lambda^2-|y-(1-\lambda)\eta|^2\right)^{2-2\delta}\\
    \lesssim \varepsilon^2 \lambda^{-5+8\delta}\frac 1{1-y\cdot \eta} \left(\lambda^2-|{y}-(1-\lambda)\eta|^2\right)^{\frac 52-2\delta},
\end{multline}

\begin{multline}
    |B|\lesssim \varepsilon^2\lambda^{-3} \frac{1}{\sqrt{\lambda^2-|y-(1-\lambda)\eta|^2}}\lambda^{-1+8\delta}\left(\lambda^2-|y-(1-\lambda)\eta|^2\right)^{2-2\delta}\\
    \lesssim \varepsilon^2\lambda^{-4+8\delta} \left(\lambda^2-|y-(1-\lambda)\eta|^2\right)^{\frac 32-2\delta},
\end{multline}
\begin{multline}
    |D|\lesssim \varepsilon^2 \lambda^{-3} \frac{\sqrt{\lambda^2-|{y}-(1-\lambda)\eta|^2}}{1-y\cdot \eta}\lambda^{1+8\delta}(\lambda^2-|{y}-(1-\lambda)\eta|^2)^{1-2\delta}\frac{|{y} -\eta|}{\lambda^2}\\
    \lesssim \varepsilon^2 \lambda^{-4+8\delta} \frac{|y-\eta|}{1-y\cdot\eta} \left(\lambda^2-|y-(1-\lambda)\eta|^2\right)^{\frac 32-2\delta},
\end{multline}
\begin{multline}
    |C_1|\lesssim \varepsilon^3 \lambda^{-3} t \lambda^{-1+8\delta}\left(\lambda^2-|y-(1-\lambda)\eta|^2\right)^{2-2\delta} \left(t\sqrt{\lambda^2-|y-(1-\lambda)\eta|^2}\right)^{-1} \lambda^{-1}\left(\lambda^2-|y-(1-\lambda)\eta|^2\right)^{\frac 12}\\
    \lesssim \varepsilon^3 (1-\lambda)\lambda^{-5+8\delta} \left(\lambda^2-|y-(1-\lambda)\eta|^2\right)^{2-2\delta},
\end{multline}
\begin{multline}
    |C_2|\lesssim \varepsilon^3 \lambda^{-3}\frac{\sqrt{\lambda^2-|y-(1-\lambda)\eta|^2}}{1-y\cdot \eta} \lambda^{-1+8\delta}\left(\lambda^2-|y-(1-\lambda)\eta|^2\right)^{2-2\delta} \left(t\sqrt{\lambda^2-|y-(1-\lambda)\eta|^2}\right)^\delta \\
    \cdot\lambda \left(\lambda^2-|{y}-(1-\lambda)\eta|^2\right)^{-\frac 12}\frac{|y -\eta|}{\lambda^2}\\
    \lesssim \varepsilon^3 \lambda^{-5+8\delta} t^\delta \left(\lambda^2-|y-(1-\lambda)\eta|^2\right)^{2-\frac 32 \delta}.
\end{multline}
Now we can use Lemma \ref{lemmaintegral} to get \[\int_{\mathbb{S}^2}\int_{\frac{|\frac xt-\eta|^2}{2(1-\frac xt\cdot\eta)}}^1 |A|+|B|+|C_1|+|C_2|+|D| \, d\lambda d\sigma(\eta)\lesssim \varepsilon^2 \left(1-\frac rt\right)^{-1}+\varepsilon^3 \left(1-\frac rt\right)^{-1+\delta}t^\delta.\] This shows that $|u_2^+|\lesssim \varepsilon^2 t^{-1} (t-r)^{-1+\delta}$ when $t-r>4$. When $t-r\leq 4$, we simply use the bound $|u_2^+|\lesssim \varepsilon t^{-1}$. The estimate of $u_2^-$ accounting for the other oscillating source term is the same, so we obtain the estimate of $u_2=u_2^++u_2^-$. This also indicates that the corresponding radiation field $F_2(q,\omega)$, whose existence follows similarly as for $F_1(q,\omega)$, decays at the rate $(1+q_-)^{-1+\delta}$.

\subsubsection{Remainder terms}
We consider $-\Box u_3=R$ with vanishing initial data. Recall we have $|R|\lesssim \varepsilon^2 t^{-4+\delta}$. Then a quick use of Lemma \ref{lemmaintegral} yields $|u_3|\lesssim \varepsilon^2 t^{-2+\delta} (1-\frac rt)^{-1+\delta}=\varepsilon^2 t^{-1} (t-r)^{-1+\delta}$, so it affects little in the interior. Again, for $t-r<4$ we use the bound $|u_3|\lesssim \varepsilon^2 t^{-1}$. The radiation field $F_3(q,\omega)$ then satisfies the decay $|F_3(q,\omega)|\lesssim \varepsilon^2 (1+q_-)^{-1+\delta}$.

\subsubsection{Homogeneous equation}
We finally consider the equation $-\Box u_4=0$ with the initial data set of $u$. In this case we have $|u_4|\lesssim \varepsilon t^{-1}(1+|t-r|)^{-1}$, so again no effect in the interior. The radiation field $|F_4(q,\omega)|\lesssim \varepsilon (1+q_-)^{-1}$ also enjoys good decay.


\subsection{Asymptotics for Klein-Gordon field}We have proved the first part of the theorem. We now have $|u-\frac {U(y)}\rho|\leq |u_2|+|u_3|+|u_4|\lesssim \varepsilon t^{-1}(t-r)^{-1+\delta}\lesssim \varepsilon \rho^{-2+\delta}$ when $t-r>4$. We have also shown that $|\Phi_\pm|\lesssim \varepsilon (1-|y|^2)^{\frac 74-\delta}$.

Now we rewrite the equation \eqref{drhoPhip} as
$$\pa_\rho \Phi_+=-\frac i2 \frac{U(y)}\rho\Phi_+ + \frac i2 e^{- 2i\rho} u\Phi_- +e^{- i\rho}\rho^{-1/2}(\triangle_y \phi+\frac 34\phi)+(u-\frac{U(y)}\rho)\Phi_+.$$
Then
$$\pa_\rho (e^{\frac i2 U(y) \ln \rho} \Phi_+ +\frac 14 e^{-2i\rho+\frac i2 U(y)\ln \rho}u\Phi_-)=\frac 14 e^{-2i\rho} \pa_\rho(e^{\frac i2 U(y)\ln \rho} u\Phi_-)+e^{- i\rho}\rho^{-1/2}(\triangle_y \phi+\frac 34\phi)+(u-\frac{U(y)}\rho) \Phi_+.$$
We denote the right hand side by $\widetilde h(\rho,y)$. Then $|\widetilde h(\rho,y)|\lesssim \varepsilon \rho^{-2+\delta}(1-|y|^2)^{\frac 74-\delta}$ using the decay we have\footnote{Here we use the decay $|\triangle_y \phi|\lesssim \varepsilon \rho^{-\frac 32+\delta}(1-|y|^2)^{\frac 74-\delta}$, for which we only have it with $\frac 47$ replaced by $\frac 45$ at this stage, but this can also be shown by dealing with commutators like the existence proof.}.

Again, we want to integrate along the hyperboloidal ray, but this time we start the integration at $t-r=4$, also written as $\rho=\widetilde\rho(y)$. Along this we have $|\Phi_\pm|\lesssim \varepsilon(1-|y|^2)^{\frac 74-\delta}$.

Then
\begin{multline}
    e^{\frac i2 U(y) \ln \rho} \Phi_+ +\frac 14 e^{-2i\rho+\frac i2 U(y)\ln \rho}u\Phi_- -(e^{\frac i2 U(y) \ln \rho} \Phi_+ +\frac 14 e^{-2i\rho+\frac i2 U(y)\ln \rho}u\Phi_-)|_{\rho=\widetilde\rho(y)}\\
    =\int_{\widetilde\rho(y)}^\infty \widetilde h(\rho,y) d\rho-\int_\rho^\infty \widetilde h(\rho,y) d\rho,
\end{multline}
so we have
\begin{equation}
    e^{\frac i2 U(y) \ln \rho} \Phi_+=\widetilde b_+(y)+\widetilde h_+(\rho,y)
\end{equation}
where $|\widetilde b_+(y)|\lesssim \varepsilon (1-|y|^2)^{\frac 94-\frac 32\delta}$, and $|\widetilde h_+(\rho,y)|\lesssim \varepsilon\rho^{-1+\delta} (1-|y|^2)^{\frac 74-\delta}$. Similarly for $\Phi_-$. Therefore we get
$$\phi=\rho^{-\frac 32}\left(e^{i\rho-\frac i2 U(y) \ln \rho}\, \widetilde a_+(y)+e^{-i\rho+\frac i2 U(y)\ln \rho}\, \widetilde a_-(y)\right)+\widetilde R,$$
where $|\widetilde R|\lesssim \varepsilon^2\rho^{-\frac 52+\delta} (1-|y|^2)^{\frac 54-\frac \delta 2}\lesssim \varepsilon^2 t^{-\frac 52+\delta}$. This finishes the case when $t-r>4$. For $t-r\leq 4$ we have $$|\phi|\lesssim \varepsilon \rho^{-\frac 32}(\rho/t)^{\frac 72-2\delta}\lesssim \varepsilon\rho^{2-2\delta} t^{-\frac 72+2\delta}\lesssim \varepsilon t^{-\frac 52+\delta}.$$

\section{Quasilinear system}
We now discuss the coupled system \eqref{quasilinearsystem}
$$-\Box u=(\pa_t \phi)^2+\phi^2,\ \ -\Box\phi+\phi=H^{\alpha\beta} u\pa_\alpha \pa_\beta \phi.$$
In \cite{LMmodel} a complete proof of the small data global existence is given (see also \cite{LMextmodel} for non-compactly supported data). The new difficulty of this model is the worse behavior near the light cone: the equation can be rewritten as
\begin{equation}
    -\Box \phi+\phi=H^{\rho\rho} u\pa_\rho^2\phi+2H^{\rho y_i} u\pa_\rho\pa_{y_i}\phi+H^{y_i y_j}u\pa_{y_i} \pa_{y_j}\phi+R_1
\end{equation}
where $R_1$ is the difference when changing the coordinates, i.e.
\begin{equation}
    R_1=H^{\alpha\beta} u \pa_\alpha\pa_\beta \phi-H^{\rho\rho} u\pa_\rho^2\phi-H^{y_i y_j}u\pa_{y_i} \pa_{y_j}\phi.
\end{equation}
Recall that $d\rho=\frac t\rho dt-\frac{x^i}\rho dx^i$, $dy^i=-\frac{x_i}{t^2} dt+\frac 1t dx^i$. Therefore, we have the estimate $|H^{\rho\rho}|\lesssim \frac 1{1-|y|^2}$, $|H^{\rho y_i}|\lesssim \rho^{-1}$, $|H^{y_i y_j}|\lesssim t^{-2}$. Compared with the semilinear model, we have an $H^{\rho\rho}$ factor in the term which is expected to be on the phase correction, and this factor is unfavorable near the light cone.

We can further write the equation as
\begin{equation*}
    (1-H^{\rho\rho}u)\pa_\rho^2(\rho^\frac 32\phi)+\rho^\frac 32 \phi=-\rho^{-\frac 12}(\triangle_y \phi+\frac 34\phi)+2\rho^\frac 32 H^{\rho y_i}u\pa_\rho \pa_{y_i}\phi+\rho^\frac 32 H^{y_i y_j} u\pa_{y_i} \pa_{y_j} \phi+\rho^\frac 32 (R_1+R_2)
\end{equation*}
where $R_2=H^{\rho\rho} u\pa_\rho^2 \phi-\rho^{-\frac 32} H^{\rho\rho} u \pa_\rho^2 (\rho^\frac 32 \phi)$.

Now we consider the change of variable $\rho^*=\int (1-H^{\rho\rho}u)^{-\frac 12} d\rho$ with $\rho^*=\rho$ on $\{t-r=1\}$, so that $\pa_{\rho^*}=(1-H^{\rho\rho} u)^\frac 12 \pa_\rho$. One can also show that $\rho^*\sim \rho$. Then we get
\begin{equation}\label{quasilinearODE}
    \pa_{\rho^*}^2 \Phi+\Phi=\frac 12 \pa_\rho (H^{\rho\rho} u) \pa_\rho \Phi-\rho^{-\frac 12}(\triangle_y \phi+\frac 34 \phi)+2\rho^\frac 32 H^{\rho y_i}u\pa_\rho \pa_{y_i}\phi+\rho^\frac 32 H^{y_i y_j} u\pa_{y_i} \pa_{y_j} \phi+\rho^\frac 32 (R_1+R_2).
\end{equation}

We have the estimate
\begin{equation*}
    |R_1|\lesssim |H^\ab u (\pa_\alpha\pa_\beta \rho)\pa_\rho \phi|+|H^\ab u (\pa_\alpha\pa_\beta y_i)\pa_{y_i} \phi|,
\end{equation*}
where $\alpha,\beta$ represent the rectangular coordinates.
Note that $|\pa^2 \rho| \lesssim \frac{t^2}{\rho^3}$, so this causes much worse behavior near the light cone compared with the semilinear case.
We also have
\begin{equation*}
    |R_2|\lesssim \rho^{-\frac 32} |H^{\rho\rho} u| (\rho^{-1}|\pa_\rho \phi|+\rho^{-2}|\phi|).
\end{equation*}

Suppose we have the weak decay estimates
\begin{equation}
    |\pa^I L^J \phi|\lesssim \varepsilon \rho^{-\frac 32+\delta} (1-|y|^2)^{a},\quad |\pa L^J u|\lesssim \varepsilon t^{-\frac 12}\rho^{-1+\delta},\quad |I|+|J|\leq N_1,
\end{equation}
for some $a\geq 0$.
Then using Lemma \ref{lemmawaveexistencedecay} we get $|L^J u|\lesssim \varepsilon t^{-1}(t-r)^\delta$.
We also have 
\begin{equation}\label{drhophiimproved}
    |\pa_\rho L^{J}\phi|\lesssim \frac\rho t|\pa_t L^J\phi|+\rho^{-1}|LL^J\phi|\lesssim \varepsilon \rho^{-\frac 32+\delta}(1-|y|^2)^{a+\frac 12}+\varepsilon \rho^{-\frac 52+\delta} (1-|y|^2)^{a}\lesssim \varepsilon\rho^{-\frac 32+\delta} (1-|y|^2)^{a+\frac 12}
\end{equation}
for $|J|\leq N_1-1$, using $\rho\geq (1-|y|^2)^{-\frac 12}$ in the relevant region $t-r\geq 1$. Therefore, also using $|\pa_{y_i}\phi|\lesssim (1-|y|^2)^{-1}|\Omega\phi|$, we have
\begin{equation}
    |\rho^\frac 32 R_1|\lesssim \varepsilon^2 t^{-1} t^2 \rho^{-3+\delta} (1-|y|^2)^{a+\frac 12}+\varepsilon^2 t^{-3}(1-|y|^2)^{-1}\rho^\delta (1-|y|^2)^a\lesssim \varepsilon^2 \rho^{-2+\delta} (1-|y|^2)^{a},
\end{equation}
\begin{equation}
    |\rho^\frac 32 R_2|\lesssim \varepsilon^2 \rho^{-2+2\delta} (1-|y|^2)^a.
\end{equation}

We also need to estimate other terms on the right hand side of \eqref{quasilinearODE}. We have $|\pa_\rho u|\lesssim \frac\rho t|\pa_t u|+\varepsilon\rho^{-1}|Lu|\lesssim \varepsilon t^{-\frac 32}\rho^\delta+\varepsilon \rho^{-1}t^{-1}\rho^\delta\lesssim \varepsilon\rho^{-\frac 32+\delta}(1-|y|^2)^{\frac 34}$. Therefore we have
\begin{equation}
    |\pa_\rho (H^{\rho\rho} u)\pa_\rho \Phi|\lesssim \varepsilon\rho^{-\frac 32+\delta}(1-|y|^2)^{-\frac 14}(\rho^\frac 32 |\pa_\rho \phi|+\rho^\frac 12|\phi|)\lesssim \varepsilon^2\rho^{-\frac 32+2\delta}(1-|y|^2)^{a+\frac 14}.
\end{equation}
We also have the estimate
\begin{equation}
    |\rho^\frac 32 H^{\rho y_i}u \pa_\rho\pa_{y_i}\phi|\lesssim \varepsilon \rho^\frac 32 \rho^{-1} t^{-1} \frac{1}{1-|y|^2}|\pa_\rho L\phi|\lesssim \varepsilon^2 \rho^{-2+\delta}(1-|y|^2)^a
\end{equation}
using \eqref{drhophiimproved}, and
\begin{equation}
    |\rho^\frac 32 H^{y_i y_j} u \pa_{y_i} \pa_{y_j} \phi|\lesssim \varepsilon \rho^\frac 32 t^{-3} (1-|y|^2)^{-2} \sum_{|J|\leq 2}|L^J \phi|\lesssim \varepsilon^2 \rho^{-3+\delta} (1-|y|^2)^{a-\frac 12}.
\end{equation}

We want to apply a similar integration argument to \eqref{quasilinearODE} as in previous sections. We will get an integral like
\begin{equation}
    \int_{\rho(y)}^\rho s^{-\frac 32+2\delta} (1-|y|^2)^{a+\frac 14}+s^{-2+2\delta}(1-|y|^2)^a+s^{-3+\delta}(1-|y|^2)^{a-\frac 12} ds
\end{equation}
when controlling $\rho^\frac 32 \phi$.
So in view of the bound $\rho(y)\sim (1-|y|^2)^{-\frac 12}$, we derive $|\phi|\lesssim \varepsilon \rho^{-\frac 32} (1-|y|^2)^{a+\frac 12-\delta}$. Hence we improve the bounds in terms of the decay in $(1-|y|^2)$.

The same idea also works for higher orders, as we have
$$-\Box\pa^I L^J \phi+\pa^I L^J \phi=H^{\alpha\beta} u \pa_\alpha \pa_\beta \pa^I L^J \phi-[H^{\alpha\beta} u\pa_\alpha\pa_\beta,\pa^I L^J]\phi,$$
and we can apply the same argument to $\pa^I L^J \phi$. The commutator term produces terms with fewer vector fields, and this allows an induction argument to deal with the small decay loss in $\rho$, which is standard. The commutator term behaves well in the decay in the homogeneous variable $y$ when $|y|\rightarrow 1$. Nevertheless, in view of the way we control the derivative in $\rho$, we lose two order derivatives in this process.

Therefore, we can show that $|\pa^I\phi|\lesssim \rho^{-\frac 32}(1-|y|^2)^{a+\frac 12-\frac \delta 2}$, $|\pa^I L^J \phi|\lesssim \varepsilon \rho^{-\frac 32+\delta} (1-|y|^2)^{a+\frac 12-\frac\delta 2}$ for $|I|+|J|\leq N_1-2$. We can then apply this argument iteratively to get even better decay, at the expense of two order derivatives each time. This explains how to deal with the new difficulty, and the proof of global existence follows similarly as the semilinear case.

Now we turn to the asymptotics. In the last step of the integration along (rescaled) hyperboloidal rays, we can instead use the same way as in the semilinear case to obtain the asymptotic expansion
\[\phi=\rho^{-\frac 32}(e^{i\rho^*} a_+(y)+e^{-i\rho^*} a_-(y))+h(\rho,y),\]
where we can derive good bounds, e.g. $|a_\pm (y)|\lesssim \varepsilon (1-|y|^2)^{3}$, $|h(\rho,y)|\lesssim \varepsilon\rho^{-\frac 52+\delta}(1-|y|^2)^{-\frac 52}$.


\subsection*{The wave component}
Again, in \cite{LMmodel} the improvement of the derivative of the wave field was not established, and one can improve this by the conformal energy and integration along the null characteristics to get
\begin{equation}
    |\pa_\rho L^J u|\lesssim \varepsilon \rho^{-2+\delta} (\rho/t)^\frac 12, \quad |\pa L^J u|\lesssim \varepsilon t^{-1}(1+|t-r|)^{-1+2\delta}.
\end{equation}

We have
\begin{equation}
    \phi^2=2\rho^{-3}a_+(y)a_-(y)+\rho^{-3}(e^{2i\rho^*}a_+(y)^2+e^{-2i\rho^*}a_-(y)^2)+R.
\end{equation}
Again, this only holds when $t-r\geq 1$, but we use this expression everywhere in $\{t<r\}$ and put the error in $R$.

Solving $-\Box u_1= \rho^{-3} a_+(y) a_-(y)$ is the same as in the previous section. We have $u_1=U(y)/\rho$ for some $U(y)$ in the region $\{t-r>4\}$. For the remainder we have the estimate $|R|\lesssim \varepsilon^2\rho^{-\frac 32}(1-|y|^2)^{3} \rho^{-\frac 52+\delta}(1-|y|^2)^{\frac 52}\lesssim \varepsilon^2\rho^{-4+\delta} (1-|y|^2)^{\frac {11}2}\lesssim  \varepsilon^2 t^{-4+\delta}$.


For the oscillating part, we again consider the equation $-\Box u_2^+=\rho^{-3} e^{2i\rho^*} a_+^2(y)$. We rewrite this as $-\Box u_2^+=t^{-3} e^{2i\rho^*} P(y)$. Then we have $|P(y)|\lesssim \varepsilon^2 (1-|y|^2)^{\frac 92}$. 

We can apply the same integration by part argument as in the semilinear case, with the phase function $\varphi(\rho,y)=\int (1-H^{\rho\rho}u(\rho,y))^\frac 12 d\rho-\rho$. 
We have $\pa_\rho \varphi=(1-H^{\rho\rho}u)^\frac 12-1$, $\pa_{y_i} \varphi=-\int \frac 12(1-H^{\rho\rho}u)^{-\frac 12}(H^{\rho\rho}\pa_{y_i}u+\pa_{y_i}(H^{\rho\rho})u) d\rho$. Therefore, since $|H^{\rho\rho}u|\leq C\varepsilon$, we have $|\pa_\rho\varphi|\lesssim |H^{\rho\rho} u|\lesssim \varepsilon$, and $|\nabla_y \varphi| \lesssim \varepsilon (1-|y|^2)^{-\frac 32} \rho^{\delta}$ using 
$|\nabla_y H^{\rho\rho}|\lesssim (1-|y|^2)^{-2}$ and $|\nabla_y u|\lesssim \varepsilon\rho^{-1+\delta} (1-|y|^2)^{-\frac 12}$.

The only difference in the argument is that now for $C_1$, $C_2$ in \eqref{ABCD}, where the derivative falls on the phase function, we get more singular behavior near the light cone (which corresponds to negative powers of $(\lambda^2-|\frac xt-(1-\lambda)\eta|^2)$ in the integral\footnote{Recall in Lemma \ref{lemmaintegral}, we require that $\alpha$ cannot be too worse. i.e. close to $-1$. However, once $\alpha$ satisfies this condition, the value of $\alpha$ does not affect the outcome of the estimate.}). However, this is not a problem as we have good decay of the same factor provided by the source term (the Klein-Gordon field). Precisely, we have
\begin{multline}
    C_1+C_2=(1-\lambda)\lambda^{-3}\frac{\sqrt{\lambda^2-|\frac xt-(1-\lambda)\eta|^2}}{1-\frac xt\cdot \eta}e^{2i\varphi(t\sqrt{\lambda^2-|\frac xt-(1-\lambda)\eta|^2},\frac{\frac xt-(1-\lambda)\eta}{\lambda})} P\left (\frac{\frac xt-(1-\lambda)\eta}{\lambda}\right )\\
   \Big(\frac{t-x\cdot \eta}{\sqrt{\lambda^2-|\frac xt-(1-\lambda)\eta|^2}}\pa_\rho\varphi(\textstyle t\sqrt{\lambda^2-|\frac xt-(1-\lambda)\eta|^2},\frac{\frac xt-(1-\lambda)\eta}{\lambda})\\
    +\nabla_y \varphi(\textstyle t\sqrt{\lambda^2-|\frac xt-(1-\lambda)\eta|^2},\frac{\frac xt-(1-\lambda)\eta}{\lambda})\cdot \frac{-\frac xt+\eta}{\lambda^2}\Big).
\end{multline}
Then using the bound for $P(y)$, we have 
\begin{multline}
    |C_1|\lesssim \varepsilon^3(1-\lambda) \lambda^{-3} t (1-\frac{|\frac xt-(1-\lambda)\eta|^2}{\lambda^2})^{\frac 92} (1-\frac{|\frac xt-(1-\lambda)\eta|^2}{\lambda^2})^{-\frac 12} (t\sqrt{\lambda^2-|\frac xt-(1-\lambda)\eta|^2})^{-1}\\
    \lesssim (1-\lambda)\lambda^{-11} (\lambda^2-|\frac xt-(1-\lambda)\eta|^2)^{\frac 72},
\end{multline}
\begin{multline}
    |C_2|\lesssim \varepsilon^3 (1-\lambda)\lambda^{-3}\frac{\sqrt{\lambda^2-|\frac xt-(1-\lambda)\eta|^2}}{1-\frac xt\cdot \eta} (1-\frac{|\frac xt-(1-\lambda)\eta|^2}{\lambda^2})^{\frac 92} (1-\frac{|\frac xt-(1-\lambda)\eta|^2}{\lambda^2})^{-\frac 32}\\
    (t\textstyle\sqrt{\lambda^2-|\frac xt-(1-\lambda)\eta|^2})^\delta \frac{|\frac xt-\eta|}{\lambda^2}\\
    \lesssim (1-\lambda)\lambda^{-11}(\lambda^2-|\frac xt-(1-\lambda)\eta|^2)^{\frac 72-\frac \delta 2}\frac{|\frac xt-\eta|}{1-\frac xt\cdot \eta} t^\delta,
\end{multline}

We then use Lemma \ref{lemmaintegral} to get
\begin{equation*}
\int_{\mathbb{S}^2}\int_{\frac{|\frac xt-\eta|^2}{2(1-\frac xt\cdot\eta)}}^1 |C_1|+|C_2| \, d\lambda d\sigma(\eta)\lesssim \varepsilon^2 \left(1-\frac rt\right)^{-1}+\varepsilon^3 \left(1-\frac rt\right)^{-1+\delta}t^\delta.  
\end{equation*}
Therefore we get $|u_2|\lesssim \varepsilon^2 t^{-2}(1-\frac rt)^{-1}+\varepsilon^2 t^{-2+\delta}(1-\frac rt)^{-1+\delta}\lesssim \varepsilon^2 t^{-1} (t-r)^{-1+\delta}$, again in the region $\{t-r>4\}$.

\subsection*{The asymptotics for the Klein-Gordon field}
We conclude this section with a similar argument to derive the asymptotics of the Klein-Gordon field. We have shown that $u_1=U(y)/\rho$ when $t-r>4$. Therefore $|u-U(y)/\rho|\lesssim |u_2|+|u_3|+|u_4|\lesssim \varepsilon t^{-1}(t-r)^{-1+\delta}\lesssim \varepsilon\rho^{-2+\delta}$ in this region.

We revisit the Klein-Gordon equation
$-\Box \phi-H^\ab u\pa_\alpha \pa_\beta \phi+\phi=0$. We have $-\Box \phi-H^\ab \frac{U(y)}\rho \pa_\alpha \pa_\beta \phi+\phi=H^\ab (u-\frac{U(y)}\rho)\pa_\alpha\pa_\beta\phi$.

In this case we can consider a similar rescaled coordinate $\widetilde\rho^*=\int (1-H^{\rho\rho} \frac{U(y)}\rho)^{-\frac 12} d\rho$. Compared with above we get an additional term $H^\ab (u-\frac{U(y)}\rho) \pa_\alpha \pa_\beta \phi$ that behaves well. Therefore we can get
\[\phi=\rho^{-\frac 32} (e^{i\widetilde\rho^*} \widetilde a_+(y)+e^{-i\widetilde\rho^*} \widetilde a_-(y))+k(\rho,y)\]
in $\{t-r>4\}$, where $|a_\pm (y)|\lesssim \varepsilon(1-|y|^2)^{3}$, and $|k(\rho,y)|\lesssim \varepsilon\rho^{-\frac 52+\delta}(1-|y|^2)^{\frac 52}$. When $t-r<4$, the behavior of $\phi$ is good similar to the semilinear case.

\vspace{1ex}


\section{Scattering from infinity of the semilinear model}\label{scatteringsection}
We consider the scattering from infinity problem of the wave-Klein-Gordon system
\begin{equation}
    -\Box u=(\pa_t \phi)^2+\phi^2, \quad -\Box \phi+\phi=u\phi.
\end{equation}

In view of the asymptotics result, we assign a pair of homogeneous functions $a_\pm (y)$. Recall that then the function $U(y)$ describing the interior asymptotics is determined. We let $u_1$ solve the equation\footnote{Note that here we take the contribution from $(\pa_t\phi)^2$ into account, which we did not consider in the introduction part for simplicity.}
\begin{equation}
    -\Box u_1=2\rho^{-3} (1+(1-|y|^2)^{-1}) a_+(y) a_-(y)
\end{equation}
with vanishing data at $\{t=2\}$. In the forward problem we have already shown that $u_1=U(y)/\rho$ when $t-r>4$.

We need a bit of calculation on $\pa_t\phi_0$. Using $\pa_t=\frac t\rho \pa_\rho-\frac 1t y\cdot\nabla_y$, We have
\begin{multline}
    \pa_t \phi_0=\frac{t}\rho \pa_\rho \phi_0-\frac 1t y\cdot \nabla_y \phi_0
    =\frac{t}\rho \left(-\frac 32\rho^{-\frac 52} \left(e^{i\rho-\frac i2 \int u_1 d\rho} a_+(y)+e^{-i\rho+\int u_1 d\rho} a_-(y)\right)\right)\\
    +\frac t\rho\left(\rho^{-\frac 32}\left(\left(i-\frac i2 u_1 \right)e^{i\rho-\frac i2 \int u_1 d\rho}a_+(y)+\left(-i+\frac i2 u_1\right)e^{-i\rho+\frac i2 \int u_1 d\rho}a_-(y)\right)\right)
    -\frac 1t y\cdot\nabla_y \phi_0\\
    =\rho^{-\frac 32}\left(e^{i\rho-\frac i2 \int u_1 d\rho}(1-|y|^2)^{-\frac 12} ia_+(y)-e^{-i\rho+\frac i2 \int u_1 d\rho} (1-|y|^2)^{-\frac 12} i a_-(y)\right)\\
    -\rho^{-\frac 52} \Big(e^{i\rho-\frac i2 \int u_1 d\rho} (1-|y|^2)^{-\frac 12} (\frac 32+\frac i2 \rho u_1)a_+(y)\\
    +e^{-i\rho+\frac i2 \int u_1 d\rho} (1-|y|^2)^{-\frac 12} (\frac 32-\frac i2 \rho u_1)a_-(y)\Big)-\frac 1t y\cdot\nabla_y \phi_0.
\end{multline}

We then let $u_2$ solve the equation
\begin{equation}\label{equ2}
    -\Box u_2=\rho^{-3} \left(e^{2i\rho-i\int u_1 d\rho}(1-(1-|y|^2)^{-1})a_+^2(y)+e^{-2i\rho+i\int u_1 d\rho}(1-(1-|y|^2)^{-1})a_-^2(y)\right)
\end{equation}
with vanishing initial data at $\{t=2\}$, and $u_3$ such that, again with vanishing data,
\begin{equation}
    -\Box(u_1+u_2+u_3)=\phi_0^2+(\pa_t\phi_0)^2.
\end{equation}
In the asymptotics part we have also derived estimates for the same type of equations as \eqref{equ2}. 

We consider the following approximate solutions
\begin{equation}
    u_0=u_1+u_2+u_3 +\chi({\textstyle\frac{\langle t-r\rangle}r})\textstyle\frac{F_0(t-r,\omega)}r+\chi(\frac{\langle t-r\rangle}r)\frac{F_1(r-t,\omega)}{r^2},
\end{equation}
\begin{equation}
    \phi_0=\rho^{-\frac 32}\left(e^{i\rho-\frac i2 \int u_1 d\rho} a_+(y)+e^{-i\rho+\frac i2 \int u_1 d\rho} a_-(y)\right).
\end{equation}
Note that since $u_1=U(y)/\rho$ when $t-r>4$, the phase correction equals $\pm\frac 12U(y)\ln\rho$ towards timelike infinity.
The free radiation field $F_0$ corresponds to the part of the radiation field other than $u_1$, $u_2$ and $u_3$, and $F_1$ is a second-order approximation determined by $2\pa_q F_1(q,\omega)=\triangle_\omega F_0(q,\omega)$ and $F_1(0,\omega)=0$. We impose the decay condition for $F_0$:
\begin{equation}\label{decayF0}
    |(q\pa_q)^k \pa_\omega^\beta F_0(q,\omega)|\langle q\rangle^{1-\alpha}\leq C\varepsilon,\quad k+|\beta|\leq N+2
\end{equation}
for some $0<\alpha<1/6$. This is the type of decay we get in the forward problem. Also in view of the result in the forward problem, we can assume good decay of $a_\pm (y)$ as $|y|\rightarrow 1$, e.g. $|\nabla^k a_\pm (y)|\leq C\varepsilon (1-|y|^2)^l$ for some $l\geq N$ and for $k\leq N+2$. 

In this section, we prove the following theorem, which implies Theorem \ref{thmscattering}.
\begin{theorem}\label{thmscattering2}
Given the functions $a_\pm (y)$ and $F_0(q,\omega)$ satisfying the conditions, we can define the approximate solutions $(u_0,\phi_0)$ as explained above. Then there exists a solution $(u,\phi)$ of \eqref{semilinearsystem}, such that $u$ and $\phi$ present the desired behavior, i.e. the same behavior as the approximate solutions, at infinity, and the remainder $v=u-u_0$ and $w=\phi-\phi_0$ satisfy the estimate
\begin{equation}
    |v|\leq C\varepsilon \rho^{-2+\alpha+2\delta}(\rho/t)^\frac 12,\quad |w|\leq C\varepsilon \rho^{-\frac 52+\alpha+2\delta} (\rho/t)^\frac 32,\quad \text{when } t-r\geq r^\frac 12,
\end{equation}
\begin{equation}
    |v|\leq C\varepsilon t^{-\frac {13}8+\frac 34 \alpha+\frac 32\delta},\quad |w|\leq C\varepsilon t^{-\frac 74+\frac 34\alpha+\frac 32\delta},\quad \text{when } t-r\leq r^\frac 12.
\end{equation}
\end{theorem}

\subsection{Estimates of the approximate solutions}
\subsubsection{Estimates of $u_1$}
We first derive the estimates of $u_1$. Since
\begin{multline}
    |\pa^I S^K \Omega^J (\rho^{-3} F(y))|\lesssim\sum_{|I_1|+|I_2|\leq |I|} (1-|y|^2)^{-\frac {|I_1|-|I_2|}2}\rho^{-3-|I|} |\pa_{y}^{I_2} \Omega^J F(y)|\\
    \lesssim (1-|y|^2)^{-|I|} \rho^{-3-|I|}\sum_{|J'|\leq |I|+|J|} |\Omega^{J'} F(y)|,
\end{multline}
gives the control of the source term of the wave equation of $\pa^I S^K \Omega^J u_1$, we see that $|\pa^I S^K \Omega^J u_1|\lesssim \varepsilon^2 t^{-1}$ as the right hand side is bounded by $C\varepsilon^2 t^{-3}$.

We want to estimate $\pa^I \Omega^J \int u_1 d\rho$. For example, we have
\begin{equation*}
    \pa_t \Omega^J \int u_1 d\rho=\pa_t \int \Omega^J u_1 d\rho=\frac t\rho \Omega^J u_1-\frac 1t y\cdot \int \nabla_y \Omega^J u_1 d\rho=\frac t\rho \Omega^J u_1-\frac 1t \frac{1}{1-|y|^2}\int \omega^i \Omega_{0i} \, \Omega^J u_1 d\rho.
\end{equation*}
Let $Y=(1-|y|^2)^\frac 12 \pa_y$, $P=(1-|y|^2)^{-\frac 12}\rho\, \pa_\rho=(1-|y|^2)^{-\frac 12} S$, $Q=Y$ or $P$. Then one can show that
\begin{equation*}
    |\pa^I \Omega^J \int u_1 d\rho|\lesssim \rho^{-|I|+1} \sum_{|I_1|+|I_2|\leq |I|-1}|(Y^{|I_1|} (1-|y|^2)^{-\frac 12}) Q^{I_2} \Omega^{J} u_1|+\rho^{-|I|} |\int Y^{I} \Omega^J u_1 d\rho|,\quad |I|\geq 1.
\end{equation*}
We also have
\begin{equation}
    |Q^I f|\lesssim (1-|y|^2)^{-\frac {|I|}2}|Z^I f|,\quad |Y^I f|\lesssim (1-|y|^2)^{-\frac {|I|}2} |\Omega^I f|.
\end{equation}
Then since $|Z^I u_1|\lesssim \varepsilon^2 t^{-1}$, we get
\begin{equation}
    |\pa^I\Omega^J \int u_1 d\rho|\lesssim \varepsilon^2\rho^{-|I|} (1-|y|^2)^{-\frac {|I|-1}2} \ln\rho.
\end{equation}

\subsubsection{Estimates of $\phi_0$}We now estimate the approximate solution of the Klein-Gordon equation.
\begin{lemma}
We have $|\pa^I \Omega^J \phi_0|\lesssim \varepsilon\sum_{k\leq |J|} \rho^{-\frac 32-|I|} (\ln\rho)^{k} (1-|y|^2)^{\frac{|I|-(k-1)}2} \sum_{|I'|\leq |I|+|J|} |\Omega^{I'} a_\pm (y)|$.
\end{lemma}
\begin{proof}
Note that $\Omega\rho=0$. First we have
\begin{multline}
    \Omega^J \phi_0=\Omega^J \left(\rho^{-\frac 32}(e^{i\rho-\frac i2 \int u_1 d\rho} a_+(y)+e^{-i\rho+\frac i2 \int u_1 d\rho} a_-(y))\right)\\
    =\sum_{J_1+J_2+\cdots+J_k=J}\rho^{-\frac 32} \left(e^{i\rho-\frac i2 \int u_1 d\rho} \left(-\frac i2 \int \Omega^{J_1} u_1 d\rho\right)\cdots \left(-\frac i2\int \Omega^{J_{k-1}} u_1 d\rho\right)
    \Omega^{J_k} a_+(y)\right)\\
    +\rho^{-\frac 32} \left(e^{-i\rho+\frac i2 \int u_1 d\rho} \left(\frac i2\int \Omega^{J_1} u_1 d\rho\right)\cdots \left(\frac i2\int \Omega^{J_{k-1}} u_1 d\rho\right)
    \Omega^{J_k}
    a_-(y)\right),
\end{multline}
and the case when $|I|=0$ follows. The case when partial derivatives are present follows similarly as the calculation for $u_1$.
\end{proof}

In the interior of the light cone, i.e.\ $\{t>r\}$, we have
\begin{multline}
        -\Box \phi_0+\phi_0=\rho^{-\frac 32}\pa_\rho^2(\rho^\frac 32\phi_0)+\phi_0-\rho^{-2}(\triangle_y \phi_0+\frac 34\phi_0)\\
        =\rho^{-\frac 32}(\pa_\rho^2+1) \left(e^{i\rho-\frac i2 \int u_1 d\rho} a_+(y)+e^{-i\rho+\frac i2 \int u_1 d\rho} a_-(y)\right)\\
        -\rho^{-\frac 72}\left(e^{i\rho-\frac i2 \int u_1 d\rho}\left(\triangle_y a_+(y)+a_+(y)\right)+e^{-i\rho+\frac i2 \int u_1 d\rho}\left(\triangle_y a_-(y)+a_-(y)\right)\right).
\end{multline}
Since
\begin{multline*}
    \pa_\rho^2(e^{i\rho-\frac i2 \int u_1 d\rho})=-e^{i\rho-\frac i2 \int u_1 d\rho}+2i(-\frac i2)u_1 e^{i\rho-\frac i2 \int u_1 d\rho}+(-\frac i2 u_1)^2 e^{i\rho-\frac i2 \int u_1 d\rho}+(-\frac i2)(\pa_\rho u_1)e^{i\rho-\frac i2\int u_1 d\rho}\\
    =(-1+u_1-\frac 14 u_1^2-\frac i2 \rho^{-1} Su_1)e^{i\rho-\frac i2 \int u_1 d\rho},
\end{multline*}
and similarly for the other one, we have
\begin{multline}
    -\Box \phi_0+\phi_0=\rho^{-\frac 32}\pa_\rho^2(\rho^\frac 32\phi_0)+\phi_0-\rho^{-2}(\triangle_y \phi_0+\frac 34\phi_0)\\
    =\rho^{-\frac 32}(u_1-\frac 14 u_1^2-\frac i2 \rho^{-1}Su_1)(e^{i\rho-\frac i2 \int u_1 d\rho} a_+(y)+e^{-i\rho+\frac i2 \int u_1 d\rho} a_-(y))\\
    \quad -\rho^{-\frac 72}\left(e^{i\rho-\frac i2 \int u_1 d\rho}(\triangle_y(a_+(y))+a_+(y))+e^{-i\rho+\frac i2 \int u_1 d\rho}(\triangle_y (a_-(y))+a_-(y))\right)\\
    =u_1 \phi_0
    -\rho^{-\frac 72}e^{i\rho-\frac i2 \int u_1 d\rho}(\triangle_y a_+(y)+a_+(y)+(\frac 14 (\rho u_1)^2+\frac i2 (\rho Su_1)) a_+(y))\\
    -\rho^{-\frac 72} e^{-i\rho+\frac i2 \int u_1 d\rho}(\triangle_y a_-(y)+a_-(y)+(\frac 14 (\rho u_1)^2-\frac i2 (\rho Su_1)) a_-(y))
    =:u_1 \phi_0-R_0.
\end{multline}
Similar to the estimate of $\pa^I\Omega^J \phi_0$, we can show that, using $\triangle_y=\sum_{\alpha<\beta} \Omega_{\alpha\beta}^2$ and the estimates of $u_1$,
\begin{equation}
    |\pa^I \Omega^J R_0|\lesssim \sum_{|J'|\leq |J|+2} \rho^{-\frac 72} (\ln\rho)^{|J|} (1-|y|^2)^{-\frac {|I|}2} |\Omega^{J'} a_\pm (y)|.
\end{equation}

\subsubsection{The free radiation field}

Here as in \cite{lindblad2017scattering}, we choose $F_1$ so that
\begin{equation}
    2F_1'(q,\omega)=\triangle_\omega F_0(q,\omega),\quad F_1(0,\omega)=0.
\end{equation}
Then from the decay assumption of $F_0$, we have $|(\langle q\rangle \pa_q)^k \pa_\omega^\beta F_1(q,\omega)|\langle q\rangle^{-\alpha}\leq C\varepsilon$.
We also denote $\psi_{01}:=\chi({\frac{\langle t-r\rangle}{r}})\frac {F_0(r-t,\omega)}r+\chi({\frac{\langle t-r\rangle}{r}})\frac{F_1(r-t,\omega)}{r^2}$.
It is straightforward to verify that
\begin{equation}
    \left|\pa^I\Omega^J \psi_{01}\right| \lesssim \frac{\chi({\frac{\langle t-r\rangle}{2r}})}{\langle t+r\rangle}\left(\sum_{k+|\beta|\leq |I|+|J|}\left|(\langle q\rangle \pa_q)^k \pa_\omega^\beta F_0(q,\omega)\right|+\frac{1}{\langle t+r\rangle}\left|(\langle q\rangle \pa_q)^k \pa_\omega^\beta F_1(q,\omega)\right|\right).
\end{equation}
Under the choice of $F_1$, we also have the following estimate in \cite{lindblad2017scattering}:
\begin{equation}
    |\pa^I\Omega^J (\Box\psi_{01})|\lesssim \frac{\chi({\frac{\langle t-r\rangle}{2r}})}{\langle t+r\rangle^4}\sum_{|\beta|+k\leq |I|+|J|+2}\left(\langle q\rangle\left|(q\pa_q)^k \pa_\omega^\beta F_0(q,\omega)\right|+\left|(q\pa_q)^k \pa_\omega^\beta F_1(q,\omega)\right|\right).
\end{equation}
Therefore, using the decay conditions, we have
\begin{equation}
     |\pa^I\Omega^J \psi_{01}|\lesssim \varepsilon\langle t+r \rangle^{-1} \langle q\rangle^{-1+\alpha},\quad |\pa^I\Omega^J(\Box\psi_{01})|\lesssim \varepsilon \langle t+r\rangle^{-4}\langle q\rangle^\alpha
\end{equation}
in their support.

\subsubsection{The estimate of $u_2$ and $u_3$}
\begin{lemma}
We have
\begin{equation}
    |\pa^I \Omega^J u_2|+|\pa^I \Omega^J u_3|\lesssim \varepsilon^2 t^{-1} (1+|t-r|)^{-1+\delta}.
\end{equation}
\end{lemma}

\begin{proof}
We commute the equations of $u_2$ and $u_3$ with $\pa^I \Omega^J$. Recall $u_2$ satisfies
\begin{equation}
    -\Box u_2=\rho^{-3} \left(e^{2i\rho-i\int u_1 d\rho}(1-(1-|y|^2)^{-1})a_+^2(y)+e^{-2i\rho+i\int u_1 d\rho}(1-(1-|y|^2)^{-1})a_-^2(y)\right)
\end{equation}
For simplicity we only derive the estimate of the $a_+(y)$ part. Again by direct computation, we have
\begin{multline}
    \Omega^J \left(\rho^{-3} \left(e^{2i\rho-i\int u_1 d\rho}(1-(1-|y|^2)^{-1})a_+^2(y)\right)\right)\\
    =\sum_{J_1+J_2+\cdots+J_k=J}\rho^{-3} \Bigg(e^{2i\rho-i\int u_1 d\rho} \left(-i\int \Omega^{J_1} u_1 d\rho\right)\left(-i\int \Omega^{J_2} u_1 d\rho\right)\cdots \left(-i\int \Omega^{J_{k-1}}u_1 d\rho\right)\\
    \cdot \left(\Omega^{J_k} (1-(1-|y|^2)^{-\frac 12})a_+(y)^2\right)\Bigg).
\end{multline}
Each summand here has an oscillation factor, so we can then use the integration by part argument as in Section \ref{subsectionoscillation}. Note that we do not have perfect functions of $y$ here, but it is not hard to see that the argument still works (we will get a new term with the estimate similar to $C_1$ there). Then by linearity we get the estimate of $\Omega^J u_2$. This concludes the case when $|I|=0$.
If $\pa^I$ is present with $|I|>0$, the only term that is not one order better in $\rho$ is when all $\pa_\rho$ parts fall on the phase $e^{i\rho}$. In this case we get a term
$$(2i)^{|I|}\rho^{-3}e^{2i\rho-i\int u_1 d\rho}(1-|y|^2)^{-\frac {|I|}2}(1-(1-|y|^2)^{-1})a_+^2(y),$$
which is the same type of term where we can use integration by parts.
The remaining terms can be bounded by
$$\sum_{k\leq |J|}\rho^{-3-|I|} (1-|y|^2)^{-\frac {|I|-(k-1)}2} (\ln\rho)^{k-1} \sum_{|J'|\leq |I|+|J|}|\Omega^{J'}(1-(1-|y|^2)^{-1})a_+^2(y)|,$$
which is then bounded by $t^{-4}(\ln t)^{|J|}$ provided the condition on $a_\pm (y)$. 
The estimate for $u_3$ also follows as the source term behaves like these remaining terms.
\end{proof}

\begin{remark}
This shows that the radiation field of $\pa^I\Omega^J u_2$ and $\pa^I\Omega^J u_3$, which clearly exist, satisfy the decay $(1+q_-)^{-1+\delta}$ in the interior. Moreover, using the argument in \cite[Section 6.2]{H97}, we have these radiation fields correspond to the radiation field of $u_2$ and $u_3$ applied with vector fields. Therefore, if we denote the radiation field of $u_2+u_3$ by $\widetilde F(q,\omega)$, then we have the bound
\begin{equation}
    |(q\pa_q)^k \pa_\omega^\beta \widetilde F(q,\omega)|\lesssim \varepsilon^2 (1+q_-)^{-1+\delta}.
\end{equation}
This shows the equivalence between conditions \eqref{decayF0} and \eqref{decayF}.
\end{remark}

\subsection{The system}
We consider the system
\begin{equation}
    -\Box (u_0+v)=(\pa_t \phi_0+\pa_t w)^2+(\phi_0+w)^2,\quad -\Box (\phi_0+w)+(\phi_0+w)=(u_0+v)(\phi_0+w).
\end{equation}
This implies
\begin{equation}
    -\Box v=2(\pa_t\phi_0)\pa_t w+(\pa_t w)^2+2\phi_0 w+w^2+\Box\psi_{01},
\end{equation}
and
\begin{equation}
    -\Box w+w=(u_2+u_3+\psi_{01})\phi_0+u_0 w+\phi_0 v+vw+R_0.
\end{equation}
Now let $T>2$. We consider the functions $v_T$ and $w_T$ such that the following equations hold:
\begin{equation}
    \begin{split}
        &-\Box v_T=\chi(t/T) \left(2(\pa_t\phi_0)\pa_t w_T+(\pa_t w_T)^2+2\phi_0 w_T+w_T^2+\Box\psi_{01}\right),\\
        &-\Box w_T+w_T=\chi(t/T) \left(\left(u_2+u_3+\psi_{01}\right)\phi_0+u_0 w_T+\phi_0 v_T+v_T w_T+R_0\right).
    \end{split}
\end{equation}
with vanishing initial data at $t=T$. As in \cite{lindblad2017scattering}, the cutoff in time is used for technical reasons that we want $v_T$ and $w_T$ to be zero near $t=T$ after applying vector fields $\pa^I \Omega^J$. Since $\chi(t/T)=0$ when $t\geq T/4$, we see that $v_T$ and $w_T$ are zero functions near (and after) $t=T$, so in particular they (and their derivatives) vanish on the hyperboloid $H_T$ (which is in the future of $\{t=T\}$).

Applying vector fields to the equations, we get
\begin{equation}
    \begin{split}
        -\Box \pa^I \Omega^J v_T&=\pa^I \Omega^J (\chi(t/T) (2(\pa_t\phi_0)\pa_t w_T+(\pa_t w_T)^2+2\phi_0 w_T+w_T^2+\Box \psi_{01}))\\
        -\Box \pa^I \Omega^J w_T+w_T&=\pa^I \Omega^J (\chi(t/T) ((u_2+u_3+\psi_{01})\phi_0+u_0 w_T+\phi_0 v_T+v_T w_T+R_0)),
    \end{split}
\end{equation}

\begin{lemma}
We have $|Z^I (\chi(\frac tT))|\lesssim \chi(\frac t{2T})$ in the relevant region. By relevant region we mean the region where $v_T$ and $w_T$ are nonzero.
\end{lemma}
\begin{proof}
Notice that $v_T$ and $w_T$ is zero when $t>T/2$, and supported in $|x|\lesssim T/2$ near $t=T/2$. Therefore, by the finite speed of propagation, we have in the relevant region that $t+r\lesssim T$. Since taking derivative of $\chi$ means $t\sim T$, we have $r\lesssim T$. Then the estimate follows in view of the expressions of vector fields.
\end{proof}

\subsection{Energy estimates}
We consider the following foliation: the interior part $\Sigma_\rho^i=\widetilde H_\rho$ of the slice $\Sigma_\rho$ is defined as the restriction of $H_\rho$ in the region where $t-r\geq r^\frac 12$. Then we extend it to the exterior by constant time slices. We denote the constant time part by $\Sigma_\rho^e$, which means the exterior part of $\Sigma_\rho^e$, so $\Sigma_\rho=\Sigma_\rho^i \cup \Sigma_\rho^e$.

The hyperboloidal part $\widetilde H_\rho$ intersects with $\Sigma_\rho^e$ at a point $(t(\rho),x(\rho))$ on $\{t-r=r^\frac 12\}$. Since $t(\rho)-r(\rho)=r(\rho)^\frac 12$ and $t(\rho)^2-r(\rho)^2=\rho^2$ where $r(\rho)=|x(\rho)|$, we have $\rho=\sqrt{2r(\rho)^\frac 32+r(\rho)}$. This determines the function $r(\rho)$, and hence $t(\rho)$, and we have $t(\rho)\sim r(\rho) \sim \rho^{\frac 43}$. By implicit function theorem, we also have $r'(\rho)=2\rho/(3(r(\rho))^\frac 12+1)\sim \rho^\frac 13$, $t'(\rho)=(1-\frac 12 (r(\rho))^{-\frac 12}) r'(\rho)\sim \rho^{\frac 13}$.

Under this notation, we also see that $v_T$ and $w_T$ vanish near $\Sigma_T$.

We define the energy
\begin{equation}
    E_w(\rho,f)=\int_{\widetilde H_\rho} |(\rho/t) \pa_t f|^2+|\overline\pa_i f|^2 dx+\int_{\Sigma_\rho^e} |\pa f|^2 dx,
\end{equation}
\begin{equation}
    E_{KG}(\rho,f)=\int_{\widetilde H_\rho} |(\rho/t) \pa_t f|^2+|\overline\pa_i f|^2+|f|^2 dx+\int_{\Sigma_\rho^e} |\pa f|^2+|f|^2 dx,
\end{equation}
and their higher-order versions
\begin{equation}
    E_{w;k}(\rho,f)=\sum_{|I|+|J|\leq k} E_{w}(\rho,\pa^I \Omega^J f),\quad E_{KG;k}(\rho,f)=\sum_{|I|+|J|\leq k} E_{KG}(\rho,\pa^I \Omega^J f).
\end{equation}
We make the bootstrap assumptions that
\begin{equation}
    E_{w;k}(\rho,v_T)^\frac 12\leq C_b \varepsilon \rho^{-\frac 32+\alpha+k\delta},\quad E_{KG;k}(\rho,w_T)^\frac 12\leq C_b\varepsilon \rho^{-1+\alpha+k\delta},\quad k\leq N,\quad \text{for all } \widetilde{T}\leq \rho\leq T.
\end{equation}
where $C_b$ is some constant which we will determine later. This clearly holds when $\rho=T$. We then improve this bound to show that $\widetilde T=2$. From now on, the notation ``$\lesssim$'' means the existence of a constant that does not depend on $C_b$.

We now state the backward energy estimate for our foliation.
\begin{prop}
For $\rho_1<\rho_2$, we have the energy estimate
\begin{multline}
    \left(\int_{\widetilde H_{\rho_1}} |(\rho/t) \pa_t f|^2+|\overline\pa_i f|^2+m^2|f|^2 dx+\int_{\Sigma_{\rho_1}^e} |\pa f|^2+m^2 |f|^2 dx\right)^\frac 12\\
    \lesssim \left(\int_{\widetilde H_{\rho_2}} |(\rho/t) \pa_t f|^2+|\overline\pa_i f|^2+m^2|f|^2 dx+\int_{\Sigma_{\rho_2}^e} |\pa f|^2+m^2 |f|^2 dx\right)^\frac 12\\
    +\int_{\rho_1}^{\rho_2}||(-\Box+m^2) f||_{L^2(\widetilde H_\rho)} d\rho+\int_{\rho_1}^{\rho_2} \rho^\frac 13 ||(-\Box+m^2) f||_{L^2(\Sigma_\rho^e)} d\rho,
\end{multline}
where $m=0$ or $1$. Note that same as $H_\rho$, we use the measure $dx$ to define the $L^2$ norm on $\widetilde H_\rho$.
\end{prop}
\begin{proof}
By standard energy identity induced in our region, we get
\begin{multline*}
    \int_{\widetilde H_\rho} |(\rho/t) \pa_t f|^2+|\overline\pa_i f|^2+m^2|f|^2 dx+\int_{\Sigma_{\rho}^e} |\pa f|^2+m^2 |f|^2 dx\\
    =\int_{\widetilde H_{\rho_2}} |(\rho/t) \pa_t f|^2+|\overline\pa_i f|^2+m^2|f|^2 dx+\int_{\Sigma_{\rho_2}^e} |\pa f|^2+m^2 |f|^2 dx\\
    +\int_\rho^{\rho_2}\int_{\widetilde H_s} \pa_t f (-\Box+m^2)f \cdot (s/t) dx dt+\int_\rho^{\rho_2}\int_{\Sigma_s^e} \pa_t f (-\Box+m^2)f dx dt.
\end{multline*}
In the interior part we have $ds=\frac st dt$, and in the exterior part $dt\approx s^\frac 13 ds$. Hence, The last line is bounded by
\begin{equation*}
    \int_\rho^{\rho_2} ||(-\Box+m^2)f||_{L^2(\widetilde H_s)}||(s/t)\pa_t f||_{L^2(\widetilde H_s)}+||(-\Box+m^2) f||_{L^2(\Sigma_s^e)} ||\pa_t f||_{L^2(\Sigma_s^e)} s^\frac 13 ds.
\end{equation*}
Denote $E(\rho)=\int_{H_\rho} |(\rho/t) \pa_t f|^2+|\overline\pa_i f|^2+m^2|f|^2 dx+\int_{\Sigma_{\rho}^e} |\pa f|^2+m^2 |f|^2 dx$. Then we have
\begin{equation*}
    \left|2E(\rho)^\frac 12\frac{d}{d\rho} (E(\rho)^\frac 12)\right|=\left|\frac{d}{d\rho} E(\rho)\right|\lesssim \left(||(-\Box+m^2) f||_{L^2(\widetilde H_\rho)}+\rho^\frac 13||(-\Box+m^2) f||_{L^2(\Sigma_s^e)}\right) E(\rho)^\frac 12,
\end{equation*}
so \[\left|\frac{d}{d\rho} (E(\rho)^\frac 12)\right|\lesssim  ||(-\Box+m^2) f||_{L^2(\widetilde H_s)}+\rho^\frac 13 ||(-\Box+m^2) f||_{L^2(\Sigma_s^e)} dx,\]
and the estimate follows.
\end{proof}

We want to establish the estimates for $v_T$ and $w_T$. Notice that everything is zero after $t=T/2$. Therefore, everything is zero on $\Sigma_\rho^e$ for $\rho\geq 2T^\frac 43$. Then the energy estimates read
\begin{multline}
    E_{w;k}(\rho,v_T)^\frac 12\leq E_{w;k}(T,v_T)^\frac 12\\
    +\int_\rho^{T/2}  \sum_{|I|+|J|\leq k}\left|\left|\pa^I \Omega^J \left(2(\pa_t\phi_0)\pa_t w_T+(\pa_t w_T)^2+2\phi_0 w_T+w_T^2+\Box\psi_{01}\right)\right|\right|_{L^2(\widetilde H_s,dx)} ds\\
    +\int_\rho^{2T^\frac 34} s^
    \frac 13 \sum_{|I|+|J|\leq k} \left|\left|\pa^I \Omega^J \left(2(\pa_t\phi_0)\pa_t w_T+(\pa_t w_T)^2+2\phi_0 w_T+w_T^2+\Box \psi_{01}\right)\right|\right|_{L^2(\Sigma_s^e)} ds
\end{multline}
and
\begin{multline}
    E_{KG;k}(\rho,w_T)^\frac 12\leq E_{KG;k}(T,w_T)^\frac 12\\
    +\int_\rho^{T/2} \sum_{|I|+|J|\leq k} \left|\left|\pa^I \Omega^J\left((u_2+u_3+\psi_{01})\phi_0+u_0 w_T+\phi_0 v_T+v_T w_T+R_0\right)\right|\right|_{L^2(\widetilde H_s)} ds\\
    +\int_\rho^{2T^\frac 34} s^\frac 13 \sum_{|I|+|J|\leq k} \left|\left|\pa^I \Omega^J \left((u_2+u_3+\psi_{01})\phi_0+u_0 w_T+\phi_0 v_T+v_T w_T+R_0\right)\right|\right|_{L^2(\Sigma_s^e)} ds.
\end{multline}

\subsubsection{$L^2$ estimates of approximate solutions}
For energy estimates, we need to control the $L^2$ norms of quantities involving approximate solutions. We first consider the estimate in the interior $\{t-r\geq r^\frac 12\}$. In view of the support of $\psi_{01}$, we have $r\geq \frac{1}{\sqrt{3}}\rho$ wherever $\psi_{01}\neq 0$. Therefore we have
\begin{multline}
    ||\pa^I\Omega^J(\Box \psi_{01})||_{L^2(\widetilde H_\rho)}\lesssim ||\varepsilon\langle t+r\rangle^{-4} \langle q\rangle^{\alpha}||_{L^2(|x|\geq \frac {1}{\sqrt 3}\rho,dx)}
    \lesssim \varepsilon\rho^{-2+2\alpha} \left(\int_{\frac 1{\sqrt 3} \rho}^\infty (r^{-2-\alpha})^2 r^2 dr\right)^\frac 12\\
    \lesssim \varepsilon\rho^{-2+2\alpha} (\rho^{-1-2\alpha})^\frac 12\lesssim \varepsilon\rho^{-\frac 52+\alpha}.
\end{multline}

Similarly, we have using $\ln t\sim \ln\rho$ in the interior (in fact, one should be able to replace the $|t-r|$ by $t$ because of the decay of $\phi_0$ in $(1-|y|^2)$)
\begin{multline}
    ||\pa^I\Omega^J(\psi_{01}\phi_0)||_{L^2(\widetilde H_\rho)}\lesssim ||\varepsilon^2\langle t+r\rangle^{-\frac 52} \langle q\rangle^{-1+\alpha}(\ln \rho)^{|J|}||_{L^2(\widetilde H_\rho)}\\
    \lesssim \varepsilon^2\rho^{-2+2\alpha}(\ln\rho)^{|J|}||\langle t+r\rangle^{-\frac 32-\alpha}||_{L^2(\widetilde H_\rho)}
    \lesssim \varepsilon^2\rho^{-2+2\alpha}(\ln\rho)^{|J|}\left(\int_{\sqrt 3 \rho}^{10\rho^\frac 43} r^{-3-2\alpha} r^2 dr\right)^\frac 12\\
    \lesssim_\alpha \varepsilon^2\rho^{-2+2\alpha}(\ln\rho)^{|J|}(\rho^{-2\alpha})^\frac 12\lesssim \varepsilon^2\rho^{-2+\alpha}(\ln\rho)^{|J|}.
\end{multline}
We also have
\begin{multline}
    ||\pa^I \Omega^J ((u_2+u_3)\phi_0)||_{L^2(\widetilde H_\rho)}\lesssim ||\varepsilon^2 t^{-\frac 52} (t-r)^{-1+\delta} (\ln \rho)^{|J|}||_{L^2(\widetilde H_\rho)}\\
    \lesssim \varepsilon^2\rho^{-2+2\delta} (\ln\rho)^{|J|}\left(\int_0^{10\rho^{\frac 43}} \langle r\rangle^{-3}r^{2} dr\right)^\frac 12
    \lesssim \varepsilon^2\rho^{-2+2\delta}(\ln\rho)^{|J|+\frac 12},
\end{multline}
\begin{equation}
    ||\pa^I \Omega^J R_0||_{L^2(\widetilde H_\rho)}\lesssim ||\varepsilon t^{-\frac 72} (\ln\rho)^{|J|}||_{L^2(\widetilde H_\rho)}\lesssim \varepsilon\rho^{-2}(\ln\rho)^{|J|+\frac 12}.
\end{equation}

We now estimate the part in $\{t-r\leq r^\frac 12\}$. We have, using $r\sim t$ in the support of $\chi(\frac{\langle t-r\rangle}{2r})$:
\begin{multline}
    ||\pa^I\Omega^J(\Box \psi_{01})||_{L^2(\Sigma_\rho^e)}\lesssim ||\varepsilon\langle t+r\rangle ^{-4} \langle q\rangle ^{\alpha}||_{L^2(\Sigma_\rho\cap\{0\leq t-r\leq r^\frac 12\})}
    \lesssim \left(\int_{-t\leq t-r\leq 2t^{\frac 12}} \varepsilon\frac{r^2}{\langle t+r\rangle^8}\langle q\rangle^{2\alpha} d\omega dr\right)^\frac 12\\
    \lesssim \left(\int_{-t}^{2t^\frac 12} \varepsilon t^{-6} \langle q\rangle^{2\alpha} dq d\omega\right)^\frac 12\lesssim \varepsilon (t^{-6}t^{2\alpha+1})^{\frac 12}=\varepsilon t^{-\frac 52+\alpha}\sim \varepsilon\rho^{-\frac{10}3+\frac 43\alpha}.
\end{multline}
We also have the following estimates
\begin{multline}
    ||\pa^I\Omega^J((u_2+u_3+\psi_{01})\phi_0)||_{L^2(\Sigma_\rho^e)}\lesssim ||\varepsilon^2 t^{-1}(\ln t)^{|J|} t^{-\frac 32}||_{L^2(\Sigma_\rho^e\cap \{t\leq r\})}\lesssim \varepsilon^2 t^{-\frac 52}(\ln t)^{|J|}\left(\int_{t-2t^\frac 12}^t 1 d\omega dr\right)^\frac 12\\
    \lesssim \varepsilon^2 t^{-\frac 94}(\ln t)^{|J|} \sim \varepsilon^2 \rho^{-3}(\ln\rho)^{|J|},
\end{multline}

\begin{equation}
    ||\pa^I\Omega^J R_0||_{L^2(\Sigma_\rho^e)}\lesssim \varepsilon t^{-\frac 72}(\ln t)^{|J|} (\int_{t-2t^\frac 12}^t 1 d\omega dr)^\frac 12\lesssim \varepsilon t^{-\frac{13}4}(\ln t)^{|J|}\sim \varepsilon \rho^{-\frac {13}3}(\ln \rho)^{|J|}.
\end{equation}

\subsubsection{Sobolev inequalities}We need Sobolev inequalities on hyperboloids and constant time slices to derive decay estimates. In the interior, we have the truncated version of the Klainerman-Sobolev inequality on hyperboloids:
\begin{prop}\label{KStruncatedprop}
We have
\begin{equation}
    \sup_{\widetilde H_\rho} t^\frac 32 |f|\leq C \sum_{|I|\leq 2} ||L^I f||_{L^2(\widetilde H_\rho)}.
\end{equation}
\end{prop}
This version follows from the proof of Proposition 4 in \cite{FangMKG}.

In the exterior we shall use rotation vector fields.
\begin{lemma}\label{decaySobolevKG}
Suppose $\lim_{r\rightarrow\infty}\sum_{|I|\leq 2}\int_{S_r}|\Omega_{ij}^I f|^2 d\omega=0$. Then
\begin{equation}
    |f(r_0\omega)|^2\lesssim (r_0)^{-2} \sum_{|I|\leq 2}\int_{|x|\geq r} |\Omega^I f|^2+|\pa_r \Omega^I f|^2\, dx.
\end{equation}
\end{lemma}

\begin{proof}
Using Sobolev embedding on the unit sphere, we have
\begin{equation*}
\begin{split}
    |f(r_0\omega)|^2&\lesssim \sum_{|I|\leq 2}\int_{S_{r_0}} |\Omega_{ij}^I f|^2 d\omega\lesssim \sum_{|I|\leq 2}\int_{r_0}^\infty \pa_r\left(\int_{\mathbb{S}^2} |\Omega_{ij}^I f(r\omega)|^2 d\omega\right) dr\\
    &\lesssim \sum_{|I|\leq 2}\int_{r_0}^\infty \int_{\mathbb{S}^2}\Omega_{ij}^I f\cdot\pa_r\Omega_{ij}^I f d\omega dr\lesssim \sum_{|I|\leq 2}\int_{|x|\geq r_0}(|\Omega_{ij}^I f|^2+|\pa_r\Omega_{ij}^If|^2)r^{-2}r^2 d\omega dr\\
    &\lesssim (r_0)^{-2}\sum_{|I|\leq 2}\int_{|x|\geq r_0} |\Omega_{ij}^I f|^2+|\pa_r\Omega_{ij}^I f|^2 dx
    \end{split}
\end{equation*}
as required.
\end{proof}

We also need a version that does not use the norm of $f$ itself for the wave component.
\begin{lemma}\label{lemmasphereintegral}
Suppose $f$ is spatially compactly supported. Then
\begin{equation}
    \int_{\mathbb{S}^2} |f(r_0\omega)|^2 d\omega\leq 2(r_0)^{-1} \int_{|x|\geq r_0} |\pa_r f|^2.
\end{equation}
\end{lemma}

\begin{proof}
We have
\begin{equation*}
    \begin{split}
        \int_{\mathbb{S}^2} |f(r_0\omega)|^2 d\omega&=\int_{\mathbb{S}^2} \left(f(r_0\omega)-f(R\omega)\right)^2+2f(r_0\omega) f(R\omega) -f^2(R\omega)\, d\omega\\
        &\leq \int_{\mathbb{S}^2}\left(\int_{r_0}^R \pa_r f dr\right)^2 d\omega+\int_{\mathbb{S}^2} \frac 12 f(r_0\omega)^2+2f^2(R\omega) d\omega
    \end{split}
\end{equation*}
For $R$ big enough, we have $f(R\omega)=0$, so
\begin{equation*}
    \begin{split}
        \int_{\mathbb{S}^2} |f(r_0\omega)|^2 d\omega&\leq 2\int_{\mathbb{S}^2}\left(\int_{r_0}^{\infty} \pa_r f dr\right)^2 d\omega=2\int_{\mathbb{S}^2} \left(\int_{r_0}^{\infty} r^{-1}\cdot r\pa_r f dr\right)^2d\omega\\
        &\leq 2\int_{\mathbb{S}^2} \left(\int_{r_0}^{\infty} r^{-2} dr\right)\left(\int_{r_0}^{\infty} |\pa_r f|^2 r^2 dr\right) d\omega \leq 2(r_0)^{-1} \int_{r_0}^{\infty} |\pa_r f|^2 r^2 dr d\omega\\
        &\leq 2(r_0)^{-1} \int_{|x|\geq r_0} |\pa_r f|^2 dx.\qedhere
    \end{split}
\end{equation*}
\end{proof}
Then applying Sobolev embedding on the unit sphere, we get
\begin{lemma}\label{decaySobolevWave}
We have
\begin{equation}
    |f(r_0\omega)|^2\lesssim (r_0)^{-1} \sum_{|I|\leq 2}\int_{|x|\geq r_0} |\pa_r \Omega^I f|^2 dx.
\end{equation}
\end{lemma}


\subsubsection{Hardy-type estimate} The energy of $v_T$ gives control of derivatives of $v_T$. We also need to estimate the $L^2$ norm of $v_T$ itself. For this we derive Hardy-type estimates.

\begin{lemma}[Hardy estimates]
Let $f$ be a real-valued spatially compactly supported function. We have
\begin{equation}
    ||r^{-1}f||_{L^2(\widetilde H_\rho)}^2\leq 2 E_{w}(\rho,f),\quad  ||r^{-1} f||^2_{L^2(\Sigma_\rho^e)}\leq 8E_w(\rho,f).
\end{equation}
\end{lemma}
\begin{proof}
Define $f_\rho(x)=f(\sqrt{\rho^2+|x|^2},x)$. Note that when the hyperboloid intersects with the exterior slice, we have $r=r(\rho)\approx \rho^\frac 43$. Then one has
\begin{equation}
    \begin{split}
        \int_{|x|\leq r(\rho)}& \frac{(f_\rho(x))^2}{|x|^2} dx=\int_0^{r(\rho)}\int_{\mathbb{S}^2} |f_\rho(r\omega)|^2 d\omega dr=\int_{\mathbb{S}^2} (rf_\rho^2)|_{r=0}^{r=r(\rho)} d\omega -2\int_{\mathbb{S}^2}\int_0^{r(\rho)} rf_\rho\, \pa_r f_\rho\, drd\omega \\
        &\ \ \leq r(\rho)\int_{\mathbb{S}^2} f_\rho^2|_{r=r(\rho)}d\omega+\frac 12 \int_{\mathbb{S}^2}\int_0^{r(\rho)} \frac{|f_\rho|^2}{r^2} r^2 drd\omega +2\int_{\mathbb{S}^2}\int_0^{r(\rho)} |\pa_r f_\rho|^2 r^2 dr d\omega,
    \end{split}
\end{equation}
so we get
\begin{equation}\label{rhssphere}
    \int_{|x|\leq r(\rho)} \frac{(f_\rho(x))^2}{|x|^2} dx\leq 2r(\rho)\int_{S_{r(\rho)}} f_\rho^2 d\omega+4\int_{|x|\leq r(\rho)} |\pa_r f_\rho|^2 dx.
\end{equation}
For the first term on the right hand side, we use Lemma \ref{lemmasphereintegral}. We then get
$$\int_{|x|\leq r(\rho)} \frac{(f_\rho(x))^2}{|x|^2} dx\leq 4\int_{\Sigma_\rho^e} |\pa f|^2 dx+4\int_{|x|\leq {r(\rho)}} |\pa_r f_\rho|^2 dx.$$
Recall the definition of $f_\rho$, we have $\pa_r f_\rho=\omega^i\overline\pa_i f$. Hence we have
\begin{equation*}
    ||r^{-1}f||_{L^2(\widetilde H_\rho)}^2\leq 2 E_{w}(\rho,f).
\end{equation*}
This proves the first estimate. For the exterior estimate, one has
\begin{multline}
    \int_{|x|\geq r(\rho)} \frac{f^2}{|x|^2} dx=\int_{\mathbb{S}^2}\int_{r(\rho)}^\infty f^2 drd\omega=\int_{\mathbb{S}^2} rf^2|_{r=r(\rho)}^{r=\infty}d\omega -2\int_{r(\rho)}^\infty f \pa_r f dr d\omega \\
    \leq \int_{S_{r(\rho)}} r(\rho) f^2 d\omega+\int_{|x|\geq r(\rho)} \frac 12 \frac{f^2}{|x|^2} + 2|\pa_r f|^2 dx,
\end{multline}
so we get similarly that
\begin{equation*}
    ||r^{-1} f||^2_{L^2(\Sigma_\rho^e)}\leq 2 r(\rho) \int_{S_{r(\rho)}} |f|^2 d\omega+4||\pa f||_{L^2(\Sigma_\rho^e)}^2\leq 8||\pa f||_{L^2(\Sigma_\rho^e)},
\end{equation*}
where we again used Lemma \ref{lemmasphereintegral} for the last inequality.
\end{proof}
Therefore, by the energy decay assumptions, we have
\begin{equation}\label{Hardybound}
    \sum_{|I|+|J|\leq k}||r^{-1}\pa^I \Omega^J v_T||_{L^2(\widetilde H_\rho)}\lesssim C_b\varepsilon\rho^{-\frac 32+\alpha+k\delta},\quad \sum_{|I|+|J|\leq k}||r^{-1} \pa^I\Omega^J v_T||_{L^2(\Sigma_\rho^e)}\lesssim C_b\varepsilon\rho^{-\frac 32+\alpha+k\delta}.
\end{equation}

\subsubsection{Interior estimate}
By the bootstrap bounds, the estimate from Hardy's inequality \eqref{Hardybound}, and Proposition \ref{KStruncatedprop}, we have the following decay estimates:
$$\sup_{\widetilde H_\rho} t^\frac 32|t^{-1}\pa^I \Omega^J v_T|\lesssim C_b \varepsilon \rho^{-\frac 32+\alpha+(k+2)\delta},\quad |I|+|J|\leq k$$
$$\sup_{\widetilde H_\rho} t^\frac 32|\pa^I \Omega^J w_T|\lesssim C_b\varepsilon \rho^{-1+\alpha+(k+2)\delta},\quad |I|+|J|\leq k$$
for $k\leq N-2$.

We need to estimate
$$I_{w,i;k}:=\int_\rho^{T/2} \sum_{|I|+|J|\leq k} \left|\left|\pa^I\Omega^J\left(2(\pa_t\phi_0)\pa_t w_T+(\pa_t w_T)^2+2\phi_0 w_T+w_T^2+\Box\psi_{01}\right)\right|\right|_{L^2(\widetilde H_s)} ds$$
and
$$I_{KG,i;k}:=\int_\rho^{T/2} \sum_{|I|+|J|\leq k}\left|\left|\pa^I\Omega^J\left((u_2+u_3+\psi_{01})\phi_0+u_0 w_T+\phi_0 v_T+v_T w_T+R_0\right)\right|\right|_{L^2(\widetilde H_s)} ds.$$

We have
\begin{multline*}
    I_{w,i;k}\lesssim \int_\rho^{T/2} \sum_{|I|+|J|\leq k}||\phi_0||_{L^\infty(\widetilde H_s)}||\pa^I \Omega^J w_T||_{L^2(\widetilde H_s)}+||(t/s)\pa_t\phi_0||_{L^\infty(\widetilde H_s)}||(s/t)\pa_t\pa^I \Omega^J w_T||_{L^2(\widetilde H_s)}\\
    +\sum_{\substack{|(I_1,J_1,I_2,J_2)|\leq k\\ |I_2|+|J_2|\leq k-1}}||\pa^{I_1} \Omega^{J_1} \phi_0||_{L^\infty(\widetilde H_s)}||\pa^{I_2}\Omega^{J_2} w_T||_{L^2(\widetilde H_s)}+||\frac ts \pa_t\pa^{I_1} \Omega^{J_1} \phi_0||_{L^\infty(\widetilde H_s)}||\frac st \pa_t\pa^{I_2}\Omega^{J_2} w_T||_{L^2(\widetilde H_s)}\\
    +\sum_{|(I_1,J_1,I_2,J_2)|\leq k}||(t/s)\pa_t \pa^{I_1} \Omega^{J_1} w_T||_{L^\infty(\widetilde H_s)} ||(s/t)\pa_t \pa^{I_2} \Omega^{J_2} w_T||_{L^2(\widetilde H_s)}\\
    +\sum_{|(I_1,J_1,I_2,J_2)|\leq k}||\pa^{I_1} \Omega^{J_1} w_T||_{L^\infty(\widetilde H_s)} ||\pa^{I_2} \Omega^{J_2} w_T||_{L^2(\widetilde H_s)}+\sum_{|I|+|J|\leq k}||\pa^I\Omega^J (\Box\psi_{01})||_{L^2(\widetilde H_s)} ds\\
    \lesssim \int_\rho^{T/2} C_b\varepsilon^2 s^{-\frac 32} s^{-1+\alpha+k\delta}+ C_b\varepsilon^2 s^{-\frac 32} (\ln s)^{k} s^{-1+\alpha+(k-1)\delta}+C_b^2 \varepsilon^2 s^{-\frac 52+\alpha+(k+2)\delta} s^{-1+\alpha+k\delta}+\varepsilon s^{-\frac 52+\alpha} ds\\
    \lesssim (C_b^2 \varepsilon^2+\varepsilon) \rho^{-\frac 32+\alpha+k\delta},
\end{multline*}
and
\begin{multline*}
    I_{KG,i;k}\lesssim \int_\rho^{T/2} \sum_{|I|+|J|\leq k}||\pa^I\Omega^J R_0||_{L^2(\widetilde H_s)}+||\pa^{I} \Omega^{J}((u_2+u_3+\psi_{01})\pa^{I_2}\Omega^{J_2} \phi_0)||_{L^2(\widetilde H_s)}\\
    +\sum_{|(I_1,J_1,I_2,J_2)|\leq k}||\pa^{I_1}\Omega^{J_1} u_0 \pa^{I_2}\Omega^{J_2} w_T||_{L^2(\widetilde H_s)}\\
    +||\pa^{I_1}\Omega^{J_1} u_0 \pa^{I_2} w_T||_{L^2(\widetilde H_s)}+||\pa^{I_1}\Omega^{J_1} \phi_0 \pa^{I_2}\Omega^{J_2} v_T||_{L^2(\widetilde H_s)}+||\pa^{I_1}\Omega^{J_1} v_T \pa^{I_2}\Omega^{J_2} w_T||_{L^2(\widetilde H_s)} ds\\
    \lesssim \int_\rho^{T/2} C\varepsilon s^{-2}(\ln s)^{k+\frac 12} ds+C\varepsilon^2 s^{-2+2\delta}(\ln s)^{k+\frac 12}+C\varepsilon^2 s^{-2+\alpha}(\ln s)^k\\
    +||u_0||_{L^\infty(\widetilde H_s)}\sum_{|I|+|J|\leq k}||\pa^I\Omega^J w_T||_{L^2(\widetilde H_s)}+ ||(t/s)\phi_0||_{L^\infty(\widetilde H_s)} \sum_{|I|+|J|\leq k}||\frac st \pa^I \Omega^J v_T||_{L^2(\widetilde H_s)}\\
    +\sum_{\substack{|(I_1,J_1,I_2,J_2)|\leq k\\|I_2|+|J_2|\leq k-1}}||\pa^{I_1}\Omega^{J_1}u_0||_{L^\infty(\widetilde H_s)}||\pa^{I_2}\Omega^{J_2}w_T||_{L^2(\widetilde H_s)}+||(t/s)\pa^{I_1}\Omega^{J_2}\phi_0||_{L^\infty(\widetilde H_s)} ||\frac st \pa^{I_2}\Omega^{J_2} v_T||_{L^2(\widetilde H_s)}\\
    +\sum_{\substack{|(I_1,J_1,I_2,J_2)|\leq k\\ |I_1|+|J_1|\leq k/2}} ||\pa^{I_1}\Omega^{J_1} v_T||_{L^\infty(\widetilde H_s)} ||\pa^{I_2}\Omega^{J_2} w_T||_{L^2(\widetilde H_s)}+ ||\frac t\rho \pa^{I_2}\Omega^{J_2} w_T||_{L^\infty(\widetilde H_s)} ||\frac \rho t\pa^{I_1}\Omega^{J_1} v_T||_{L^2(\widetilde H_s)} ds\\
    \lesssim \int_\rho^{\frac T2} C\varepsilon s^{-2+\alpha}(\ln s)^{k}+C_b\varepsilon^2 s^{-2+\alpha+k\delta}+C_b\varepsilon^2 s^{-2+\alpha+(k-1)\delta}(\ln\rho)^k+C_b^2 \varepsilon^2 s^{-3+2\alpha+(2k+2)\delta}ds\\
    \lesssim (C\varepsilon+C_b^2\varepsilon^2)\rho^{-1+\alpha}.
\end{multline*}

\subsubsection{Exterior estimates}
Recall we have the energy bound
\begin{equation*}
    E_{w;k}(\rho,v_T)^\frac 12\leq C_b\varepsilon \rho^{-\frac 32+\alpha+k\delta},\quad E_{KG;k}(\rho, w_T)^\frac 12 \leq C_b \varepsilon\rho^{-1+\alpha+k\delta}.
\end{equation*}
We then apply Lemma \ref{decaySobolevWave} for the wave component $\pa^I\Omega^J v_T$, and Lemma \ref{decaySobolevKG} for the Klein-Gordon component $\pa^I \Omega^J w_T$. We get in the region $\{t-r\leq r^\frac 12\}$ that
$$|\pa^I \Omega^J v_T|\lesssim C_b\varepsilon\rho^{-\frac 32+\alpha+(k+2)\delta}r^{-\frac 12}\lesssim C_b\varepsilon\rho^{-\frac{13}6+\alpha+(k+2)\delta},\quad \text{on }\Sigma_\rho^e,$$
$$|\pa^I \Omega^J w_T|\lesssim C_b\varepsilon\rho^{-1+\alpha+(k+2)\delta} r^{-1}\lesssim C_b\varepsilon\rho^{-\frac 73+\alpha+(k+2)\delta},\quad \text{on }\Sigma_\rho^e$$
for $|I|+|J|\leq k\leq N-2$.

We want to control
$$I_{w,e;k}:=\int_\rho^{2T^\frac 34}  s^\frac 13\sum_{|I|+|J|\leq k}\left|\left|\pa^I \Omega^J \left(2(\pa_t\phi_0)\pa_t w_T+(\pa_t w_T)^2+2\phi_0 w_T+w_T^2+\Box\psi_{01}\right)\right|\right|_{L^2(\Sigma_s^e)} ds$$
and
$$I_{KG,e;k}:=\int_\rho^{2T^\frac 34} s^\frac 13\sum_{|I|+|J|\leq k}\left|\left|\pa^I\Omega^J\left((u_2+u_3+\psi_{01})\phi_0+u_0 w_T+\phi_0 v_T+v_T w_T+R_0\right)\right|\right|_{L^2(\Sigma_s^e)} ds.$$

Then using the decay of $\pa^I\Omega^J \phi_0$ and $\pa^I\Omega^J w$, we have
\begin{multline*}
    I_{w,e;k}
    \lesssim \int_\rho^{2T^\frac 34} s^\frac 13 \sum_{\substack{|I|+|J|\leq k+1\\|J|\leq k}}(\varepsilon s^{-2}(\ln s)^{k+1}||\pa^I \Omega^J w_T||_{L^2(\Sigma_s^e)}+C_b \varepsilon s^{-\frac 73+\alpha+(k+3)\delta} ||\pa^I \Omega^J w_T||_{L^2(\Sigma_s^e)}\\
    +\varepsilon s^{-\frac{11}3+\frac 23\alpha}) ds\\
    \lesssim \int_\rho^{2T^\frac 34} C_b\varepsilon^2 s^{-\frac 83+\alpha+k\delta}(\ln s)^{k+1}+C_b^2 \varepsilon^2 s^{-3+2\alpha+(2k+3)\delta}+\varepsilon s^{-\frac {10}3+\frac 23\alpha}\lesssim (C_b^2 \varepsilon^2+\varepsilon)\rho^{-2+2\alpha}.
\end{multline*}
In the exterior region $\{t-r<r^\frac 12\}$, where $(t-r)_+$ is bounded by $r^\frac 12$, the decay of $\phi_0$ is very good, so for example we have $|\pa^I L^J \phi_0|\lesssim \varepsilon t^{-2}$ and is zero when $r>t$. Then we have
\begin{multline}
    I_{KG,e;k}
    \lesssim \int_\rho^{2T^\frac 34} s^\frac 13 \sum_{|I|+|J|\leq k}||\pa^I \Omega^J((u_2+u_3+\psi_{01})\phi_0)||_{L^2(\Sigma_s^e)}+||\pa^I\Omega^J R_0||_{L^2(\Sigma_s^e)}\\
    +\sum_{|(I_1,J_1,I_2,J_2)|\leq k} ||r\pa^{I_1}\Omega^{J_1} \phi_0||_{L^\infty(\Sigma_s^e)} ||r^{-1}\pa^{I_2} \Omega^{J_2} v_T||_{L^2(\Sigma_s^e)}\\
    +\sum_{|I|+|J|\leq k}||u_0||_{L^\infty(\Sigma_s^e)}||\pa^I\Omega^J w_T||_{L^2(\Sigma_s^e)}+\sum_{\substack{|(I_1,J_1,I_2,J_2)|\leq k\\ |I_2|+|J_2|\leq k-1}}||\pa^{I_1}\Omega^{J_1} u_0||_{L^\infty(\Sigma_s^e)}||\pa^{I_2}\Omega^{J_2} w_T||_{L^2(\Sigma_s^e)}\\
    +\sum_{\substack{|(I_1,J_1,I_2,J_2)|\leq k\\ |I_1|+|J_1|\leq k/2}}||\pa^{I_1}\Omega^{J_1} v_T||_{L^\infty(\Sigma_s^e)} ||\pa^{I_2}\Omega^{J_2} w_T||_{L^2(\Sigma_s^e)}+||r\pa^{I_1}\Omega^{J_1} w_T||_{L^\infty(\Sigma_s^e)}||r^{-1}\pa^{I_2}\Omega^{J_2} v_T||_{L^\infty(\Sigma_s^e)}\\
    \lesssim \int_\rho^{2T^\frac 34} s^\frac 13 (\varepsilon^2 s^{-3}(\ln s)^k+\varepsilon s^{-\frac{13}3}(\ln s)^k+\varepsilon (s^\frac 43)^{-1}(C_b \varepsilon s^{-\frac 32+\alpha+k\delta})+C_b \varepsilon^2 s^{-\frac 43} s^{-1+\alpha+k\delta}\\
    +C_b \varepsilon^2 s^{-\frac 43}(\ln s)^k s^{-1+\alpha+(k-1)\delta}
    +C_b^2 \varepsilon^2 s^{-\frac {13}6+\alpha+(k+2)\delta} s^{-1+\alpha+k\delta}
    +C_b^2 \varepsilon^2 s^{-1+\alpha+(k+2)\delta} s^{-\frac 32+\alpha+k\delta}) ds\\
    \lesssim \int_\rho^{2T^\frac 34} \varepsilon^2 s^{-\frac 83}(\ln s)^k+C_b \varepsilon^2 s^{-2+\alpha+k\delta}+C_b^2\varepsilon^2 s^{-\frac {17}6+2\alpha+(2k+2)\delta} +C_b^2\varepsilon^2 s^{-\frac{13}6+2\alpha+(2k+2)\delta} ds\\
    \lesssim (\varepsilon+C_b^2\varepsilon^2)(\rho^{-1+\alpha+k\delta}+\rho^{-\frac 76+2\alpha}).
\end{multline}
Recall $0<\alpha<\frac 16$, so the decay is given by $\rho^{-1+\alpha+k\delta}$ since $\delta$ is small. These estimates imply that there exists a constant $C$ such that
\begin{equation}
    E_{w;k}(\rho,v_T)^\frac 12\leq (CC_b^2\varepsilon^2+C\varepsilon) \rho^{-\frac 32+\alpha+k\delta},\quad E_{KG;k}(\rho,w_T)^\frac 12 \leq (CC_b^2 \varepsilon^2+C\varepsilon)\rho^{-1+\alpha+k\delta}, \ \ k\leq N
\end{equation}
for $\widetilde T\leq \rho\leq 2$. Therefore, if we pick some $C_b$ big and $\varepsilon<\frac{C_b-2C}{2CC_b^2}$, we can improve the bootstrap assumption, so $\widetilde T=2$ and the estimates hold for all $2\leq \rho\leq T$. We also note that the choice $C_b$ is independent of $T$, and we now allow the implicit constant in the notation ``$\lesssim$" to be dependent on $C_b$.



\subsection{Taking the limit}We want to show that the limit as $T\rightarrow \infty$ exists. 
Let $T_2>T_1$. We denote $v_i:=v_{T_i}$, $w_i:=w_{T_i}$. Consider the difference $\hat v:=v_2-v_1$ and $\hat w:=w_2-w_1$. We have
\begin{multline}
    -\Box \hat v=\chi(t/T_2)(2(\pa_t\phi_0)\pa_t w_2+(\pa_t w_2)^2+2\phi_0 w_2+w_2^2+\Box\psi_{01})\\
    -\chi(t/T_1)(2(\pa_t\phi_0)\pa_t w_1+(\pa_t w_1)^2+2\phi_0 w_1+w_1^2+\Box\psi_{01}),
\end{multline}
\begin{multline}
    -\Box \hat w+\hat w=\chi(t/T_2)((u_2+u_3+\psi_{01})\phi_0+u_0w_2+\phi_0 v_2+v_2 w_2+R_0)\\
    -\chi(t/T_1)((u_2+u_3+\psi_{01})\phi_0+u_0 w_1+\phi_0 v_1+v_1 w_1+R_0).
\end{multline}
Then we get
\begin{multline}
    -\Box {\hat v}=(\chi(t/T_2)-\chi(t/T_1))(\Box\psi_{01}+2(\pa_t\phi_0)\pa_t w_1+(\pa_t w_1)^2+2\phi_0 w_1+w_1^2)\\
    +\chi(t/T_2)((2\pa_t\phi_0)\pa_t \hat w+2\phi_0 \hat w+((\pa_t w_2)^2-(\pa_t w_1)^2)+(w_2^2-w_1^2)),
\end{multline}
\begin{multline}
    -\Box \hat w+\hat w=(\chi(t/T_2)-\chi(t/T_1))((u_2+u_3+\psi_{01})\phi_0+R_0+u_0 w_1+\phi_0 v_1+v_1 w_1)\\
    +\chi(t/T_2)(u_0 \hat w+\phi_0 \hat v+(v_2 w_2-v_1 w_1)).
\end{multline}

We then consider the energy estimate between $\Sigma_\rho$ and $\Sigma_{T_1}$. Note that $v_1$ and $w_1$ vanish near $\Sigma_{T_1}$. We also have established the bounds
\begin{equation}
    E_{KG;k}(\rho,w_1)^\frac 12+E_{KG;k}(\rho,w_2)^\frac 12\lesssim \varepsilon \rho^{-1+\alpha+k\delta},\quad E_{w;k}(\rho,v_1)^\frac 12+E_{w;k}(\rho,v_2)^\frac 12\lesssim \varepsilon\rho^{-\frac 32+\alpha+k\delta}
\end{equation}
for all $\rho\leq T_1$, as well as the corresponding decay estimates.
We have for $k\leq N-2$ that
\begin{multline*}
    E_{w;k}(\rho,\hat v)^\frac 12\lesssim E_{w;k}(T_1,v_2)^\frac 12+\sum_{|I|+|J|\leq k}\int_{(T_1)^\frac 34/16}^{T_1} ||\pa^I\Omega^J(\Box\psi_{01}+2(\pa_t\phi_0)\pa_t w_1+(\pa_t w_1)^2+2\phi_0 w_1+w_1^2)||_{L^2(\widetilde H_s)}\\
    +s^\frac 13 ||\pa^I\Omega^J(\Box\psi_{01}+2(\pa_t\phi_0)\pa_t w_1+(\pa_t w_1)^2+2\phi_0 w_1+w_1^2)||_{L^2(\Sigma_s^e)} ds\\
    +\int_\rho^{T_1} ||\pa^I\Omega^J((2\pa_t\phi_0)\pa_t \hat w+2\phi_0 \hat w+((\pa_t w_2)^2-(\pa_t w_1)^2)+(w_2^2-w_1^2))||_{L^2(\widetilde H_s)}\\+s^\frac 13 ||\pa^I\Omega^J((2\pa_t\phi_0)\pa_t \hat w+2\phi_0 \hat w+((\pa_t w_2)^2-(\pa_t w_1)^2)+(w_2^2-w_1^2))||_{L^2(\Sigma_s^e)}\\
    \lesssim \varepsilon (T_1)^{-\frac 32+\alpha+k\delta}+\int_{(T_1)^\frac 34/16}^{T_1} \varepsilon s^{-\frac 52+\alpha+k\delta} ds\\
    +\int_\rho^{T_1} ||\phi_0||_{L^\infty(\widetilde H_s)} \sum_{|I|+|J|\leq k}  ||\pa^I \Omega^J \hat w||_{L^2(\widetilde H_s)}+s^\frac 13 ||\phi_0||_{L^\infty(\Sigma_\rho^e)} \sum_{|I|+|J|\leq k}||\pa^I\Omega^J \hat w||_{L^2(\Sigma_\rho^e)}\\
    +\sum_{\substack{|(I_1,J_1,I_2,J_2)|\leq k\\|I_2|+|J_2|\leq k-1}} ||\pa^{I_1}\Omega^{J_1} \phi_0||_{L^\infty(\widetilde H_s)}||\pa^{I_2}\Omega^{J_2} \hat w||_{L^2(\widetilde H_s)}+s^\frac 13||\pa^{I_1}\Omega^{J_1} \phi_0||_{L^\infty(\Sigma_s^e)}||\pa^{I_2}\Omega^{J_2} \hat w||_{L^2(\Sigma_s^e)}\\
    +\sum_{|(I_1,J_1,I_2,J_2)|\leq k}||\pa^{I_1}\Omega^{J_2}(w_1+w_2)||_{L^\infty(\widetilde H_s)}||\pa^{I_2} \Omega^{J_2} \hat w||_{L^2(\widetilde H_s)}\\
    +\sum_{|(I_1,J_1,I_2,J_2)|\leq k}||(t/s)\pa^{I_1}\Omega^{J_2}(\pa_t w_1+\pa_t w_2)||_{L^\infty(\widetilde H_s)}||(s/t)\pa \pa^{I_2} \Omega^{J_2} \hat w||_{L^2(\widetilde H_s)}\\
    +s^\frac 13 \sum_{\substack{|I_1|+|J_1|\leq k+1\\|I_2|+|J_2|\leq k+1}}||\pa^{I_1}\Omega^{J_1}(w_2+w_1)||_{L^\infty(\Sigma_s^e)}||\pa^{I_2}\Omega^{J_2}\hat w||_{L^2(\Sigma_s^e)} ds\\
    \lesssim \varepsilon (T_1)^{-\frac 32+\alpha+k\delta}+\varepsilon (T_1)^{-\frac 98+\frac 34\alpha+\frac 34 k\delta}\\
    +\int_\rho^{T_1} \varepsilon s^{-\frac 32} \sum_{|I|+|J|\leq k}(||\pa^I\Omega^J \hat w||_{L^2(\widetilde H_s)}+||(s/t)\pa\pa^I\Omega^J \hat w||_{L^2(\widetilde H_s)}+\sum_{|I'|\leq 1}||\pa^{I'} \pa^I\Omega^J \hat w||_{L^2(\Sigma_\rho^e)})\\
    +\varepsilon s^{-\frac 32}(\ln s)^k \sum_{|I|+|J|\leq k-1}(||\pa^I\Omega^J \hat w||_{L^2(\widetilde H_s)}+||(s/t)\pa\pa^I\Omega^J \hat w||_{L^2(\widetilde H_s)}+\sum_{|I'|\leq 1}||\pa^{I'} \pa^I\Omega^J \hat w||_{L^2(\Sigma_\rho^e)})\\
    +\varepsilon s^{-\frac 52+\alpha+(k+2)\delta}\sum_{|I|+|J|\leq k}(||\pa^I\Omega^J \hat w||_{L^2(\widetilde H_s)}+||(s/t)\pa\pa^I\Omega^J \hat w||_{L^2(\widetilde H_s)})\\
    +\varepsilon s^\frac 13 s^{-\frac 73+\alpha+(k+2)\delta}
    \sum_{\substack{|I|+|J|\leq k\\|I'|\leq 1}}||\pa^{I'} \pa^I \Omega^J \hat w||_{L^2(\Sigma_s^e)} ds
    \lesssim \varepsilon (T_1)^{-\frac 32+\alpha+k\delta}+\int_\rho^{T_1} \varepsilon s^{-\frac 32} E_{KG}(s,\hat w)^\frac 12\, ds,
\end{multline*}
and
\begin{multline*}
    E_{KG;k}(\rho,\hat w)^\frac 12\lesssim E_{KG;k}(T_1,\hat w)^\frac 12\\
    +\int_{(T_1)^\frac 34/16}^{T_1} \sum_{|I|+|J|\leq k}||\pa^I\Omega^J((u_2+u_3+\psi_{01})\phi_0+R_0+u_0 w_1+\phi_0 v_1+v_1 w_1)||_{L^2(\widetilde H_s)}\\
    +s^\frac 13 ||\pa^I\Omega^J((u_2+u_3+\psi_{01})\phi_0+R_0+u_0 w_1+\phi_0 v_1+v_1 w_1)||_{L^2(\Sigma_s^e)} ds\\
    +\int_\rho^{T_1} \sum_{|I|+|J|\leq k}||\pa^I \Omega^J (u_0 \hat w+\phi_0 \hat v+(v_2 w_2-v_1 w_1))||_{L^2(\widetilde H_s)}+s^\frac 13||\pa^I\Omega^J (u_0 w+\phi_0 v+(v_2 w_2-v_1 w_1))||_{L^2(\Sigma_\rho^e)} ds\\
    \lesssim \varepsilon(T_1)^{-1+\alpha+k\delta}+\int_{(T_1)^\frac 34/16}^{T_1} \varepsilon s^{-2+\alpha+k\delta} ds\\
    +\int_\rho^{T_1} \sum_{|I|+|J|\leq k} ||u_0||_{L^\infty(\widetilde H_s)}||\pa^I \Omega^J \hat w||_{L^2(\widetilde H_s)}+s^\frac 13||u_0||_{L^\infty(\Sigma_s^e)}||\pa^I \Omega^J \hat w||_{L^2(\Sigma_s^e)}\\
    +\sum_{\substack{|(I_1,J_1,I_2,J_2)|\leq k\\|I_2|+|J_2|\leq k-1}} ||\pa^{I_1} \Omega^{J_1} u_0||_{L^\infty(\widetilde H_s)}||\pa^{I_2} \Omega^{J_2} \hat w||_{L^2(\widetilde H_s)}+s^\frac 13||\pa^{I_1}\Omega^{J_1} u_0||_{L^\infty(\Sigma_s^e)}||\pa^{I_2} \Omega^{J_2} \hat w||_{L^2(\Sigma_s^e)}\\
    +\sum_{|(I_1,J_1,I_2,J_2)|\leq k}||r\pa^{I_1}\Omega^{J_1}\phi_0||_{L^\infty(\widetilde H_s)}||r^{-1}\pa^{I_2}\Omega^{J_2}\hat v||_{L^2(\widetilde H_s)}+s^\frac 13 ||r\pa^{I_1}\Omega^{J_1}\phi_0||_{L^\infty(\Sigma_s^e)}||r^{-1}\pa^{I_2} \Omega^{J_2} \hat v||_{L^2(\Sigma_s^e)}\\
    +\sum_{|(I_1,J_1,I_2,J_2)|\leq k}||r\pa^{I_1} \Omega^{J_1} w_1||_{L^\infty(\widetilde H_s)}||r^{-1}\pa^{I_2} \Omega^{J_2} \hat v||_{L^2(\widetilde H_s)}+||\pa^{I_1}\Omega^{J_1} v_1||_{L^\infty(\widetilde H_s)}||\pa^{I_2}\Omega^{J_2} \hat w||_{L^2(\widetilde H_s)}\\
    +s^\frac 13\sum_{|(I_1,J_1,I_2,J_2)|\leq k}(||r\pa^{I_1} \Omega^{J_1} w_1||_{L^\infty(\Sigma_s^e)}||r^{-1}\pa^{I_2} \Omega^{J_2} \hat v||_{L^2(\Sigma_s^e)}+||\pa^{I_1}\Omega^{J_1} v_1||_{L^\infty(\Sigma_s^e)}||\pa^{I_2}\Omega^{J_2} \hat w||_{L^2(\Sigma_s^e)}) ds\\
    \lesssim \varepsilon (T_1)^{-1+\alpha+k\delta}+\varepsilon (T_1)^{-\frac 34+\frac 34\alpha+\frac 34 k\delta}+\int_\rho^{T_1} s^{-1}E_{KG}(s,\hat w)^\frac 12+s^{-\frac 12}E_w(s,\hat v)^\frac 12 ds.
\end{multline*}
Note that the estimate of the integrals from $(T_1)^\frac 34/16$ (which is a number less than the minimal value of $\rho$ where $\Sigma_\rho$ intersects with the support of $\chi(t/T_2)-\chi(t/T_1)$) to $T_1$ are the same as the estimate we did before. Then with $E_k(\rho)^\frac 12=\rho^{\frac 12} E_{w;k}(\rho,\hat v)^\frac 12+E_{KG;k}(\rho,\hat v)^\frac 12$, we have for $\rho\leq T_1$ that
\begin{equation*}
    E_k(\rho)^\frac 12 \lesssim \varepsilon (T_1)^{-\frac 58+\frac 34\alpha+\frac 34 k\delta} +\int_\rho^{T_1} \varepsilon s^{-1} E_k(s)^\frac 12+\varepsilon s^{-1}(\ln s)^k E_{k-1}(s)^\frac 12 ds.
\end{equation*}
When $k=0$, the $E_{k-1}$ term does not appear. In this case,
\begin{equation*}
    E_0(\rho)^\frac 12 \lesssim \varepsilon (T_1/\rho)^{C\varepsilon} (T_1)^{-\frac 58+\frac 34\alpha}\rightarrow 0\quad\text{as }T_1\rightarrow \infty.
\end{equation*}
Then it is straightforward to show by induction that
\begin{equation*}
    E_k(\rho)^\frac 12\lesssim \varepsilon (T_1/\rho)^{C\varepsilon} (T_1)^{-\frac 58+\frac 34\alpha+\frac 34 k\delta},\quad k\leq N-2
\end{equation*}
Then by Sobolev embeddings, we have for $|I|+|J|\leq N-4$ that
\begin{equation*}
    \sup_{\rho(t,x)\leq T_1}|\pa^I \Omega^J \hat w|+|\pa^I \Omega^J \hat v|\lesssim \varepsilon (T_1)^{-\frac 58+\frac 34\alpha+\frac 34(k+2)\delta+C\varepsilon}
\end{equation*}
which tends to zero as $T_2>T_1\rightarrow \infty$. This shows the existence of the limit $w=\lim_{T\rightarrow\infty} w_T$ and $v=\lim_{T\rightarrow\infty} v_T$, and that $(v+u_0,w+\phi_0)$ gives a solution of the system.

\subsection*{Acknowledgments} H.L. was supported in part by Simons Collaboration Grant 638955. X.C. thanks Junfu Yao for helpful discussions.


\bibliographystyle{abbrv}
\bibliography{reference}

\end{document}